\documentclass[production,11pt]{jsg}\usepackage{latexsym}\usepackage{amsmath}\usepackage{xspace}\usepackage{enumerate}\usepackage{amsfonts}\usepackage{amscd}\usepackage{syntonly}\usepackage{amssymb}

\newtheorem{thm}{Theorem}
\newtheorem{prop}[thm]{Proposition}
\newtheorem{lemma}[thm]{Lemma}
\newtheorem{claim}{Claim}
\newtheorem{defi}[thm]{Definition}
\newtheorem{rmk}[thm]{Remark}

\newenvironment{pf}[1][Proof.]{\noindent \textbf{#1:} }{}
\newenvironment{enui}{\begin{enumerate}[(i)]}{\end{enumerate}}

\def\Aut{{\operatorname{Aut}}}

\def\X{\mathcal{X}}
\def\E{\mathcal{E}}

\def\id{{\operatorname{id}}}

\def\F{{\mathcal{F}}}
\def\hol{{\operatorname{hol}}}

\def\im{{\operatorname{im}}}
\def\empty{\emptyset}
\def\nn{{\nonumber}}
\newcommand\wt[1]{{\widetilde{#1}}}
\newcommand{\BAR}[1]{{\overline{#1}}}
\def\al{{\alpha}}

\def\De{\Delta}

\def\eps{\varepsilon}
\def\Om{\Omega}
\def\om{\omega}

\def\lam{\lambda}
\def\Si{\Sigma}
\def\si{\sigma}

\renewcommand\phi{\varphi}

\def\S{{\mathcal{S}}}
\newcommand{\N}{\mathbb{N}}
\newcommand{\Z}{\mathbb{Z}}

\newcommand{\R}{\mathbb{R}}
\def\C{\mathbb C}

\def\Coi{\mathcal{C}^{\operatorname{flat}}}
\def\Con{\mathcal{C}}
\def\TT{\mathbb{T}}
\def\T{\mathcal{T}}
\def\D{\mathbb{D}}

\def\Gr{\operatorname{G}}

\def\Sp{{\operatorname{Sp}}}
\def\OO{{\operatorname{O}}}
\def\U{{\operatorname{U}}}
\def\g{\mathfrak g}
\def\A{\mathcal A}

\def\pr{{\operatorname{pr}}}

\def\sub{\subseteq}

\def\x{\times}
\def\wo{\setminus}
\def\one{\mathbf{1}}
\def\iso{\cong}
\def\Iso{{\operatorname{Iso}}}
\def\Hom{{\operatorname{Hom}}}
\def\End{{\operatorname{End}}}
\def\d{\partial}

\def\lan{\langle}
\def\ran{\rangle}

\def\rank{{\operatorname{rank}}}
\def\GL{{\operatorname{GL}}}
\def\O{{\operatorname{O}}}
\def\Fix{{\operatorname{Fix}}}
\def\Ham{\operatorname{Ham}}
\def\Colon{:}
\def\CP{\mathbb{C}P}
\def\FS{{\operatorname{FS}}}

\def\GS{{\operatorname{GS}}}
\def\th{\theta}
\def\tt{\mathfrak{t}}
\def\spec{{\operatorname{spec}}}
\def\Ga{\Gamma}
\def\exp{\operatorname{exp}}

\def\disj{\coprod}

\def\cont{\supseteq}
\def\corank{{\operatorname{corank}}}
\def\gcd{\operatorname{gcd}}
\def\const{\equiv}
\def\then{\Rightarrow}

\def\codim{\operatorname{codim}}

\def\real{{\operatorname{real}}}

\title[A Maslov Map for Coisotropic Submanifolds]{A Maslov Map for Coisotropic Submanifolds, Leaf-wise Fixed Points and Presymplectic Non-Embeddings}

\author{Fabian Ziltener (University of Toronto)}

\begin{document}

\begin{abstract} Let $(M,\om)$ be a symplectic manifold, $N\sub M$ a coisotropic submanifold, and $\Si$ a compact oriented (real) surface. I define a natural Maslov index for each continuous map $u:\Si\to M$ that sends every connected component of $\d\Si$ to some isotropic leaf of $N$. This index is real valued and generalizes the usual Lagrangian Maslov index. The idea is to use the linear holonomy of the isotropic foliation of $N$ to compensate for the loss of boundary data in the case $\codim N<\dim M/2$. The definition is based on the Salamon-Zehnder (mean) Maslov index of a path of linear symplectic automorphisms. I prove a lower bound on the number of leafwise fixed points of a Hamiltonian diffeomorphism, if $(M,\om)$ is geometrically bounded and $N$ is closed, regular (i.e. "fibering"), and monotone. As an application, we obtain a presymplectic non-embedding result. I also prove a coisotropic version of the Audin conjecture.
\end{abstract}

\maketitle
\tableofcontents

\section{Motivation and main results}\label{sec:main}
This article is concerned with the following two problems. Let $(M,\om)$ be a symplectic manifold and $N\sub M$ a coisotropic submanifold. A leafwise fixed point of a map $\phi:M\to M$ is by definition a point $x\in N$ such that $\phi(x)$ lies in the isotropic leaf through $x$. We denote by $\Fix(\phi,N):=\Fix(\phi,N,\om)$ the set of such points.

{\bf Problem A:} Find conditions on $(M,\om,N,\phi)$ under which $\Fix(\phi,N)$ is non-empty and give a lower bound on $\big|\Fix(\phi,N)\big|$.

Note that in the case $N=M$ the set $\Fix(\phi,N)$ equals the set $\Fix(\phi)$ of usual fixed points. In the other extreme case, in which $N$ is Lagrangian, we have $\Fix(\phi,N)=N\cap\phi^{-1}(N)$. 

To formulate the second problem, let $V$ be a real vector space and $\om$ a skew-symmetric form on $V$. We denote $\corank\om:=\dim\ker\big(V\ni v\mapsto \om(v,\cdot)\in V^*\big)$. A presymplectic form on a manifold $M$ is a closed two-form $\om$ of constant corank. We say that a presymplectic manifold $(M',\om')$ embeds into another presymplectic manifold $(M,\om)$ iff there exists an embedding $\psi:M'\to M$ such that $\psi^*\om=\om'$. The next problem generalizes the symplectic and Lagrangian non-embedding problems:

{\bf Problem B:} Find conditions on $(M,\om)$ and $(M',\om')$ under which $(M',\om')$ does not embed into $(M,\om)$. 

In \cite{ZiLeafwise}, I gave some solution to problem A, imposing the conditions that $N$ is regular and the Hofer distance of $\phi$ and the identity is small enough. In the present article, the second condition is replaced by the assumption that $N$ is monotone. The paper \cite{ZiLeafwise} also contains some solution to problem B, assuming that $\om$ is non-degenerate and aspherical. In the present article the latter condition is replaced by monotonicity of $\om$. 

To define monotonicity for a coisotropic submanifold $N$, I introduce a natural Maslov map for $N$, which equals the usual Maslov index in the case $\dim N=\dim M/2$, and twice the first Chern class of $(M,\om)$ in the case $N=M$. 

\subsection{Definition of the Maslov map}
Let $(M,\om)$ be a symplectic manifold (without boundary), $N\sub M$ a coisotropic submanifold, and $X$ a topological manifold. We denote by $\Con(X)$ the set of connected components of $X$ and by $N_\om$ the set of isotropic leaves of $N$. We define 
\[C(X,M;N,\om):=\big\{u\in C(X,M)\,\big|\,\forall Y\in \Con(\d X)\,\exists F\in N_\om:\,u(Y)\sub F\big\}.\] 
Let $u\in C([0,1]\x X,M)$. We call $u$ an \emph{$(N,\om)$-admissible homotopy} iff for every $Y\in\Con(\d X)$ there exists $F\in N_\om$ such that $u(t,x)\in F$, for every $t\in[0,1]$, $x\in Y$. We denote by $\big[X,M;N,\om\big]$ the corresponding set of all $(N,\om)$-admissible homotopy classes of maps from $X$ to $M$. 

We denote by $\S$ the class of all compact oriented (real) topological surfaces (possibly with boundary and disconnected). Let $\Si\in\S$. The Maslov map introduced in this article is a map 
\begin{equation}\label{eq:m Si om N []}m_{\Si,N}:=m_{\Si,\om,N}:\big[\Si,M;N,\om\big]\to\R.\end{equation}
Its definition involves the following four steps. A more direct, but less natural definition is given on page \pageref{elementary}.

\subsubsection*{The Salamon-Zehnder Maslov index}
Let $(V,\om)$ be a symplectic vector space. We denote by $\Aut\om$ the group of linear symplectic automorphisms of $V$. We define the \emph{Salamon-Zehnder Maslov index}
\begin{equation}\label{eq:m om C}m_\om:C([0,1],\Aut\om)\to\R\end{equation} 
as follows. We define the winding map $\al:C\big([0,1],\R/\Z\big)\to\R$ by $\al(z):=\wt z(1)-\wt z(0)$, where $\wt z\in C([0,1],\R)$ is any path such that $\wt z(t)+\Z=z(t)$, for every $t\in\R$. We denote by $\rho_\om:\Aut(\om)\to \R/\Z\iso S^1$ the Salamon-Zehnder map (see Proposition \ref{prop:rho} below). Let $\Phi\in C([0,1],\Aut\om)$. We define $m_\om(\Phi):=2\al(\rho_\om\circ\Phi)$. 

\subsubsection*{The Maslov map for pairs of flat transports}
Let $X$ be a topological manifold. We denote by $\Pi X$ the fundamental groupoid of $X$. This is a topological groupoid. Its set of objects is $X$ and its set of morphisms consists of all homotopy classes (with fixed end-points) of continuous paths in $X$. 

For two vector spaces $V$ and $V'$ we denote by $\Iso(V,V')$ the set of all isomorphisms from $V$ to $V'$. Let $E\to X$ be a vector bundle. We denote by $\GL(E)$ the general linear groupoid of $E$. This is a topological groupoid. Its set of objects is $X$ and its set of morphisms consists of all triples $(x,y,\Phi)$, where $x,y\in X$ and $\Phi\in\Iso(E_x,E_y)$. 

By a flat (linear) transport we mean a (continuous) representation $\Phi$ of $\Pi X$ on $E$, i.e.~a morphism of topological groupoids from $\Pi X$ to $\GL(E)$ that covers the identity on $X$. Such a $\Phi$ associates to every homotopy class of paths $x\in C([0,1],X)$ an isomorphism $\Phi([x])\in \Iso(x(0),x(1))$. It is equivariant with respect to concatenation of paths. We denote by $\T(E)$ the set of all flat transports on $E$. 

We call $\Phi\in\T(E)$ \emph{regular} iff $\Phi([x])=\id$, for every $x\in C([0,1],X)$ satisfying $x(0)=x(1)$. Note that if $X$ is a smooth manifold and $E$ is a smooth vector bundle then the parallel transport of a smooth flat connection on $E$ is a flat transport. 

For symplectic vector spaces $(V,\om)$ and $(V',\om')$ we denote by $\Iso(\om,\om')$ the set of linear isomorphisms $\Phi:V\to V'$ such that $\Phi^*\om'=\om$. Let $X$ be a topological manifold and $(E,\om)$ be a symplectic vector bundle over $X$. We define $\GL(E,\om)$ to be the subgroupoid of $\GL(E)$ consisting of all $(x,y,\Phi)$ such that $\Phi\in \Iso(\om_x,\om_y)$. We call a transport $\Phi\in\T(E)$ symplectic iff $\Phi(\Pi X)\sub\GL(E,\om)$, and denote by $\T(E,\om)$ the set of all such $\Phi$'s. 

Let $X$ be an oriented closed curve (i.e.~topological real one-manifold), $(E,\om)$ a symplectic vector bundle over $X$, and $\Phi,\Phi'\in\T(E,\om)$ be such that $\Phi$ or $\Phi'$ is regular. We define the number $m_\om(\Phi,\Phi')\in\R$ as follows. Namely, we choose a path $z\in C([0,1],X)$ such that $z(0)=z(1)$ and the map $S^1\iso[0,1]/\{0,1\}\ni [t]\mapsto z(t)\in X$ has degree one. We define $\Psi\in C([0,1],\Aut(\om_{z(0)}))$ by $\Psi(t):=\Phi'([z|_{[0,t]}])^{-1}\Phi([z|_{[0,t]}])$, and 
\begin{equation}\label{eq:m om Phi Phi'}m_\om(\Phi,\Phi'):=m_{\om_{z(0)}}(\Psi).\end{equation}
By Lemma \ref{le:m Phi Phi'} below this number is well-defined.

\subsubsection*{The coisotropic Maslov map for bundles}
Let $(V,\om)$ be a symplectic vector space and $W\sub V$ be a subspace. We denote by $W^\om:=\big\{v\in V\,\big|\,\om(v,w)=0,\,\forall w\in W\big\}$ its symplectic complement. Assume that $W$ is coisotropic. We denote by $(W_\om:=W/W^\om,\om_W)$ its linear symplectic quotient, and for $\Phi\in\Aut(\om)$ we define 
\[\Phi_W:W_\om\to (\Phi W)_\om,\quad \Phi_W(v+W^\om):=\Phi v+(\Phi W)^\om.\]

Let $E$ be a vector bundle over $X$, $W\sub E$ a subbundle and $\Phi\in\T(E)$. We say that $\Phi$ leaves $W$ invariant iff $\Phi([z])W_{z(0)}=W_{z(1)}$, for every $[z]\in\Pi X$. Let $(E,\om)$ be a symplectic vector bundle over $X$. We define $\Coi(E,\om)$ to be the set of all pairs $(W,\Phi)$, where $W\sub E$ is an $\om$-coisotropic subbundle, and $\Phi\in\T(W_\om,\om_W)$. Let $W\sub E$ be a coisotropic subbundle, and $\Phi\in\T(E,\om)$ be a transport that leaves $W$ invariant. We define $\Phi_W\in\T(W_\om,\om_W)$ by $\Phi_W([z]):=\Phi([z])_{W_{z(0)}}$. 
\begin{thm}(Coisotropic Maslov map for bundles)\label{thm:m Si E om} Let $\Si\in\S$ be connected and such that $\d\Si\neq\emptyset$, and let $(E,\om)$ be a symplectic vector bundle over $\Si$. Then there exists a unique map $m_{\Si,E,\om}:\Coi(E,\om)\to\R$ with the following properties.
\begin{enui}
\item\label{thm:m Si E om:boundary}(Boundary) For every regular transport $\Phi_0\in\T(E,\om)$ and every $\Phi\in\T((E,\om)|_{\d\Si})$ we have $m_{\Si,E,\om}(E|_{\d\Si},\Phi)=m_{\d\Si,\om|_{\d\Si}}(\Phi,\Phi_0)$. 
\item\label{thm:m Si E om:inv}(Invariant subbundle) Let $W\sub E|_{\d\Si}$ be an $\om$-coisotropic subbundle, and $\Psi\in\T((E,\om)|_{\d\Si})$. If $\Psi$ leaves $W$ invariant then $m_{\Si,E,\om}(E|_{\d\Si},\Psi)=m_{\Si,E,\om}(W,\Psi_W)$. 
\end{enui}
\end{thm} 
For the proof of this theorem, the idea is to define 
\[m_{\Si,E,\om}(W,\Phi):=m_{\d\Si,\om|_{\d\Si}}(\Psi,\Phi_0),\]
where $\Psi\in\T((E,\om)|_{\d\Si})$ is a lift of $(W,\Phi)$, and $\Phi_0\in\T(E,\om)$ is a regular transport. In order to show that this does not depend on the choice of $\Psi$, the following result is crucial. Namely, let $(V,\om)$ be a symplectic vector space, $W\sub V$ a coisotropic subspace, and $\Psi\in\Aut(\om)$ be such that $\Psi W=W$. Then $\rho_\om(\Psi)=\pm\rho_{\om_W}(\Psi_W)$. (See Proposition \ref{prop:rho om Psi} below.) The proof of this identity is based on the existence of a path $\Psi^.\in C([0,1],\Aut(\om))$, such that $\Psi^1=\Psi$, $\Psi^0$ leaves three fixed subspaces of $V$ invariant, and the map $[0,1]\ni t\mapsto\rho_\om(\Psi^t)\in\R$ is constant. 

Let $\Si\in\S$ be a connected surface satisfying $\d\Si\neq\emptyset$. We define $\E_\Si$ to be the class of all quadruples $\big(E,\om,W,\Phi\big)$, where $(E,\om)$ is a symplectic vector bundle over $\Si$ and $(W,\Phi)\in \Coi((E,\om)|_{\d\Si})$. We define 
\[m_\Si:\E_\Si\to\R,\quad m_\Si(E,\om,W,\Phi):=m_{\Si,E,\om}(W,\Phi),\]
where $m_{\Si,E,\om}$ is the unique map satisfying the conditions of Theorem \ref{thm:m Si E om}.  
\subsubsection*{Definition of $m_{\Si,\om,N}$}
We now define the map (\ref{eq:m Si om N []}) as follows. Assume first that $\Si$ is connected. If $\d\Si=\emptyset$ then we define $m_{\Si,\om,N}(a):=2\lan c_1(M,\om),a\ran$. Assume now that $\d\Si\neq\empty$. We denote by $\hol^{N,\om}$ the linear holonomy of the isotropic foliation of $N$ (see (\ref{eq:hol F}) below). We define the map $\wt m_{\Si,\om,N}:C(\Si,M;N,\om)\to\R$ by 
\begin{equation}\label{eq:m Si om N}\wt m_{\Si,\om,N}(u):=m_\Si\big(u^*(TM,\om),u|_{\d\Si}^*(TN,\hol^{N,\om})\big).
\end{equation} 
It follows from Theorem \ref{thm:m}(\ref{thm:m:homotopy}) below that this map is invariant under $(N,\om)$-admissible homotopies. For a general $\Si\in\S$ we define $\wt m_{\Si,\om,N}$ by $\wt m_{\Si,\om,N}(u):=\sum_{\Si'\in\Con(\Si)}\wt m_{\Si',\om,N}(u|_{\Si'})$.
\begin{defi}\label{defi:m Si om N S} Let $(M,\om)$ be a symplectic manifold, $N\sub M$ a coisotropic submanifold, and $\Si\in\S$. We define the Maslov map $m_{\Si,N}:\big[\Si,M;N,\om\big]\to\R$ to be the map induced by $\wt m_{\Si,\om,N}$.
\end{defi}
As an example, let $\Si:=\D\sub\R^2$ be the unit disk, $M:=\R^{2n}$, $\om$ the standard structure $\om_0$, $N:=S^{2n-1}$, and $u:\D\to\R^{2n}$ the inclusion $u(z):=(z,0,\ldots,0)$. Then $m_{\D,\om_0,S^{2n-1}}(u)=2$. For more examples see the subsection on page \pageref{GS} about the Gaio-Salamon Maslov index.

The map $m_{\Si,\om,N}$ may be viewed as a \emph{mean} Maslov index. Analogously to the definition of the Conley-Zehnder index there should also be a natural \emph{integer valued} map with the same domain. 

\subsubsection*{The regular case}
Let $X$ be a compact topological manifold. We call a map $u\in C([0,1]\x X,M)$ a \emph{weakly $(N,\om)$-admissible homotopy} iff for every $Y\in\Con(\d X)$ and $t\in[0,1]$ there exists $F\in N_\om$ such that $u(t,x)\in F$, for every $x\in Y$. We denote by $\big\lan X,M;N,\om\big\ran$ the corresponding set of homotopy classes. We call $N$ regular iff its isotropic leaf relation is a closed subset and a submanifold of $N\x N$. Assume now that $N$ is regular. Then it follows from Theorem \ref{thm:m M om N}(\ref{thm:m M om N:weak homotopy},\ref{thm:m M om N:reg}) below that the Maslov map $m_{\Si,N}$ takes on integer values and is invariant under weak homotopies. If $N$ is also orientable then by Theorem \ref{thm:m M om N}(\ref{thm:m M om N:reg}) $m_{\Si,N}$ takes on even values. 

\subsection{More elementary description} \label{elementary}
In more elementary, but less natural terms, the map $\wt m_{\Si,\om,N}$ is given as follows. Let $(V,\om)$ be a symplectic vector space, and $W_0\sub V$ a coisotropic subspace. We define the \emph{framed coisotropic Grassmannian} $\Gr(\om,W_0)$ to be the manifold consisting of all pairs $(W,\Phi)$, where $W\sub V$ is a coisotropic subspace and $\Phi\in\Iso\big((W_0)_\om,\om_{W_0};W_\om,\om_W\big)$. 

Let $(W,\Phi)\in C([0,1],\Gr(\om,W_0))$ be a path such that $W(0)=W(1)$. We choose a path $\Psi\in C([0,1],\Aut\om)$ satisfying $\Psi(t)W_0=W(t)$ and $\Psi(t)_{W_0}=\Phi(t)$. (It follows from Lemma \ref{le:bundle} below that such a path exists.) We define $m_\om(W,\Phi):=m_\om(\Psi)$. (It follows from Theorem \ref{thm:m X V om}(\ref{thm:m X V om:Psi Psi'}) below that this number does not depend on the choice of $\Psi$.) 

Let now $M,\om,N,\Si$ and $u$ be as above. For simplicity, assume that $\Si=\D$. We denote $2n:=\dim M$. We choose a symplectic trivialization $\Psi:\D\x\R^{2n}\to u^*TM$, and define $W_0:=\Psi_{u(1)}^{-1}T_{u(1)}N\sub\R^{2n}$. We define the path $(W,\Phi):[0,1]\to\Gr(W_0,\om_0)$ as follows. Let $s\in[0,1]$. We set $W(s):=\Psi_{u(e^{2\pi is})}^{-1}T_{u(e^{2\pi is})}N\sub\R^{2n}$. Furthermore, we define $F_s:(T_{u(1)}N)_\om\to(T_{u(e^{2\pi is})}N)_\om$ to be the linear holonomy of the isotropic foliation of $N$ along the path $[0,1]\ni t\mapsto u(e^{2\pi ist})\in F$. We set $\Phi(s):=(\Psi_{u(e^{2\pi is})})_{W_0}^{-1}F_s(\Psi_{u(1)})_{W_0}$. The Maslov index of $u$ is now given by 
\[\wt m_{\D,\om,N}(u)=m_{\om_0}(W,\Phi).\]

\subsection{Leaf-wise fixed points, presymplectic embeddings and minimal Maslov numbers}
\subsubsection*{Leaf-wise fixed points}
Assume that $N$ is regular. We define the minimal Maslov number 
\[m(N):=m(N,\om):=\inf\big(\big\{m_{\D,N}(a)\,\big|\,a\in \big\lan X,M;N,\om\big\ran\big\}\cap\N\big)\in\N\cup\{\infty\}.\]
We call $N$ monotone iff there exists a constant $c>0$ such that for every $u\in C(\D,M;N,\om)$ we have $\wt m_{\D,\om,N,F}(u)=c\int_\D u^*\om$. 

We denote by $\Ham(M,\om)$ the group of Hamiltonian diffeomorphisms on $M$. For every $\phi\in\Ham(M,\om)$ the pair $(N,\phi)$ is called \emph{non-degenerate} iff the following holds. For $x_0\in N$ we denote by $\pr_{x_0}:T_{x_0}N\to (T_{x_0}N)_\om=T_{x_0}N/(T_{x_0}N)^\om$ the canonical projection. Let $F\sub N$ be an isotropic leaf, and $x\in C^\infty([0,1],F)$ a path. Assume that $\phi(x(0))=x(1)$, and let $v\in T_{x(0)}N\cap T_{x(0)}\phi^{-1}(N)$ be a vector. Then $v\neq0$ implies that 
\begin{equation}\label{eq:hol x pr}\hol^{\om,N}_x\pr_{x(0)}v\neq \pr_{x(1)}d\phi(x(0))v.\end{equation}
In the case $N=M$ this condition means that for every $x_0\in\Fix(\phi)$, $1$ is not an eigenvalue of $d\phi(x_0)$. Furthermore, in the case that $N$ is Lagrangian the condition means that for every connected component $N'\sub N$ we have $N'\pitchfork\phi(N')$, i.e. $N'$ and $\phi(N')$ intersect transversely. 

For a topological space $X$ and $i\in\N\cup\{0\}$ we denote by $b_i(X,\Z_2)$ the $i$-th $\Z_2$-Betti number of $X$. 
\begin{thm}\label{thm:Fix} Let $(M,\om)$ be a (geometrically) bounded symplectic manifold, $N\sub M$ a closed monotone regular coisotropic submanifold and $\phi\in\Ham(M,\om)$. If $(N,\phi)$ is non-degenerate then 
\begin{equation}\label{eq:Fix phi N}|\Fix(\phi,N)|\geq\sum_{i=\dim N-m(N)+2,\ldots,m(N)-2}b_i(N,\Z_2).\end{equation}
\end{thm}
This theorem generalizes a result for the case $\dim N=\dim M/2$, which is due to P. Albers \cite{Al}. 

\subsubsection*{Examples}
A big class of examples is given as follows. Let $(X,\si)$ and $(X',\si')$ be closed symplectic manifolds and $L\sub X$ a closed Lagrangian submanifold. We define $(M,\om,N):=(X'\x X,\si'\oplus\si,X'\x L)$. Then $N$ is a closed regular coisotropic submanifold of $M$. 

Let $\Si\in\S$. We define $\Si'$ to be the closed surface obtained from $\Si$ by collapsing each boundary circle to a point. By straight-forward arguments the map $\Phi:[\Si',X']\x\lan\Si,X;L,\si\ran\to\big\lan\Si,M;N,\om\big\ran$, $\Phi([u'],[u]):=[(u',u)]$, is well-defined and a bijection. Furthermore, $m_{\Si,\om,N}\circ\Phi([u'],[u])=2\lan c_1(TX',\si'),[u']\ran+m_{\Si,\si,L}([u])$. This follows from Theorem \ref{thm:m M om N}(\ref{thm:m M om N:prod},\ref{thm:m M om N:M}) below. It follows that $m(N,\om)$ is the greatest common divisor of twice the minimal Chern number of $(TX',\si')$ and $m(L,\si)$. 

Assume that there exists $c>0$ such that $2\lan c_1(TX',\si'),[u]\ran=c\int_{S^2}u^*\si'$ for every $u\in C^\infty(S^2,X')$, and $m_{\D,L,\si}([u])=c\int_\D u^*\si$, for every $u\in C^\infty(\D,X)$ satisfying $u(S^1)\sub L$. Then $N$ is monotone, and hence $(M,\om,N)$ satisfies the conditions of Theorem \ref{thm:Fix}. 

As a concrete example, let $n\in\N$, $k\in\{1,\ldots,2n\}$, $X':=\CP^n$, $\si'$ be the Fubini-Studi form $\om_\FS$, $X$ the torus $\TT^{2k}$ with the standard form $\si:=\om_0$, and $L=\TT^k\sub \TT^{2k}$ the standard Lagrangian subtorus. Let $\phi\in\Ham(M,\om)$ be such that $(N,\phi)$ is non-degenerate. Then applying Theorem \ref{thm:Fix} we obtain $|\Fix(\phi,N)|\geq\sum_{i=k,\ldots,2n,\,j=0,\ldots,2n}$''$k$ choose $i-2j$''. 

\subsubsection*{Idea of proof of Theorem \ref{thm:Fix}}
The idea is to find a Lagrangian embedding of $N$ into a suitable symplectic manifold, and then apply the Main Theorem in \cite{Al}. Since $N$ is regular, the set of isotropic leaves $N_\om$ carries canonical smooth and symplectic structures $\A_{N,\om}$ and $\om_N$. We define 
\begin{eqnarray}\label{eq:wt M wt om}&\wt M:=M\x N_\om,\quad \wt\om:=\om\oplus(-\om_N),&\\
\label{eq:iota N wt N} &\iota_N\Colon N\to \wt M,\quad \iota_N(x):=(x,N_x),\quad \wt N:=\iota_N(N).&
\end{eqnarray} 
Then $\iota_N$ is an embedding of $N$ into $\wt M$ that is Lagrangian with respect to the symplectic form $\wt\om$ on $\wt M$. In order for the hypotheses of Albers' result to be satisfied, the inequality $m(\wt N,\wt\om)\geq m(N,\om)$ is crucial. It follows from Theorem \ref{thm:m M om N}(\ref{thm:m M om N:wt M}) and Propositions \ref{prop:phi} below.
\subsubsection*{Application: presymplectic non-embeddings}
Let $(M,\om)$ be a symplectic manifold. We denote by $c_1^{M,\om}:[S^2,M]\to\R$ the contraction with the first Chern class of $(M,\om)$, and by $c_1(M,\om):=\inf\big(c_1^{M,\om}([S^2,M])\cap\N\big)\in\Z§$ the (spherical) minimal Chern number. Let $(M',\om')$ be a regular presymplectic manifold. This means that the isotropic leaf relation of $\om'$ is a closed subset and a submanifold of $M'\x M'$. For $x,y\in\N\cup\{\infty\}$ we denote by $\gcd(x,y)\in\N\cup\{\infty\}$ the greatest common divisor of $x$ and $y$. (Our convention is that $\gcd(x,\infty)=\gcd(\infty,x)=x$, for $x\in\N$, and $\gcd(\infty,\infty)=\infty$.) We define $\mu:=2\gcd\big(c_1(M,\om),c_1(M'_{\om'},\om'_{M'})\big)$. The proof of the following result is based on Theorem \ref{thm:Fix}. 
\begin{thm}\label{thm:non-emb} Assume that $(M,\om)$ is connected and bounded, every compact subset of $M$ is Hamiltonianly displaceable, $M'$ is connected and closed, there exists an index $i\in\big\{\dim M'-\mu+2,\ldots,\mu-2\big\}$ such that $b_i(M',\Z_2)\neq0$, for some fiber $F\sub M'$ every loop $u\in C(S^1,F)$ is contractible in $M'$, $\dim M'+\corank\om'=\dim M$, and the following condition is satisfied.
\begin{enui}
\item\label{thm:non-emb:pi 1 F pi 1 M'} There exists a constant $c>0$ such that $c_1^{M,\om}=c[\om]$ on $[S^2,M]$ and $c_1^{M'_{\om'},\om'_{M'}}=c[\om'_{M'}]$ on $[S^2,M'_{\om'}]$.
\end{enui}
Then $(M',\om')$ does not embed into $(M,\om)$.
\end{thm}
Note that the condition $\dim M'+\corank\om'=\dim M$ is critical in the sense that in the case $\dim M'+\corank\om'>\dim M$ there is no presymplectic embedding of any open non-empty subset of $M'$ into $M$, whereas in the case $\dim M'+\corank\om'\leq\dim M$ for every point $x'\in M'$ there exists an open neighbourhood that embeds presymplectically into $M$. 

The next result gives a criterion under which condition (\ref{thm:non-emb:pi 1 F pi 1 M'}) in Theorem \ref{thm:non-emb} holds and $\mu$ becomes simpler.
\begin{prop}\label{prop:F simply} Let $(M,\om)$ be a connected symplectic manifold and $(M',\om')$ a regular presymplectic manifold, such that some isotropic fiber $F\sub M'$ is simply-connected, $\dim M'+\corank\om'=\dim M$, and $(M',\om')$ embeds into $(M,\om)$. Then $\mu=2c_1(M,\om)$. Furthermore, if $(M,\om)$ is spherically monotone then condition (\ref{thm:non-emb:pi 1 F pi 1 M'}) of Theorem \ref{thm:non-emb} holds. 
\end{prop}
It follows from Theorem \ref{thm:non-emb} and Proposition \ref{prop:F simply} that $(M',\om')$ does not embed into $(M,\om)$, provided that $\dim M'+\corank\om'=\dim M$ and some conditions on $(M,\om)$ and some conditions on $(M',\om')$ are satisfied. (The point here is that there are no further assumptions involving both $(M,\om)$ and $(M',\om')$.) 

As an example, let $m$ and $n$ be positive integers, $(X,\si)$ a closed symplectic manifold and $\pi:M'\to X$ a closed smooth fiber bundle with simply connected fibers, such that $\dim X/2+k=m+n$ and there exists $i\in\{2n-k,\ldots,2m\}$ such that $b_i(M',\Z_2)\neq0$, where $k$ denotes the dimension of the fibers. We define $\om':=\pi^*\si$ and denote by $\om_{\FS}$ the Fubini-Studi form on $\CP^m$ and by $\om_0$ the standard symplectic form on $\R^{2n}$. It follows from Theorem \ref{thm:non-emb} that $(M',\om')$ does not embed into $\big(\CP^m\x\R^{2n},\om_\FS\oplus\om_0\big)$. 

More concretely, let $m$ be a positive integer and $k\in\{2,\ldots,2m\}$. Then $\big(\CP^m\x S^k,\om_\FS\oplus0\big)$ does not embed into $\big(\CP^m\x\R^{2k},\om_\FS\oplus\om_0\big)$.

\subsubsection*{Coisotropic Audin conjecture}
Recall that a topological space $X$ is called aspherical iff $\pi_k(X)=0$, for every $k\geq2$. Furthermore, a manifold is called spin iff it is orientable and its second Stiefel-Whitney number vanishes. 
\begin{thm}\label{thm:m 2} Let $(M,\om)$ be a symplectic manifold that is convex at infinity, and $N\sub M$ a coisotropic submanifold that is closed, regular, aspherical, spin, and displaceable. Then $m(N,\om)=2$. 
\end{thm}
In the Lagrangian case this result is due to K. Fukaya \cite{Fu}. It generalizes a conjecture by Audin about the minimal Maslov number of a Lagrangian submanifold of $\R^{2n}$ diffeomorphic to the torus $\TT^n$. The idea of proof of Theorem \ref{thm:m 2} is to reduce to the Lagrangian case using the construction (\ref{eq:wt M wt om},\ref{eq:iota N wt N}), and then to apply Fukaya's result.

\subsection{Related work}
\subsubsection*{Oh's Maslov index}
Let $J$ be an $\om$-compatible almost complex structure, assume that $N$ is gradable and equipped with a grading $[\De]$ in the sense of \cite{Oh}, and that $\Si=\D$. In this situation, Y.-G. Oh defined a Maslov index $\mu_{(N,\De)}:\big\{u\in C^\infty(\D,M)\,\big|\,u(S^1)\sub N\big\}\to\Z$, see Definition 3.3. in \cite{Oh}. If $u\in C(\D,M;N,\om)$ is a smooth map then $\mu_{(N,\De)}(u)=m_{\D,\om,N}(u)$. Note that $\mu_{(N,\De)}$ is defined on a larger set of maps than $m_{\D,\om,N}$ (after restriction to $C^\infty(\D,M)$), but requires $[\De]$ as an additional datum. Observe also that the definition of $m_{\Si,\om,N}$ does not involve the choice of any $\om$-compatible almost complex structure on $M$. 
\subsubsection*{The Gaio-Salamon Maslov index}\label{GS}
Let $(M,\om,G,\om)$ be a Hamiltonian $G$-manifold. This means that $(M,\om)$ is a symplectic manifold, and $G$ is a connected Lie group acting on $M$ in a Hamiltonian way, with moment map $\mu$. Assume that $G$ acts freely on $N:=\mu^{-1}(0)$. Let $\Si\in\S$. We define the map $m_{\Si,\om,\mu}:[\Si,M;N,\om]\to \Z$ as follows. Let $a\in[\Si,M;N,\om]$. We choose a representative $u$ of $a$, a symplectic vector space $(V,\Om)$ of dimension $\dim M$, a trivialization $\Psi\in\Iso\big(\Si\x V,\Om;u^*(TM,\om)\big)$, and points $z_X\in X$, for every $X\in\Con(\d\Si)$. We define $g:\d\Si\to G$ by defining $g(z)$ to be the unique solution of $u(z)=g(z)u(z_X)$, for every $z\in X$ and $X\in\Con(\d\Si)$. 

We define $m_{\Si,\om,\mu}(a):=m_\Om\big(S^1\ni z\mapsto \Psi_z^{-1}g(z)\cdot\Psi_1\big)$, where for every $g_0\in G$ we denote by $g_0\cdot:TM\to TM$ the differential of the action of $g_0$. By a standard homotopy argument, this number does not depend on the choices of $u,\Psi$ and $z_X$. By Lemma \ref{le:m om mu} below the maps $m_{\Si,\om,\mu}$ and $m_{\Si,\om,N}$ agree. 

For $\Si=\D$ the map $m_{\D,\om,\mu}$ was introduced by R.~Gaio and D.~A.~Salamon in \cite{GS}. (More precisely, their definition relies on a choice of an $\om$-compatible almost complex structure $J$ on $M$ and a unitary trivialization of $u^*TM$.)

\subsubsection*{Work by M.~Entov and L.~Polterovich and by V.~L.~Ginzburg}
Let now $(M,\om)$ be a closed (spherically) monotone symplectic manifold and $G$ a torus acting on $M$ in a Hamiltonian way, with moment map $\mu$. Then by Theorem 1.7 in the article \cite{EP} by M. Entov and L. Polterovich the pre-image $N$ of the special element of $\g^*$ under $\mu$ is strongly (i.e.~symplectically) non-displaceable. 

Assume that the action of $G$ on $N$ is free. Then by Lemma \ref{le:special} below $N\sub M$ is a closed, monotone regular coisotropic submanifold. Hence if $b_i(N,\Z_2)$ is non-zero for some $i\in\{\dim N-m(N,\om)+2,\ldots,m(N,\om)-2\}$ then it follows from Theorem \ref{thm:Fix} that $N$ is not leafwise displaceable (and hence not displaceable). Thus in this case we obtain a stronger statement than in Theorem 1.7 in \cite{EP}, provided that also $H^1(M,\R)=0$. 

In his recent paper \cite{Gi} (Theorem 1.5) V. L. Ginzburg proved an upper bound on the minimal Maslov number of a closed, stable, displaceable coisotropic submanifold.

\subsection{Organization and Acknowledgments}
\subsubsection*{Organization of the article} In Section \ref{sec:proofs} it is shown that the Maslov map for pairs of flat transports is well-defined, and Theorem \ref{thm:m Si E om} is proved. Section \ref{sec:proofs Fix non-emb m 2 F simply} contains the proofs of the other results of Section \ref{sec:main}. They are based on Theorem \ref{thm:m M om N}, which summarizes the main properties of the Maslov map. Section \ref{sec:proof:thm:m M om N} is devoted to the proof of this theorem, using a similar result for the coisotropic Maslov index for bundles (Theorem \ref{thm:m}). The appendix contains some results about the Salamon-Zehnder map, the Gaio-Salamon Maslov index, the relation with the mixed action-Maslov index, the linear holonomy of a foliation, and some topological results.

\subsubsection*{Acknowledgments}
I would like to thank Yael Karshon for her continuous support and enlightening discussions, Masrour Zoghi and Dietmar Salamon for useful comments, Shengda Hu for making me aware of Lemma \ref{le:bundle}, and Viktor L. Ginzburg for his interest in my work.

\section{Proof of Theorem \ref{thm:m Si E om} (Coisotropic Maslov map for bundles)}\label{sec:proofs}
The following lemma was used in Section \ref{sec:main}.
\begin{lemma}\label{le:m Phi Phi'} The number $m_\om(\Phi,\Phi')$ in (\ref{eq:m om Phi Phi'}) is well-defined, i.e. it does not depend on the choice of $z$. Furthermore, if $\Phi$ and $\Phi'$ are regular then $m_\om(\Phi,\Phi')\in 2\Z$. 
\end{lemma}
The next Remark is used in the proof of Lemma \ref{le:m Phi Phi'}.
\begin{rmk}\label{rmk:rho cont} Let $X$ be a topological space and $(E,\om)$ a symplectic vector bundle over $X$. Then the map $\Aut(E,\om)\ni(x,\Phi)\mapsto\rho_{\om_x}(\Phi)\in S^1$ is continuous. To see this, we choose a symplectic vector space $(V,\Om)$ of dimension $\rank E$. Let $(U,\Phi)$ be a pair, where $U\sub X$ is an open subset and $\Phi\in\Iso\big(U\x V,\Om;(E,\om)|_U\big)$. By Proposition \ref{prop:rho}(\ref{prop:rho:nat}) we have $\rho_{\om_x}(\Psi)=\rho_{\Om}\big(\Phi_x^{-1}\Psi\Phi_x\big)$, for every $x\in U$ and $\Phi\in\Aut(E_x,\om_\x)$. Since the map $\rho_\Om:\Aut(\Om)\to S^1$ is continuous, the statement follows. 
\end{rmk}
\begin{proof}[Proof of Lemma \ref{le:m Phi Phi'}]\setcounter{claim}{0} To prove the first assertion, let $z_0$ and $z_1$ be two choices of a path $z$ as above. We choose a map $z\in C\big([0,1]\x[0,1],C\big)$ such that $z(s,0)=z(s,1)$, for every $s\in[0,1]$, and $z(i,\cdot)=z_i$. We denote $z_s:=z(s,\cdot)$, and we define $\Psi_s(t):=\Phi'([z_s|_{[0,t]}])^{-1}\Phi([z_s|_{[0,t]}])$, for $s,t\in[0,1]$. We also define $f:[0,1]\x[0,1]\to S^1$ by $f(s,t):=\rho_{\om_{z_s(0)}}(\Psi_s(t))$. It follows that $\Psi_s(0)=\id_{E_{z_s(0)}}$ and hence $f(s,0)=1$, for every $s\in[0,1]$. By Remark \ref{rmk:rho cont} the map $f$ is continuous.
\begin{claim}\label{claim:rho Psi s} The map $[0,1]\ni s\mapsto f(s,1)\in S^1$ is constant.
\end{claim}
\begin{proof}[Proof of Claim \ref{claim:rho Psi s}] Consider the case in which $\Phi$ is regular. We choose a path $\wt z\in C([0,1],C)$ such that $\wt z(i)=z_i(0)$, for $i=0,1$. We fix $s\in[0,1]$. By assumption we have $\Phi\big(\big[\wt z\#z_s\#\BAR{\wt z}\big]\big)=\id_{E_{z_0(0)}}$. Furthermore, the paths $z_0$ and $\wt z\#z_s\#\BAR{\wt z}$ are homotopic with fixed end-points. It follows that $\Phi'([z_0])=\Phi'([\wt z])^{-1}\Phi'([z_s])\Phi'([\wt z])=\Phi'([\wt z])^{-1}\Psi_s(1)\Phi'([\wt z])$. Hence by Proposition \ref{prop:rho}(\ref{prop:rho:nat}) we have $f(s,0)=f(s,1)$. The case in which $\Phi'$ is regular, is treated similarly. This proves Claim \ref{claim:rho Psi s}.
\end{proof}
Claim \ref{claim:rho Psi s}, the fact $\Psi_s(0)=\id_{E_{z_s(0)}}$ (for every $s\in[0,1]$) and continuity of $f$ imply that $m_{\om_{z_0(0)}}(\Psi_0)=m_{\om_{z_1(0)}}(\Psi_1)$. Hence $m_\om(\Phi,\Phi')$ is well-defined. 

The second assertion of the lemma follows directly from the definition of the Maslov index of a path of automorphisms of a symplectic vector space. This proves Lemma \ref{le:m Phi Phi'}.
\end{proof}

For the proof of Theorem \ref{thm:m Si E om} we need the following. Let $X$ be a topological manifold and $\X\sub C([0,1],X)$. We define the equivalence relation $\sim_\X$ on $\X$ by $x_0\sim_\X x_1$ iff there exists $x\in C([0,1]\x[0,1],X)$ such that $x(s,\cdot)\in\X$, $x(s,i)=x(0,i)$ and $x(i,\cdot)=x_i$, for every $s\in[0,1]$ and $i=0,1$. We equip $\X$ with the compact open topology and $\X/\!\!\sim_\X$ with the quotient topology. Then $\X/\!\!\sim_\X$ is a topological groupoid. We call $\X$ \emph{admissible} iff it contains the constant paths, and the following conditions hold. If $x\in\X$ and $f\in C([0,1],[0,1])$ then $x\circ f\in\X$. Furthermore, if $x,x'\in\X$ are such that $x(1)=x'(0)$ then the concatenation $x\#x'$ lies in $\X$. Assume that $\X$ is admissible, and let $E\to X$ be a topological vector bundle. A \emph{flat transport on $E$ along $\X$} a morphism of topological groupoids $\Phi:\X/\!\!\sim_\X\to\GL(E)$ that descends to the identity on $X\x X$. We denote by $\T(\X,E)$ the set of such $\Phi$'s. Let $X'$ be another topological manifold and $f\in C(X',X)$. Then the pullback $f^*\X:=(f\circ)^{-1}(\X)\sub C([0,1],X')$ is again admissible. For $\Phi\in\T(\X,E)$ we define the pullback $f^*\Phi\in\T(f^*(\X,E))$ by $(f^*\Phi)_{x'}:=\Phi_{f\circ x'}$. 

Let $X$ be a topological manifold, $(E,\om)$ a symplectic vector bundle over $X$, $(W,\Phi)\in\Coi(E,\om)$ and $\Psi\in\T(E,\om)$. We call $\Psi$ a lift of $(W,\Phi)$ iff for every $z\in C([0,1],C)$ we have $\Psi([z])W_{z(0)}=W_{z(1)}$ and $\Psi([z])_{W_{z(0)}}=\Phi([z])$. Let $X$ be a closed curve. We denote by $\pi:[0,1]\x X\to X$ the canonical projection, and for $s\in[0,1]$, we define $\iota_s:X\to[0,1]\x X$ by $\iota_s(z):=(s,z)$. Furthermore, we define $\X:=\big\{t\mapsto (s,z(t))\,\big|\,s\in[0,1],\,z\in C([0,1],X)\big\}$. Let $(E,\om)$ be a symplectic vector bundle over $X$. 
\begin{thm}\label{thm:m X V om} The following statements hold.
\begin{enui}
\item\label{thm:m X V om:Psi} For every $(W,\Phi)\in\Coi(E,\om)$ there exists a lift $\Psi\in\T(E,\om)$ of $(W,\Phi)$. 
\item\label{thm:m X V om:Psi Psi'} Let $(W,\Phi)\in\Coi(E,\om)$, $\Psi_0$ and $\Psi_1$ be lifts of $(W,\Phi)$, and $\Psi\in\T(E,\om)$ a regular transport. Then $m_{X,\om}(\Psi_0,\Psi)=m_{X,\om}(\Psi_1,\Psi)$. 
\item\label{thm:m X V om:Pi} Let $W\sub\pi^*E$ be an $\pi^*\om$-coisotropic subbundle and $\Phi\in\T\big(\Pi([0,1],X),W_\om,\om_W\big)$. Then there exists $\Psi\in\T(E,\om)$ such that $\iota_s^*\Psi$ is a lift of $\iota_s^*\Phi$, for every $s\in[0,1]$.
\item\label{thm:m X V om:X} Let $W\sub\pi^*E$ be an $\pi^*\om$-coisotropic subbundle and $\Phi\in\T(\X,W_\om,\om_W)$. Then there exists $\Psi\in\T(\X,E,\om)$ such that $\iota_s^*\Psi$ is a lift of $\iota_s^*\Phi$, for every $s\in[0,1]$.
\end{enui}
\end{thm}
For the proof of Theorem \ref{thm:m X V om} we need the following. Let $f$ be a homeomorphism between two topological manifolds $X$ and $X'$, and $\pi:E\to X$ and $\pi':E'\to X'$ vector bundles. Assume that there exists $\Psi\in\Iso(E',E)$ that descends to $f$. We define $\Psi^*:\GL(E)\to\GL(E')$ by $\Psi^*(x_0,x_1,\Phi):=\big(f^{-1}(x_0),f^{-1}(x_1),\Psi_{f^{-1}(x_1)}^{-1}\Phi\Psi_{f^{-1}(x_0)}\big)$. For $\Phi\in\T(E)$ we define $\Psi^*\Phi:\Pi X'\to\GL(E')$ by $(\Psi^*\Phi)(a'):=\Psi^*(\Phi f_*a')$. 
\begin{lemma}\label{le:Psi * Phi} We have $\Psi^*\Phi\in\T(E')$. 
\end{lemma}
\begin{proof}[Proof of Lemma \ref{le:Psi * Phi}]\setcounter{claim}{0} It follows from the definitions that $\Psi^*\Phi$ descends to the identity on $X'$. Furthermore, the map $f_*:\Pi X'\to\Pi X$ is a morphism of topological groupoids. Since $\Phi\in\T(E)$, the same holds for $\Phi$. Finally, it follows from the definitions that $\Psi^*:\GL(E)\to\GL(E')$ is a morphism of groupoids. It follows from Lemma \ref{le:compact open}(\ref{le:compact open:exp}) that it is continuous. It follows that the map $\Pi X'\ni a'\mapsto\Psi^*(\Phi f_*a')\in\GL(E')$ is a flat transport. This proves Lemma \ref{le:Psi * Phi}.
\end{proof}
Let $(V,\om)$ be a symplectic vector space and $\ell\in\{\dim V/2,\ldots,\dim V\}$. We denote by $\Gr(\om,\ell)$ the set of $\om$-coisotropic subspaces of $V$ of dimension $\ell$, and equip it with the natural smooth structure. Let $W_0\sub W$ be a coisotropic subspace of dimension $\ell$. We define the \emph{framed coisotropic Grassmannian} $\Gr(\om,W_0)$ to be the set of all pairs $(W,\Phi)$, where $W\sub V$ is an coisotropic subspace and $\Phi\in\Iso(\om_{W_0},\om_W)$. This set is naturally equipped with a smooth structure. 
\begin{lemma}\label{le:bundle} The maps $\Aut(\om)\to\Gr(\om,\ell)$, $\Psi\mapsto \Psi W$, and $\Aut(\om)\to \Gr(\om,W_0)$, $\Psi\mapsto (\Psi W_0,\Psi_{W_0})$, are smooth (locally trivial) fiber bundles. 
\end{lemma}
For the proof of Lemma \ref{le:bundle} we need the following. The group $\Iso(\om)$ acts naturally on $\Gr(\om,\ell)$, and it acts on $\Gr(\om,W_0)$ by $\Psi(W,\Phi):=\big(\Psi W,\Psi_W\Phi\big)$. These actions are smooth. 
\begin{lemma}\label{le:transitive} They are transitive.
\end{lemma}
\begin{proof}[Proof of Lemma \ref{le:transitive}]\setcounter{claim}{0} Transitivity of the first action follows by an elementary argument. Let $(W,\Phi)\in\Gr(\om,W_0)$. Assume first that $W=W_0$. We choose a maximal symplectic subspace $V_0\sub W$, and define $f:V_0\to W_\om$, $fv_0:=v_0+W^\om$. Then $f\in\Iso(\om|_{V_0},\om_W)$, and hence we may define $\Psi:V=V_0\oplus V_0^\om\to V$ by $\Psi(v_0+v_1):=f^{-1}\Phi fv_0+v_1$. This map has the required properties.

For a general $W$ we choose $\Psi'\in\Iso(\om)$ such that $\Psi'W_0=W$. By what we just proved there exists $\Psi''\in\Iso(\om)$ such that $\Psi''W_0=W_0$ and $\Psi''_{W_0}={\Psi'}_{W_0}^{-1}\Phi$. The map $\Psi:=\Psi'\Psi''$ has the required properties. This proves Lemma \ref{le:transitive}.
\end{proof}
\begin{proof}[Proof of Lemma \ref{le:bundle}]\setcounter{claim}{0} If a Lie group $G$ acts smoothly on a manifold $X$ and $x\in X$, then the stabilizer $H$ of $x$ is a closed subgroup, and hence the map $G\to G/H$, $g\mapsto gH$, is a smooth fiber bundle. If the action is transitive then the map $G/H\to X$, $gH\mapsto gx$, is a diffeomorphism. Lemma \ref{le:bundle} follows from this and Lemma \ref{le:transitive}. 
\end{proof}
Let $f$ be a homeomorphism between two topological manifolds $X$ and $X'$, and $(\pi,E,\om)$ and $(\pi',E',\om')$ be symplectic vector bundles over $X$ and $X'$ respectively. Assume that there exists $F\in\wt\Iso\big(E',\om';E,\om\big)$. 
\begin{lemma}\label{le:Psi' F} The following statements hold.
\begin{enui}
\item\label{le:Psi' F:bij} The map $F^*:\Coi(E,\om)\to\Coi(E',\om')$ defined by $F^*(W,\Phi):=(F^{-1}W,F^*\Phi)$, is a bijection.
\item\label{le:Psi' F:W}Let $W\sub E$ be an $\om$-coisotropic subbundle and $\Psi\in\T(E,\om)$ be a transport that leaves $W$ invariant. Then the transport $\Psi':=F^*\Psi\in\T(E',\om')$ leaves the $\om'$-coisotropic subbundle $W':=F^{-1}W\sub E'$ invariant and $\Psi'_{W'}=F^*(\Psi_W)$. 
\end{enui}
\end{lemma}
\begin{proof}[Proof of Lemma \ref{le:Psi' F}]\setcounter{claim}{0} The statements follow from straight-forward arguments.
\end{proof}
\begin{rmk}[Naturality for one-dimensional Maslov map]\label{rmk:m Phi Phi 0} Let $X$ and $X'$ be closed oriented curves, $(E,\om)$ and $(E',\om')$ symplectic vectors bundle over $X$ and $X'$, respectively, $\Phi,\Phi_0\in\T(E,\om)$, and $\Psi\in\Iso\big(E',\om';E,\om\big)$. Assume that $\Phi_0$ is regular. Then $m_{X',\om'}(\Psi^*\Phi,\Psi^*\Phi_0)=m_{X,\om}(\Phi,\Phi_0)$. This follows from Proposition \ref{prop:rho}(\ref{prop:rho:nat}). 
\end{rmk}
\begin{proof}[Proof of Theorem \ref{thm:m X V om}]\setcounter{claim}{0} Without loss of generality, we may assume that $X$ is connected. 

To prove {\bf statement (\ref{thm:m X V om:Psi})}, assume first that $X=\R/\Z$ and there exists a symplectic vector space $(V,\Om)$ such that $E=\R/\Z\x V$ and $\om$ is constantly equal to $\Om$. Let $(W,\Phi)\in\Coi(\R/\Z\x V,\om)$. It follows from Lemma \ref{le:bundle} that there exists a path $\wt\Psi\in C([0,1],\Aut\Om)$ such that $\wt\Psi(0)=\id_V$ and $\wt\Psi(s)W_{0+\Z}=W_{s+\Z}$ and $\wt\Psi(s)_{W_{0+\Z}}=\Phi\big(\big[[0,1]\ni t\mapsto st+\Z\big]\big)$, for every $s\in[0,1]$. By an elementary argument there exists a unique transport $\Psi\in\T(E,\om)$ satisfying $\Psi\big(\big[[0,1]\ni t\mapsto st+\Z\big]\big)=\wt\Psi(s)$, for every $s\in[0,1]$. This is a lift of $(W,\Phi)$, as required.

In the general situation, we choose a homeomorphism $f:\R/\Z\to X$ and a symplectic vector space $(V,\Om)$ of dimension $\rank E$. Since $\Aut\Om$ is connected, there exists $F\in\Iso\big(\R/\Z\x V,\Om;E,\om\big)$ that descends to $f$. Statement (\ref{thm:m X V om:Psi}) follows now from what we already proved and Lemma \ref{le:Psi' F}.

To prove {\bf statement (\ref{thm:m X V om:Psi Psi'})}, we choose a symplectic vector space $(V,\Om)$ of dimension $\rank E$. Without loss of generality, we may assume that $E=X\x V$, $\om$ is constantly equal to $\Om$, and $\Psi\const\id$. (To see this, we choose $z_0\in X$ and we define $F\in\Iso(X\x V,\Om;E,\om)$ by $F_{z_1}:=\Psi([z])$, for $z_1\in X$, where $z\in C([0,1],X)$ is a path such that $z(i)=z_i$, for $i=0,1$. By regularity of $\Psi$ the map $F$ is well-defined. The claimed equality is a consequence of the equality $m_{X,\om}(F^*\Psi_0,\id)=m_{X,\om}(F^*\Psi_1,\id)$, the fact $F^*\Psi\const\id$, and Remark \ref{rmk:m Phi Phi 0}.)

Let $\Psi_0,\Psi_1\in\T(X\x V,\om)$ be lifts of $(W,\Phi)$. We choose a path $z\in C([0,1],X)$ such that $z(0)=z(1)$ and the map $S^1\iso [0,1]/\{0,1\}\ni [t]\mapsto z(t)\in X$ has degree one. We define $\pi:\Aut\om\to\Gr(\om,W_{z(0)})$ by $\pi(F):=\big(FW_{z(0)},F_{W_{z(0)}}\big)$, and $\wt\Psi_i:[0,1]\to\Aut\om$ by $\wt\Psi_i(t):= \Psi_i([z|_{[0,t]}])$, for $i=0,1$. Then $\pi\circ\wt\Psi_0(t)=\big(W_{z(t)},\Phi([z|_{[0,t]}])\big)=\pi\circ\wt\Psi_1(t)$, for every $t\in[0,1]$, and $\wt\Psi_0(0)=\id_V=\wt\Psi_1(0)$. Therefore, Lemma \ref{le:bundle} implies that there exists $\wt\Psi\in C\big([0,1]\x[0,1],\Aut\om\big)$ such that $\wt\Psi(i,\cdot)=\wt\Psi_i$, for $i=0,1$, $\wt\Psi(s,0)=\id_V$, and $\pi\circ\wt\Psi(s,t)=\big(W_{z(t)},\Phi([z|_{[0,t]}])\big)$, for every $s,t\in[0,1]$. Therefore, the hypotheses of Proposition \ref{prop:homotopy} are satisfied with $x(s,t):=z(t)$ and $\Psi:=\wt\Psi$. By the assertion of that proposition, we have $m_\om(\wt\Psi_0)=m_\om(\wt\Psi_1)$. Since $m_{X,\om}(\Psi_i,\id)=m_\Om(\wt\Psi_i)$, for $i=0,1$, it follows that $m_{X,\om}(\Psi_0,\id)=m_{X,\om}(\Psi_1,\id)$. This proves statement (\ref{thm:m X V om:Psi Psi'}).

{\bf Statements (\ref{thm:m X V om:Pi},\ref{thm:m X V om:X})} are proved similarly to statement (\ref{thm:m X V om:Psi}).

This completes the proof of Theorem \ref{thm:m X V om}. 
\end{proof}
\begin{lemma}\label{le:m Phi Phi 1} Let $\Si$ be a compact connected oriented surface with non-empty boundary, $(E,\om)$ a symplectic vector bundle over $\Si$, $\Phi,\Phi'\in\T(E,\om)$ regular transports, and $\Psi\in\T((E,\om)|_{\d\Si})$. Then $m_{\d\Si,\om|_{\d\Si}}(\Psi,\Phi|_{\d\Si})=m_{\d\Si,\om|_{\d\Si}}(\Psi,\Phi'|_{\d\Si})$. 
\end{lemma}
For the proof of Lemma \ref{le:m Phi Phi 1} we need the following.
\begin{lemma}\label{le:m om Phi Psi} Let $(V,\om)$ be a symplectic vector space, and $\Phi,\Psi\in C([0,1],\Aut\om)$ be such that $\Phi(0)=\Phi(1)=\id$ and $\Psi(0)=\id$. Then $m_\om(\Phi\Psi)=m_\om(\Phi)+m_\om(\Psi)$.  
\end{lemma}
\begin{proof}[Proof of Lemma \ref{le:m om Phi Psi}]\setcounter{claim}{0} By an elementary argument, the map $\Phi\Psi$ is homotopic with fixed end-points to the concatenation of $\Phi$ with $\Psi$. The statement follows from this. 
\end{proof}
\begin{lemma}\label{le:m Phi Psi T} Let $\Si$ be compact connected oriented surface with non-empty boundary, $(E,\om)$ a symplectic vector bundle over $\Si$, and $\Phi,\Psi\in\T(E,\om)$ be regular transports. Then $m_{\d\Si,\om|_{\d\Si}}(\Phi|_{\d\Si},\Psi|_{\d\Si})=0$. 
\end{lemma}
\begin{proof}[Proof of Lemma \ref{le:m Phi Psi T}]\setcounter{claim}{0} For $z_0\in\Si$ we define $f_{z_0}:\Si\to S^1$ by $f_{z_0}(z_1):=\rho_{\om_{z_0}}\big(\Psi([z])^{-1}\Phi([z])\big)$, for $z_1\in\Si$, where $z\in C([0,1],\Si)$ is a path such that $z(i)=z_i$, for $i=0,1$. By regularity of $\Phi$ and $\Psi$ this map is well-defined. Let $X$ be a connected component of $\d\Si$. Then for every $z_0\in X$ we have $\deg(f_{z_0}|_X)=m_{X,\om|_X}(\Phi|_X,\Psi|_X)$. Furthermore, for $z_0,z_0'\in\Si$ the maps $f_{z_0}$ and $f_{z_0'}$ are homotopic and hence $\deg(f_{z_0}|_X)=\deg(f_{z_0'}|_X)$. Let $z_0\in\Si$. It follows that $m_{\d\Si,\om|_{\d\Si}}(\Phi|_{\d\Si},\Psi|_{\d\Si})=\deg(f_{z_0}|_{\d\Si})=0$. This proves Lemma \ref{le:m Phi Psi T}.
\end{proof}

\begin{proof}[Proof of Lemma \ref{le:m Phi Phi 1}]\setcounter{claim}{0} Let $X$ be a connected component of $\d\Si$. We choose a path $z\in C([0,1],X)$ such that $z(0)=z(1)$ and the map $S^1\iso[0,1]/\{0,1\}\ni[t]\mapsto z(t)\in X$ has degree one. For $s\in[0,1]$ we define $z_s\in C([0,1],X)$ by $z_s(t):=z(st)$. Furthermore, we define $F\in C([0,1],\Aut(\om_{z(0)}))$ by $F(s):=\Phi'([z_s])^{-1}\Phi([z_s])$. By definition, we have $m_{X,\om|_X}(\Phi|_X,\Phi'|_X)=m_{\om_{z(0)}}(F)$. Therefore, using Lemma \ref{le:m om Phi Psi}, we obtain $m_{X,\om|_X}(\Psi,\Phi'|_X)=m_{X,\om|_X}(\Psi,\Phi|_X)+m_{X,\om|_X}(\Phi|_X,\Phi'|_X)$. The claimed equality follows now from Lemma \ref{le:m Phi Psi T}.
This proves Lemma \ref{le:m Phi Phi 1}.
\end{proof}
\begin{rmk}\label{rmk:T E om} Let $\Si$ be a compact connected oriented surface with non-empty boundary, and $(E,\om)$ a symplectic vector bundle over $\Si$. Then there exists a regular transport $\Phi\in\T(E,\om)$. To see this, we choose a symplectic vector space $(V,\Om)$ of dimension $\rank E$. Since $\Aut(\om)$ is connected and $\d\Si\neq\empty$, there exists $\Psi\in\Iso(\Si\x V,\Om;E,\om)$. We define $\Phi\in\T(E,\om)$ by $\Phi([z]):=\Psi_{z(1)}\Psi_{z(0)}^{-1}$. 
\end{rmk}
\begin{proof}[Proof of Theorem \ref{thm:m Si E om}]\setcounter{claim}{0} We show existence of the map $m_{\Si,E,\om}$. Let $(W,\Phi)\in\Coi(E,\om)$. By Remark \ref{rmk:T E om} we may choose a regular transport $\Psi_0\in\T(E,\om)$. By Theorem \ref{thm:m X V om}(\ref{thm:m X V om:Psi}) we may choose a lift $\Psi\in\T((E,\om)|_{\d\Si})$ of $(W,\Phi)$. We define $m_{\Si,E,\om}(W,\Phi):=m_{\d\Si,\om|_{\d\Si}}(\Psi,\Psi_0)$. By Theorem \ref{thm:m X V om}(\ref{thm:m X V om:Psi Psi'}) and Lemma \ref{le:m Phi Phi 1} this number does not depend on the choices of $\Psi$ and $\Psi_0$. %
Furthermore, the conditions (\ref{thm:m Si E om:boundary},\ref{thm:m Si E om:inv}) follow from the definition of $m_\Si$. 

To show uniqueness of the map $m_{\Si,E,\om}$, let $m_{\Si,E,\om}:\Coi(E,\om)\to\R$ be a map satisfying (\ref{thm:m Si E om:boundary},\ref{thm:m Si E om:inv}). Let $(W,\Phi)\in\Coi(E,\om)$. By Theorem \ref{thm:m X V om}(\ref{thm:m X V om:Psi}) and Remark \ref{rmk:T E om} we may choose a lift $\Psi\in\T((E,\om)|_{\d\Si})$ of $(W,\Phi)$ and a regular transport $\Psi_0\in\T(E,\om)$. By condition (\ref{thm:m Si E om:inv}) we have $m_{\Si,E,\om}(W,\Phi)=m_{\Si,E,\om}(E|_{\d\Si},\Psi)=m_{\d\Si,\om|_{\d\Si}}(\Psi,\Psi_0)$. Uniqueness follows.

This proves Theorem \ref{thm:m Si E om}.
\end{proof}
\section{Proofs of Theorems \ref{thm:Fix}, \ref{thm:non-emb}, \ref{thm:m 2}, and Proposition \ref{prop:F simply}}\label{sec:proofs Fix non-emb m 2 F simply} 
For the proof of these results, we need the following theorem, which summarizes some properties of the Maslov map. If $\Si\in\S$, $(M,\om)$ and $(M',\om')$ are symplectic manifolds, and $N\sub M$ and $N'\sub M'$ are coisotropic submanifolds, then there is a canonical bijection $\Phi_{M,M',\om,\om',N,N'}:\big[\Si,M;N,\om\big]\x\big[\Si,M';N',\om'\big]\to\big[\Si,M\x M';N\x N',\om\oplus\om'\big]$. If $X$ and $X'$ are sets and $f:X\to\R$ and $f':X'\to\R$ are maps then we define $f\oplus f':X\x X'\to\R$ by $f\oplus f'(x,x'):=f(x)+f'(x')$. Let $X$ be a manifold and $Y\sub X\wo\d X$ is a submanifold of codimension one. We define $X_Y$ to be the manifold with boundary obtained from $X$ by cutting along $Y$. (Note that if $X$ is orientable then $\d X_Y=\d X\disj Y\disj Y$.) There is a canonical map $f_{X,Y}:X_Y\to X$. If $\Si\in\S$, $(M,\om)$ is a symplectic manifold and $L\sub M$ is a Lagrangian submanifold then we denote by $m^L_{\Si,\om}$ the Lagrangian Maslov map (see the appendix, (\ref{eq:m L Si}). For a manifold $X$ we denote by $\sim_X$ the equivalence relation on $X$ given by $x\sim_X x'$ iff $x=x'$ or $x$ and $x'$ lie in the same connected component of $\d X$, and we denote by $\pi_X:X\to X/\!\!\sim_X$ the canonical projection. Let now $\Si,\Si'\in\S$, $f:\Si'\to\Si$ be an embedding (restricting an embedding of $\d\Si'$ into $\d\Si$), $M$ and $M'$ manifolds of the same dimension, $\om$ a symplectic form on $M$, $N\sub M$ a coisotropic submanifold and $\phi:M'\to M$ an embedding. We denote $\phi^*N:=\phi^{-1}(N)$. The map $\phi$ induces a map $\phi_*:\big[\Si',M';\phi^*(N,\om)\big]\to\big[\Si,M;N,\om\big]$. Recall the definition (\ref{eq:wt M wt om},\ref{eq:iota N wt N}). We define the map $\phi:\big\lan\D,M;N,\om\big\ran\to\big[\D,S^1;\wt M,\wt N\big]$ by $\phi(a):=[u,u(z_0)]$, where $u$ is an arbitrary representative of $a$ and $z_0\in S^1$ is any point. This map is well-defined. 
\begin{thm}[Properties of the Maslov map]\label{thm:m M om N} The following assertions hold.
\begin{enui}
\item\label{thm:m M om N:nat}(Naturality) Let $\Si,\Si',f,M,M',N$ and $\phi$ be as above. If $f$ is surjective and orientation preserving then $m_{\Si',\phi^*(\om,N)}=m_{\Si,\om,N}\circ\phi_*$.
\item\label{thm:m M om N:prod}(Product) If $\Si\in\S$, $(M,\om)$ and $(M',\om')$ are symplectic manifolds, and $N\sub M$ and $N'\sub M'$ are coisotropic submanifolds, then $m_{\Si,\om,N}\oplus m_{\Si,\om',N'}=m_{\Si,\om\oplus\om',N\x N'}\circ\Phi_{M,M',\om,\om',N,N'}$. 
\item\label{thm:m M om N:homotopy} If $u\in C([0,1]\x\Si,M)$ is an admissible homotopy then the map $[0,1]\ni t\mapsto m_{\Si,N}(u(t,\cdot))\in\R$ is constant. 
\item\label{thm:m M om N:weak homotopy} If $u\in C([0,1]\x\Si,M)$ is a weakly $(N,\om)$-admissible homotopy then the map $[0,1]\ni t\mapsto m_{\Si,N}(u(t,\cdot))\in\R$ is continuous. 
\item\label{thm:m M om N:split}(Splitting) Let $\Si\in\S$, $C\sub\Si\wo\d\Si$ a closed curve (possibly disconnected), $(M,\om)$ a symplectic manifold, $N\sub M$ a coisotropic submanifold, and $u\in C(\Si,M)$ be such that for every $C'\in\Con(\d\Si\cup C)$ there exists $F\in N_\om$ such that $u(C')\sub F$. Then $m_{\om,N}(u)=m_{\om,N}(u\circ f_{\Si,C})$. 
\item\label{thm:m M om N:reg}(Regular case) If $N$ is regular then $\im(m_{\Si,N})\sub\Z$. If $N$ is also orientable then $\im(m_{\Si,N})\sub2\Z$. 
\item\label{thm:m M om N:Lag}(Lagrangian case) Let $\Si\in\S$, $(M,\om)$ be a symplectic manifold and $L\sub M$ a Lagrangian submanifold. Then $m_{\Si,\om,L}=m^L_{\Si,\om}$. 
\item\label{thm:m M om N:M}(Removal of point) Let $\Si\in\S$ be such that $\d\Si\neq\empty$, $C\in\Con(\Si)$, $(M,\om)$ be a symplectic manifold, $N\sub M$ a coisotropic submanifold, and $u\in C(\Si,M;N,\om)$. Assume that $u$ maps $C$ to a point in $N$. We define $\wt u:\Si/\!\!\sim_C\to M$ by $\wt u([z]):=u(z)$. Then $m_{M,\om,N}(u)=m_{M,\om,N}(\wt u)$. 
\item\label{thm:m M om N:N om}(Chern class of quotient) Let $M,\om$ and $N$ be as above, with $N$ regular. Let $u\in C(\D,M;N,\om)$ be such that $u(\D)\sub N$. We define $u':\D/\!\!\sim_\D\iso S^2\to N_\om$ by $u'([z]):=\pi_N\circ u(z)$. Then $m_{M,\om,N}(u)=2c_1^{N_\om,\om_N}(u')$. 
\item\label{thm:m M om N:wt M} Let $(M,\om)$ be a symplectic manifold and $N\sub M$ a regular coisotropic submanifold. Then $m_{\D,M,\om,N}=m_{\D,\wt M,\wt\om,\wt N}\circ\phi$. 
\end{enui}
\end{thm}
The proof of Theorem \ref{thm:m M om N} is given on page \pageref{proof:thm:m M om N}. The proof of Theorem \ref{thm:Fix} is based on the following result, which is due to P. Albers. 
\begin{thm}[\cite{Al}, Corollary 2.3] \label{thm:Lag} Let $(M,\om)$ be a bounded symplectic manifold, $L\sub M$ a closed monotone Lagrangian submanifold of minimal Maslov number $m(L)$, and $\phi\in\Ham(M,\om)$ be such that $L\pitchfork\phi(L)$. Then $|L\cap\phi(L)|\geq\sum_{i=\dim L-m(L)+2}^{m(L)-2}b_i(L,\Z_2)$. 
\end{thm}
Note that in \cite{Al}, Corollary 2.3, it is assumed that $M$ is closed. However, the proof of the result carries over to the case in which $(M,\om)$ is bounded.

\begin{proof}[Proof of Theorem \ref{thm:Fix}]\setcounter{claim}{0} Without loss of generality we may assume that $N$ is connected. Since $N$ is regular, there exists a unique smooth structure $\A_{N,\om}$ on the set of isotropic leaves $N_\om$ such that the canonical projection $\pi_N:N\to N_\om$ is a submersion. (See \cite{ZiLeafwise}, Lemma 15.) We define $\wt M,\wt\om,\iota_N$ and $\wt N$ as in (\ref{eq:wt M wt om},\ref{eq:iota N wt N}), and $\wt\phi:=\phi\x\id_{N_\om}:\wt M\to\wt M$. Then $\wt M$ is closed, the map $\iota_N:N\to\wt M$ is an embedding, and its image $\wt N$ is a closed Lagrangian submanifold, see \cite{ZiLeafwise}, Lemma 8. By the same lemma, $\wt\phi(\wt N)\pitchfork\wt N$. We denote by $R^{N,\om}$ the isotropic leaf relation on $N$. By Ehresmann's fibration theorem the map $\pi_N$ is a smooth (locally trivial) fiber bundle. (See \cite{Eh}, the proposition on p.~31.) Hence the hypotheses of Proposition \ref{prop:phi} with $(X,Y,\sim,\iota,\pi,k):=(M,N,R^{N,\om},\iota_N,\pi_N,2)$ are satisfied. Therefore, by the statement of this result and by Theorem \ref{thm:m M om N}(\ref{thm:m M om N:wt M}) the Lagrangian $\wt N$ is monotone and $m(\wt N,\wt\om)=m(N,\om)$. Therefore, the hypotheses of Theorem \ref{thm:Lag} are satisfied with $M,\om$ replaced by $\wt M,\wt\om$, and $L:=\wt N$. Inequality (\ref{eq:Fix phi N}) follows from the statement of this theorem and the fact $|\Fix(\phi,N)|=|\wt N\cap\wt\phi(\wt N)|$, see \cite{ZiLeafwise}, Lemma 8. This proves  Theorem \ref{thm:Fix}.
\end{proof}

For the proof of Theorem \ref{thm:non-emb} we need the following. Let $(M,\om)$ and $(M',\om')$ be presymplectic manifolds such that $\dim M+\corank\om=\dim M'+\corank\om'$. Assume that there exists a presymplectic embedding $\phi$ of $(M',\om')$ into $(M,\om)$. Then $N:=\phi(M')\sub M$ is a coisotropic submanifold (see \cite{ZiLeafwise}). Furthermore, $(M',\om')$ is regular if and only if $N$ is regular. 
\begin{prop}\label{prop:M'} Assume that $\corank\om=0$, $(M',\om')$ is regular and for every isotropic leaf $F\sub M'$ every loop $u\in C(S^1,F)$ is contractible in $M'$. If there exists a constant $c\in\R$ such that $2c_1^{M,\om}=c[\om]$ on $[S^2,M]$ and $2c_1^{M'_{\om'},\om'_{M'}}=c[\om']$ on $[S^2,M'_{\om'}]$, then $m_{\D,N,\om}(u)=c\int_\D u^*\om$, for every $u\in C^\infty(\D,M;N,\om)$. Furthermore, if $M$ is connected then
\begin{equation}\label{eq:mu N om 2 gcd}m(N,\om)=2\gcd\big(c_1(M,\om),c_1(M'_{\om'},\om'_{M'})\big).\end{equation} 
\end{prop}
\begin{proof}[Proof of Proposition \ref{prop:M'}]\setcounter{claim}{0} To prove the first statement, assume that there exists a constant $c\in\R$ such that $2c_1^{M,\om}=c[\om]$ on $[S^2,M]$ and $2c_1^{M'_{\om'},\om'_{M'}}=c[\om']$ on $[S^2,M']$. Let $a\in[\D,M;N,\om]$. We choose a smooth representative $u\in a$. Then $\phi^{-1}\circ u|_{S^1}$ is a continuous loop in $M'^{\om'}_{\phi^{-1}\circ u(1)}$, and hence by assumption it is contractible in $M'$. Hence there exists $v\in C(\D,M')$ such that $v|_{S^1}=\phi^{-1}\circ u|_{S^1}$. Smoothing the map $\phi\circ v$ out, we obtain a map $w\in C^\infty(\D,N)$ such that $w|_{S^1}=u|_{S^1}$. We denote by $\bar \D$ the disk with the reversed orientation and by $u\# w:\D\#\bar\D\to M$ the connected sum of $u$ and $w$. We have
\begin{eqnarray}\nn m_{M,\om,N}(a)&=&m\big(u^*(TM,\om),u|_{S^1}^*(TN,\hol^{N,\om})\big)\\
\nn&=&2c_1^{M,\om}\big((u\#w)^*(TM,\om\big)-m\big(w^*(TM,\om),w|_{S^1}^*(TN,\hol^{N,\om})\big)
\end{eqnarray}
The first statement follows from this. 

To prove the second statement, assume that $M$ is connected. We claim that 
\begin{equation}\label{eq:m D M N om}m\big(\big[\D,M;N,\om\big]\big)= 2c_1^{M,\om}([S^2,M])+2c_1^{M'_{\om'},\om'_{M'}}([S^2,M'_{\om'}]).\end{equation}
In order to show that the inclusion ``$\sub$'' in (\ref{eq:m D M N om}) holds, let $a\in\big[\D,M;N,\om\big]$. We choose a representative $u\in C(\D,M;N,\om)$ of $a$. By assumption the map $\pi_1(N^\om_{u(1)})\to\pi_1(E)$ vanishes. Hence there exists $\wt u\in C(\D,N)$ such that $\wt u|_{S^1}=u|_{S^1}$. We denote by $\BAR\D$ the disk with the opposite orientation, and define $v$ to be the connected sum $u\#\wt u:\D\#\BAR\D\iso S^2\to M$. It follows from Theorem \ref{thm:m M om N}(\ref{thm:m M om N:split}) that $m_{M,\om,N}(a)=2c_1^{M,\om}(v)-m_{M,\om,N}(\wt u)$. 
We define $u':\D/\!\!\sim_\D\iso S^2\to N_\om$ by $u'([z]):=\pi_N\circ \wt u(z)$. By Theorem \ref{thm:m M om N}(\ref{thm:m M om N:N om}) we have $m_{M,\om,N}(\wt u)=2c_1^{N_\om,\om_N}(u')$. The inclusion ``$\sub$'' in (\ref{eq:m D M N om}) follows. 

To prove the inclusion ``$\cont$'', observe that $c_1^{M,\om}([S^2,M])=c_1^{M,\om}([\D/\!\!\sim_\D,M])$ and $c_1^{M'_{\om'},\om'_{M'}}([S^2,M'_{\om'}])=c_1^{M'_{\om'},\om'_{M'}}([\D/\!\!\sim_\D,M'_{\om'}])$, since $\D/\!\!\sim$ is homeomorphic to $S^2$. Let $a\in [\D/\!\!\sim_\D,M]$. Since by assumption $M$ is connected, there exists a representative $u\in C(\D/\!\!\sim_\D,M)$ of $a$ such that $u([1])\in N$. It follows from Theorem \ref{thm:m M om N}(\ref{thm:m M om N:M}) that $m_{M,\om,N}(u\circ\pi_\D)=2c_1^{M,\om}(u)$. It follows that $2c_1^{M,\om}([S^2,M])\sub m_{M,\om,N}\big(\big[\D,M;N,\om\big]\big)$. 

Let now $a'\in [\D/\!\!\sim_\D,M'_{\om'}]$. We choose a representative $u'\in C(\D/\!\!\sim_\D,M'_{\om'})$ of $a'$. We claim that there exists a map $v:\D\to M'$ such that $\pi_{M'}\circ v=u'$. To see this, we define $h':[0,1]\x S^1\to M'$ by $h'(r,z):=u'(rz)$, and we choose $x_0\in \pi_N^{-1}(u'(0))\sub N$. By the homotopy lifting property there exists a map $h:[0,1]\x S^1\to M'$ such that $\pi_{M'}\circ h=h'$ and $h(0,z)=x_0$, for every $z\in S^1$. We define $v:\D\to M'$ by $v(0):=x_0$ and $v(z):=h(|z|,z/|z|)$, for every $z\neq0$. This map has the required properties. This proves the claim. 

We define $\phi':M'_{\om'}\to N_\om$ to be the unique map satisfying $\pi_N\circ\phi=\phi'\circ\pi_{M'}$. Then $\phi'\in\Iso(\om'_{M'},\om_N)$. Theorem \ref{thm:m M om N}(\ref{thm:m M om N:N om}) implies that $m_{M,\om}(\phi\circ v)=2c_1^{N_\om,\om_N}(\phi'\circ u')$. Furthermore, by Theorem \ref{thm:m M om N}(\ref{thm:m M om N:nat}) we have $c_1^{N_\om,\om_N}(\phi'\circ u')=c_1^{M'_{\om'},\om'_{M'}}(u')$. It follows that $2c_1^{M',\om'}([S^2,M'])\sub m_{M,\om,N}\big(\big[\D,M;N,\om\big]\big)$. The inclusion ``$\cont$'' in (\ref{eq:m D M N om}) follows. This proves (\ref{eq:m D M N om}). Since $M$ is connected, we have $c_1^{M,\om}([S^2,M])=c_1(M,\om)\Z$ and $c_1^{M'_{\om'},\om'_{M'}}([S^2,M'_{\om'}])=c_1(M'_{\om'},\om'_{M'})\Z$ (with the convention $\infty\Z=\{0\}$). Combining this with (\ref{eq:m D M N om}), the second statement follows. 

This completes the proof of Proposition \ref{prop:M'}.
\end{proof}
\begin{proof}[Proof of Theorem \ref{thm:non-emb}]\setcounter{claim}{0} Let $(M,\om)$ be a connected symplectic manifold and $(M',\om')$ a regular connected presymplectic manifold. We define 
\[\mu:=2\gcd\big(c_1(M,\om),c_1(M'_{\om'},\om'_{M'})\big).\] 
Assume that $\dim M'+\corank\om'=\dim M$ and there exists an embedding $\phi$ of $(M',\om')$ into $(M,\om)$. It follows that $N:=\phi(M')\sub M$ is a regular coisotropic submanifold (see \cite{ZiLeafwise}). Furthermore, if there exists a constant $c\in\R$ such that $2c_1^{M,\om}=c[\om]$ on $[S^2,M]$ and $2c_1^{M'_{\om'},\om'_{M'}}=c[\om']$ on $[S^2,M'_{\om'}]$ then Proposition \ref{prop:M'} implies that the coisotropic submanifold $N:=\phi(M')\sub M$ is monotone and $m(N,\om)=\mu$. Hence the statement of Theorem \ref{thm:non-emb} follows from Theorem \ref{thm:Fix}. 
\end{proof}
For the proof of Proposition \ref{prop:F simply}, we need the following remarks.
\begin{rmk}\label{rmk:min Chern} Let $(M,\om)$ be a connected symplectic manifold. Then $c_1^{M,\om}([S^2,M])=c_1(M,\om)\Z$, if $c_1(M,\om)<\infty$, and $c_1^{M,\om}([S^2,M])=\{0\}$, otherwise. To see this, we choose a point $x_0\in M$. Then the composition of the forgetful map $\pi_2(M,x_0)\to[S^2,M]$ with the map $c_1^{M,\om}:[S^2,M]\to\Z$ is a group homomorphism. The statement follows from this. 
\end{rmk}
\begin{proof}[Proof of Proposition \ref{prop:F simply}]\setcounter{claim}{0} Let $M,\om,M',\om'$ and $F$ be as in the hypothesis. Using Remark \ref{rmk:min Chern}, the statement of Proposition \ref{prop:F simply} is a consequence of the following.
\begin{claim}\label{claim:a} For every $a'\in[S^2,M'_{\om'}]$ there exists $a\in[S^2,M]$ such that $\lan[\om'_{M'}],a'\ran=\lan[\om],a\ran$ and $c_1^{M'_{\om'},\om'_{M'}}(a')=c_1^{M,\om}(a)$. 
\end{claim}
\begin{pf}[Proof of Claim \ref{claim:a}] We choose an isotropic leaf $F\sub M'$ and an orientation preserving homeomorphism $f:\D/\sim_\D\to S^2$. Since $F$ is simply-connected, it follows from the long exact homotopy sequence for the fibration $\pi_{M'}:M'\to M'_{\om'}$ that there exists $u'\in C(S^2,M')$ such that $[\pi_{M'}\circ u']=a'$. We define $a:=[\phi\circ u']$. To see that $a$ has the required properties, we denote by $\pi_\D:\D\to\D/\sim_\D$ the canonical projection. Then $\lan[\om],a\ran=\lan[\om'],[u']\ran=\lan[\om'_{M'}],a'\ran$. We define $u:=\phi\circ u'\circ f\circ\pi_\D$. Since $N\sub M$ is a coisotropic submanifold, it follows from Theorem \ref{thm:m M om N}(\ref{thm:m M om N:M},\ref{thm:m M om N:nat}) that $m_{M,\om,N}(u)=m_{M,\om,N}(\phi\circ u'\circ f)=m_{M,\om,N}(\phi\circ u')=2c_1^{M,\om}(a)$. On the other hand, by Theorem \ref{thm:m M om N}(\ref{thm:m M om N:N om}) we have $m_{M,\om,N}(u)=2c_1^{N_\om,\om_N}(\pi_N\circ\phi\circ u'\circ f)$. We denote by $\pi_{M'}:M'\to M'_{\om'}$ the canonical projection. The map $M'_{\om'}\ni \pi_{M'}(x')\mapsto \pi_N\circ\phi(x')\in N_\om$ is a well-defined $(\om'_{M'},\om_N)$-isomorphism. Therefore, Theorem \ref{thm:m M om N}(\ref{thm:m M om N:nat}) implies that $c_1^{N_\om,\om_N}(\pi_N\circ\phi\circ u'\circ f)=c_1^{M'_{\om'},\om'_{M'}}(\pi_{M'}\circ u')=c_1^{M'_{\om'},\om'_{M'}}(a')$. It follows that $c_1^{M,\om}(a)=c_1^{M'_{\om'},\om'_{M'}}(a')$. This proves Claim \ref{claim:a} and completes the proof of Proposition \ref{prop:F simply}.
\end{pf}\end{proof} 
The proof of Theorem \ref{thm:m 2} is based on the following result, which is due to K. Fukaya.
\begin{thm}[\cite{Fu}, Theorem 12.2.]\label{thm:Fukaya} Let $(M,\om)$ be a symplectic manifold and $L\sub M$ a Lagrangian submanifold. Assume that $(M,\om)$ is convex at infinity and $L$ is closed, relatively spin, aspherical and displaceable in a Hamiltonian way. Then there exists $a\in[\D,S^1;M,L]$ such that $m_{\D,\om,L}(a)=2$. 
\end{thm}
\begin{proof}[Proof of Theorem \ref{thm:m 2}]\setcounter{claim}{0} Since $N$ is regular and orientable, by Theorem \ref{thm:m M om N}(\ref{thm:m M om N:reg}) we have $\im(m_{\D,\om,N})\sub2\Z$. Hence the statement follows from Theorem \ref{thm:Fukaya} applied with $M,\om$ replaced by $\wt M,\wt\om$ and $L:=\wt N$ (as in (\ref{eq:wt M wt om},\ref{eq:iota N wt N})), Propositions \ref{prop:phi} and Theorem \ref{thm:m M om N}(\ref{thm:m M om N:wt M}).  
\end{proof}
\section{Proof of Theorem \ref{thm:m M om N} (Properties of the Maslov map)}\label{sec:proof:thm:m M om N}
The proof of Theorem \ref{thm:m M om N} is based on the following. Let $X$ be a topological manifold. We define $\E^0_X\sub \E_X$ to be the subclass of all quadruples $(E,\om,W,\Phi)$ such that $W=E|_{\d\Si}$ and for every $x\in C([0,1],\d X)$ satisfying $x(0)=x(1)$ we have $\Phi([x])=\id$. Furthermore, we define $\E^L_X\sub \E_X$ to be the subclass of all quadruples $(E,\om,W,\Phi)$ such that $W\sub E|_{\d\Si}$ is Lagrangian. 

Let $X$ be a topological manifold and $Y\sub X\wo\d X$ a hypersurface (i.e. a (real) codimension one submanifold) without boundary. Assume that $Y$ is closed as a subset. Then cutting $X$ along $Y$ we obtain a manifold with boundary $X_Y$. We denote by $\pr_Y^X:X_Y\to X$ the natural map, and define $Y^X:=(\pr_Y^X)^{-1}(Y)\sub X_Y$. (Note that if $Y$ is co-orientable in $X$ then $Y^X$ consists of two copies of $Y$.) As an example, let $Y$ be a topological manifold. We define $X:=\R\x Y$. Then $X_Y=((-\infty,0]\x Y)\disj([0,\infty)\x Y)$ and $Y^X=(\{0\}\x Y)\disj(\{0\}\x Y)$. We define $m^L_\Si:\E^L_\Si\to \Z$ as in (\ref{eq:m L Si}) in the appendix.

Let $X$ be topological manifold, $Y\sub X$ a closed subset, $E\to X$ a real vector bundle and $\Phi:Y\x Y\to\GL(E)$ a morphism of topological groupoids whose composition with the canonical projection $\GL(E)\to X\x X$ is the identity. We denote by $X/Y$ the topological space obtained by collapsing $Y$ to a point. Furthermore, we define the equivalence relation $\sim_\Phi$ on $E$ by $(x,v)\sim_\Phi(x',v')$ iff $(x,v)=(x',v')$ or ($x,x'\in Y$ and $v'=\Phi_x^{x'}v$). We define $\pi_\Phi:E/\sim_\Phi\to X/Y$ by $\pi_\Phi([x,v]):=[x]$. Assume that there exists a pair $(U,r)$, where $U\sub X$ is an open neighborhood of $Y$ and $r\in C([0,1]\x U,U)$ is a strong deformation retraction to $Y$. Then by Lemma \ref{le:E Phi} below $(E_\Phi:=E/\sim_\Phi,\pi_\Phi)$ is a vector bundle. Let $k\in\N$ and $T:E^{\oplus k}\to\R$ be a tensor, such that $T(\Phi_x^{x'}v_1,\ldots,\Phi_x^{x'}v_k)=T(v_1,\ldots,v_k)$, for every $x,x'\in Y$ and $v_1,\ldots,v_k\in E_x$. We define $T_\Phi:E_\Phi^{\oplus k}\to\R$ by $T_\Phi\big([x,v_1],\ldots,[x,v_k]):=T_x(v_1,\ldots,v_k)$. By Lemma \ref{le:E Phi} this is a tensor. Let now $(E,\om,W,\Phi)\in\E_\Si$ and $C\sub \d\Si$ be a connected component. Assume that $W|_C=E_C$ and $\Phi|_C$ is regular. We define $\Psi:C\x C\to\GL(E)$ by $\Psi_{z_0}^{z_1}:=\Phi([z]):E_{z_0}\to E_{z_1}$, where $z\in C([0,1],C)$ is any path such that $z(i)=z_i$ for $i=0,1$. By regularity of $\Phi$ this map is well-defined. We denote $(E,\om,W,\Phi)/C:=\big(E_\Psi,\om_\Psi,W|_{\d\Si\wo C},\Phi|_{\d\Si\wo C}\big)\in\E_{\Si/C}$. Assume that $\Si=[0,1)\x S^1$, there exists a vector space $V$ such that $E=\Si\x V$, and $\om$ is constant. For every point $z_0\in S^1$ define $\Psi'_{z_0}:\Si'\x V\to E'$ by $\Psi'_{[t,z]}v:=[t,z,\Phi_{z_0}^zv]$. These maps induce on $E'$ the structure of a (trivial) vector bundle over $\Si'$. In the general case we equip $E'$ with the vector bundle structure that restricts to the structure of $E$ on $\Si\wo\d\Si$ and is given as above on collar neighborhoods of the components of the boundary. The form $\om$ induces a fiberwise symplectic form $\om'$ on $E'$. We denote by $\Coi(\X,(E,\om)|_{[0,1]\x\d\Si})$ the set of all pairs $(W,\Phi)$, where $W\sub E$ is an $\om$-coisotropic subbundle, and $\Phi\in\T(\X,W_\om,\om_W)$.

Let $X$ be a topological manifold and $\X\sub C([0,1],X)$ be an admissible subset. We call $\Phi\in\T(\X,E)$ \emph{regular} iff $\Phi([x])=\id$ for every $x\in \X$ satisfying $x(0)=x(1)$. For a symplectic vector bundle $(E,\om)$ over an oriented topological surface $\Si$ we denote by $c_1(E,\om)$ its first Chern number. 
\begin{thm}[Properties of the coisotropic Maslov map for bundles]\label{thm:m} The following statements hold.
\begin{enui}
\item\label{thm:m:nat}\emph{(Naturality)} If $\Si,\Si'\in\S$, $\big(E,\om,W,\Phi\big)\in\E_\Si$, $\big(E',\om',W',\Phi'\big)\in\E_{\Si'}$, and $\Psi\in\wt\Iso(\om,\om')$ is such that $\Psi|_{\d\Si}^*(W',\Phi')=(W,\Phi)$, then $m_\Si\big(E,\om,W,\Phi\big)=m_{\Si'}\big(E',\om',W',\Phi'\big)$. 
\item\label{thm:m:sum} \emph{(Direct sum)} For every $\Si\in\S$ and $\big(E,\om,W,\Phi\big),\big(E',\om',W',\Phi'\big)\in\E_\Si$ we have
\[m_\Si\big(E\oplus E',\om\oplus\om',W\oplus W',\Phi\oplus\Phi'\big)=m_\Si\big(E,\om,W,\Phi\big)+m_\Si\big(E',\om',W',\Phi'\big).\] 
\item\label{thm:m:homotopy}\emph{(Homotopy)} Let $\Si\in\S$, $(E,\om)$ be a symplectic vector bundle over $[0,1]\x\Si$, and $(W,\Phi)\in\Coi((E,\om)|_{[0,1]\x\d\Si})$. Then the map $[0,1]\ni t\mapsto m_\Si\big((E,\om)|_{\{t\}\x\Si},W|_{\{t\}\x\d\Si},\Phi|_{\Pi(\{t\}\x\d\Si)}\big)$ is constant. 
\item\label{thm:m:weak homotopy}\emph{(Weak homotopy)} Let $\Si\in\S$, $(E,\om)$ be a symplectic vector bundle over $[0,1]\x\Si$, $\X:=\big\{t\mapsto (s,z(t))\,\big|\,s\in[0,1],\,z\in C([0,1],\d\Si)\big\}$, and $(W,\Phi)\in\Coi(\X,(E,\om)|_{[0,1]\x\d\Si})$. Then the map $[0,1]\ni t\mapsto m_\Si\big((E,\om)|_{\{t\}\x\Si},W|_{\{t\}\x\d\Si},\Phi|_{\Pi(\{t\}\x\d\Si)}\big)$ is continuous.
\item\label{thm:m:split}\emph{(Splitting)} Let $\Si\in\S$, $C\sub\Si\wo\d\Si$ be a closed curve, $\big(E,\om,W,\Phi\big)\in\E_\Si$ and $(W',\Phi')\in\Coi((E,\om)_C)$. 
Then
\begin{equation}\label{eq:m Si E}m(E,\om,W,\Phi)=m\big({\pr_C^\Si}^*\big(E,\om,(W,\Phi)\disj(W',\Phi')\big)\big).\end{equation} 
\item\label{thm:m:Lag}\emph{(Lagrangian case)} If $\Si\in\S$ and $(E,\om,W)\in\E^L_\Si$ then $m_\Si(E,\om,W,0)=m^L_\Si(E,\om,W)$.
\item\label{thm:m:full}\emph{(Full case)} Let $\Si\in\S$ and $(E,\om,W,\Phi)\in\E_\Si$. Assume that there exists $C\in\Con(\d\Si)$ such that $W|_C=E|_C$ and $\Phi|_C$ is regular. Then $m(E,\om,W,\Phi)=m\big((E,\om,W,\Phi)/C\big)$. 
\item\label{thm:m:W om}\emph{(Quotient)} Let $(E,\om)$ a symplectic vector bundle over $\D$, $W\sub E$ a coisotropic subbundle, and $\Phi\in\T((W_\om,\om_W)|_{S^1})$. Then $m\big(E,\om,W|_{S^1},\Phi\big)=m\big(W_\om,\om_W,W_\om|_{S^1},\Phi\big)$. 
\item\label{thm:m:reg}\emph{(Regular case)} Let $\Si\in\S$ and $\big(E,\om,W,\Phi\big)\in\E_\Si$ be such that $\Phi$ is regular. Then $m_\Si(E,\om,W,\Phi)\in\Z$. Furthermore, if $W$ is orientable then this integer is even.
\item\label{thm:m:wt W}\emph{(Lagrangian embedding)} Let $(E,\om,W,\Phi)\in\E_\D$ and $(V',\om')$ a symplectic vector space. Assume that there exists a surjective homomorphism $\Psi:W\to S^1\x V'$ such that $\Psi^*\om'=\om$ and the following holds. Denoting by $\Psi_W:W_\om\to S^1\x V'$ the map induced by $\Psi$, we have $(\Psi_W)_{z(1)}\Phi([z])=(\Psi_W)_{z(0)}$, for every $z\in C([0,1],S^1)$. Then the following equality holds. We define $\wt E:=E\oplus (S^1\x V'),\wt\om:=\om\oplus(-\om')$, $\wt W:=\big\{(z,v,\Psi_zv)\,\big|\,(z,v)\in W\big\}\sub \wt E$. Then $m_\D(E,\om,W,\Phi)=m_\D(\wt E,\wt\om,\wt W,0)$.
\end{enui}
\end{thm}
This result is proved in Section \ref{subsec:proof:thm:m} (page \pageref{proof:thm:m}). The trickiest part is the proof of property (\ref{thm:m:homotopy}). It is based on the invariance under homotopy of $m_{\d\Si,\om}:\Coi(\d\Si\x V,\om)\to\R$, where $(V,\om)$ is a symplectic vector space. This follows from Proposition \ref{prop:homotopy}.

\begin{proof}[Proof of Theorem \ref{thm:m M om N}]\setcounter{claim}{0} \label{proof:thm:m M om N} Statements (\ref{thm:m M om N:nat},\ref{thm:m M om N:prod},\ref{thm:m M om N:homotopy},\ref{thm:m M om N:weak homotopy},\ref{thm:m M om N:split},\ref{thm:m M om N:Lag}) follow from Theorem \ref{thm:m}(\ref{thm:m:nat},\ref{thm:m:sum},\ref{thm:m:homotopy},\ref{thm:m:weak homotopy},\ref{thm:m:split},\ref{thm:m:Lag}). 

To prove {\bf statement (\ref{thm:m M om N:reg})}, assume that $N$ is regular. Let $u\in C(\Si,M;N,\om)$. We choose a symplectic vector space $(V,\Om)$ of dimension $\dim M$, and $\Psi\in\Iso\big(\Si\x V,\Om;u^*(TM,\om)\big)$. We define $(W,\Phi):=\Psi\circ u|_{\d\Si}^*(TN,\hol^{N,\om})$. It follows from regularity of $N$ that $\hol^{N,\om}$ is regular. (See \cite{ZiLeafwise}, Lemma 15.) Hence (\ref{thm:m M om N:reg}) follows from Theorem \ref{thm:m}(\ref{thm:m:reg}). 

We prove {\bf assertion (\ref{thm:m M om N:M})}. We define 
\begin{eqnarray}\nn&(E',\om'):=u^*(TM,\om),\,(W',\Phi'):=u|_{\d\Si\wo C}*(TN,\hol^{N,\om})\disj\big(C\x T_{x_0}M,\id_{T_{x_0}M}\big),&\\
\nn&(\wt E,\wt\om):=\wt u^*(TM,\om),\,(\wt W,\wt\Phi):=\wt u|_{\d\Si\wo C}^*(TN,\hol^{N,\om}).&\end{eqnarray}
We have $m_{M,\om,N}(\wt u)=m(\wt E,\wt\om,\wt W,\wt\Phi)$. Furthermore, since $u|_C^*(TN,\hol^{N,\om})=\big(C\x T_{x_0}N,\id_{T_{x_0}N}\big)$, it follows from the definitions that $m_{M,\om,N}(u)=m(E',\om',W',\Phi')$. On the other hand, the map $\wt E\ni([z],v)\mapsto [z,v]\in E/\Phi'|_C$ is an $(\wt\om,\om'/\Phi'|_C)$-isomorphism that is the identity outside the point $C\in \wt\Si:=\Si/C$, and hence carries $(\wt W,\wt\Phi)$ to $(W,\Phi')/C$. Therefore, Theorem \ref{thm:m}(\ref{thm:m:nat},\ref{thm:m:full}) imply that $m(\wt E,\wt\om,\wt W,\wt\Phi)=m\big((E',\om',W',\Phi')/C\big)=m(E',\om',W',\Phi')$. It follows that $m_{M,\om,N}(\wt u)=m_{M,\om,N}(u)$. This proves (\ref{thm:m M om N:M}).

We prove {\bf assertion (\ref{thm:m M om N:N om})}. We have
\begin{eqnarray}\nn m_{M,\om,N}(u)&=&m\big(u^*(TM,\om),u|_{S^1}^*(TN,\hol^{N,\om})\big)\\
\nn&=&m\Big(u^*\big((TN)_\om,\om_{TN}\big),u|_{S^1}^*\big((TN)_\om,\hol^{N,\om}\big)\Big)\\
\nn&=&m\Big(u^*\big(T(N_\om),\om_N\big),S^1\x T_{u(1)}N_\om,\id_{T_{u(1)}N_\om}\Big)\\
\label{eq:2 c 1}&=&2c_1\big({u'}^*(T(N_\om),\om_N)\big)=2c_1^{N_\om,\om_N}(u').
\end{eqnarray}
Here in the second equality we used Theorem \ref{thm:m}(\ref{thm:m:W om}), in the third equality we used Theorem \ref{thm:m}(\ref{thm:m:nat}), and in the forth equality we used Theorem \ref{thm:m}(\ref{thm:m:full}). Assertion (\ref{thm:m M om N:N om}) follows from this. 

We prove {\bf assertion (\ref{thm:m M om N:wt M})}. Let $a\in[\D,S^1;M,N_\om]$. We choose a representative $u\in C(\D,M;N,\om)$ of $a$. The claimed equality follows from Theorem \ref{thm:m}(\ref{thm:m:wt W}) with $(E,\om,W,\Phi):=u^*(TM,\om,TN,\hol^{N,\om})$ and $(V',\om'):=\big(T_{N_{u(1)}}(N_\om),(\om_N)_{N_{u(1)}}\big)$, using the map $\Psi:u|_{S^1}^*TN\to S^1\x V'$ given by $\Psi(z,v):=(z,(\pi_N)_*v)$.

This proves assertion (\ref{thm:m M om N:wt M}) and completes the proof of Theorem \ref{thm:m M om N}.
\end{proof}

For the proof of Theorem \ref{thm:m}(\ref{thm:m:sum}) we need the following.
\begin{rmk}\label{rmk:m C om} Let $X$ be a compact oriented curve, $(E,\om)$ and $(E',\om')$ symplectic vector bundles over $X$, $\Phi,\Psi\in\T(E,\om)$ and $\Phi',\Psi'\in\T(E',\om')$, with $\Psi$ and $\Psi'$ regular. Then $m_{X,\om\oplus\om'}(\Phi\oplus\Phi',\Psi\oplus\Psi')=m_{X,\om}(\Phi,\Psi)+m_{X,\om'}(\Phi',\Psi')$. This follows from Proposition \ref{prop:rho}(\ref{prop:rho:sum}).%
\end{rmk}
For the proof of Theorem \ref{thm:m}(\ref{thm:m:homotopy},\ref{thm:m:weak homotopy}) we need the following. Let $X$ be a closed oriented curve. We denote by $\pi:[0,1]\x X\to X$ the canonical projection. For $s\in[0,1]$ we denote by $\iota_s:\{s\}\x X\to [0,1]\x X$ the inclusion. We define $\X:=\big\{t\mapsto(s,z(t))\,\big|\,s\in[0,1],\,z\in C([0,1],X)\big\}$. Let $(E,\om)$ be a symplectic vector bundle over $X$. 
\begin{lemma}\label{le:m C om cont} Let $W\sub\pi^*E$ be an $\pi^*\om$-coisotropic subbundle, and $\Phi_0\in\T(\pi^*(E,\om))$ a regular transport. The following assertions hold.
\begin{enui}
\item\label{le:m C om cont:Pi} Let $\Phi\in\T\big(\Pi([0,1]\x X),W_\om,\om_W\big)$. Assume that $\Psi\in\T\big(\Pi([0,1]\x X),E,\om\big)$ is such that $\iota_s^*\Psi$ is a lift of $\iota_s^*\Phi$, for every $s\in[0,1]$. Then the map $[0,1]\ni s\mapsto m_{X,\om}\big(\iota_s^*\Psi,\iota_s^*\Phi_0\big)\in\R$ is constant. 
\item\label{le:m C om cont:X} Let $\Phi\in\T(\X,W_\om,\om_W)$. Assume that $\Psi\in\T(\X,E,\om)$ is such that $\iota_s^*\Psi$ is a lift of $\iota_s^*\Phi$, for every $s\in[0,1]$. Then the map $[0,1]\ni s\mapsto m_{X,\om}\big(\iota_s^*\Psi,\iota_s^*\Phi_0\big)\in\R$ is continuous. 
\end{enui}
\end{lemma} 
For the proof of Lemma \ref{le:m C om cont} we need the following.
\begin{prop}\label{prop:homotopy} Let $(V,\om)$ be a symplectic vector space, $X$ a topological manifold, $(W,\Phi)\in\Coi\big(X\x V,\om\big)$, $x\in C([0,1]\x[0,1],X)$ and $\Psi\in C([0,1]\x[0,1],\Aut(\om))$ be such that $x(s,0)=x(s,1)$, $\Psi(s,t)W_{x(s,0)}=W_{x(s,t)}$ and $\Psi(s,t)_{W_{x(s,0)}}=\Phi([x(s,\cdot)|_{[0,t]}])$, for $s,t\in[0,1]$. Then the map $[0,1]\ni s\mapsto m_\om(\Psi(s,\cdot))\in\R$ is constant. 
\end{prop}
For the proof of this result, we need the following.
\begin{prop}\label{prop:rho om Psi}Let $W\sub V$ be a coisotropic subspace and $\Psi\in\Iso(\om)$ be such that $\Psi W=W$. Then 
\begin{equation}\label{eq:rho om Psi pm}\rho_\om(\Psi)=\pm\rho_{\om_W}(\Psi_W).\end{equation}
Furthermore, if $\det(\Psi|_W)>0$ then $\rho_\om(\Psi)=\rho_{\om_W}(\Psi_W)$.
\end{prop}
For the proof of Proposition \ref{prop:rho om Psi} we need the following. 
\begin{lemma}\label{le:rho om Psi} Let $(V,\om)$ be a symplectic vector space, $W,W'\sub V$ Lagrangian subspaces and $\Psi\in\Iso(\om)$. Assume that $W+W'=V$, $\Psi W=W$ and $\Psi W'=W'$. Then $\rho_\om(\Psi)=\pm1$. If also $\det(\Psi|_W)>0$ then $\rho_\om(\Psi)=1$. 
\end{lemma}
Let $V$ and $W$ be real vector spaces and $\Psi\in\Hom(V,W)$. We denote by $\Psi^\C:V^\C\to W^\C$ the complex linear extension. If $V=W$ and this space has dimension $n$ then for every $\lam\in\C$ we denote $E_\Psi^\lam:=\ker((\lam-\Psi^\C)^n)\sub V^\C$. Let now $(V,\om)$ be a symplectic vector space, $\Psi\in\Iso(\om)$ and $\lam\in S^1\wo\{\pm1\}$. We define 
\[m_+(\om,\Psi,\lam):=\max\big\{\dim_\C W\,\big|\,W\sub E_\Psi^\lam\textrm{ complex subspace, }\Im\om(\bar v,v)>0,\,\forall v\in W\big\}.\]
The following remarks are used in the proof of Lemma \ref{le:rho om Psi}.
\begin{rmk}\label{rmk:-om} We have $m_+(-\om,\Psi,\bar\lam)=m_+(\om,\Psi,\lam)$. This follows, since the map $E_\Psi^\lam\to E_\Psi^{\bar\lam}$, $v\mapsto\bar v$, is a real isomorphism. 
\end{rmk}
\begin{rmk} \label{rmk:m +} If $V'$ is another vector space and $\Phi\in\Iso(V',V)$ then $m_+\big(\Phi^*\om,\Phi^{-1}\Psi\Phi,\lam\big)=m_+(\om,\Psi,\lam)$. This follows from the fact $\Phi E_{\Phi^{-1}\Psi\Phi}^\lam=E_\Psi^\lam$. \end{rmk}
We define $\om^*$ to be the symplectic form on $V^*$ defined by $\om^*(\phi,\psi):=\phi(w)$, where $w\in V$ is determined by $\om(w,\cdot)=\psi$. 
\begin{rmk}\label{rmk:om *} The map $\Psi^{-*}:=(\Psi^*)^{-1}$ is $\om^*$-symplectic, and the map $\om_\#:V\to V^*$ defined by $\om_\#v:=\om(v,\cdot)$ satisfies $\om_\#^*(\om^*,\Psi^{-*})=(\om,\Psi)$. 
\end{rmk}
Let now $W$ be a finite dimensional vector space. We define the canonical symplectic form $\om^W$ on $V:=W\oplus W^*$ by $\om^W\big((v,\phi),(v',\phi')\big):=\phi'(v)-\phi(v')$. Furthermore, we denote by $\iota^W:W\to W^{**}$ the canonical isomorphism, and define the map $\Phi_W:V\to V^*$ by $\Phi_W(v,\phi):=(\phi,\iota^Wv)$.
\begin{rmk}\label{rmk:W W *} We have $\Phi_W^*\om^{W^*}=\om^W$. Furthermore, if $\Psi\in\Aut(\om^W)$ is such that $\Psi W=W$ and $\Psi W^*=W^*$ then $\Phi_W^{-1}\Psi^{-*}\Phi_W=\Psi$. This follows from the fact $\Psi|_{W^*}=\Psi|_W^{-*}$. 
\end{rmk}
\begin{rmk}\label{rmk:det} Let $V$ be a real vector space and $\Phi\in\End(V)$ be such that $\det\Phi>0$. For $\lam\in\C$ we denote by $m(\Phi,\lam)\in\N\cup\{0\}$ the algebraic multiplicity over $\C$ of $\lam$ as an eigenvalue of $\Phi$. Then $\sum_{\lam\in(-\infty,0)}m(\Phi,\lam)$ is even. \end{rmk} 
\begin{proof}[Proof of Lemma \ref{le:rho om Psi}]\setcounter{claim}{0} We define the map $\Phi:V=W\oplus W'\to W\oplus W^*$ by $\Phi(w,w'):=(w,-(\om^\#w')|_W)$. Then the tuple $(\wt V,\wt W,\wt W',\wt\om,\wt\Psi):=\big(W\oplus W^*,W,W^*,\om^W,\Phi\Psi\Phi^{-1}\big)$ satisfies $\wt W+\wt W'=\wt V$, $\wt\Psi\wt W=W$, $\wt\Psi\wt W'=\wt W'$ and $\det\Psi|_W=\det\wt\Psi|_{\wt W}$. Hence by (Naturality) for $\rho$, we may assume without loss of generality that $V=W\oplus W^*,W'=W^*$ and $\om=\om^W$. 

Let $\lam\in S^1\wo\{\pm1\}$. Remarks \ref{rmk:W W *} and \ref{rmk:m +} imply that $m_+(\om,\Psi,\lam)=m_+(\om^{W^*},\Psi^{-*},\lam)$. On the other hand, $\om^{W^*}=-(\om^W)^*$, hence by Remarks \ref{rmk:-om}, \ref{rmk:m +} and \ref{rmk:om *}, we obtain $m_+(\om^{W^*},\Psi^{-*},\lam)=m_+(\om^W,\Psi,\bar\lam)=m_+(\om,\Psi,\bar\lam)$. It follows that $m_+(\om,\Psi,\lam)=m_+(\om,\Psi,\bar\lam)$, and therefore by Lemma \ref{le:rho om m -} $\rho_\om(\Psi)=\pm1$. 

Assume now also that $\det(\Psi|_W)>0$. We have $\Psi|_{W^*}=\Psi|_W^{-*}$ and $\det\Psi|_W^{-*}=(\det\Psi|_W)^{-1}>0$. Hence by Remark \ref{rmk:om *}, 
\[\sum_{\lam\in(-\infty,0)}m(\Psi,\lam)=\sum_{\lam\in(-\infty,0)}m(\Psi|_W,\lam)+\sum_{\lam\in(-\infty,0)}m(\Psi|_{W^*},\lam)\in4\Z.\]
It follows now from Remark \ref{rmk:rho - om} that $\rho_\om(\Psi)=1$. This proves Lemma \ref{le:rho om Psi}.
\end{proof}
\begin{proof}[Proof of Proposition \ref{prop:rho om Psi}]\setcounter{claim}{0} Assume first that there exists a coisotropic subspace $W'\sub V$ such that 
\begin{equation}\label{eq:dim W'}\dim W'=\dim W,\quad W+{W'}^\om=V,\quad \Psi W'=W'.\end{equation}
We define $U:=W\cap W'$. Since $\Psi W=W$, we have $\Psi W^\om=W^\om$, and since $\Psi W'=W'$, we have $\Psi {W'}^\om={W'}^\om$. Furthermore, by an elementary argument, we have
\begin{equation}\label{eq:U om W}U^\om=W^\om+{W'}^\om.\end{equation}
It follows that $\Psi U^\om=U^\om$ and hence $\Psi U=U$. 
\begin{claim}\label{claim:iso} The map 
\begin{equation}\label{eq:U W om W}U\to W_\om=W/W^\om,\qquad v\mapsto [v]\end{equation} 
is bijective.
\end{claim}
\begin{proof}[Proof of Claim \ref{claim:iso}] By an elementary argument, we have
\begin{equation}\nn W'\cap W^\om=({W'}^\om+W)^\om=\{0\}.\end{equation}
Here in the second equality we used the facts ${{W'}^\om}^\om=W'$ and ${W'}^\om+W=V$. It follows that 
\begin{equation}\label{eq:U W om}U\cap W^\om=W'\cap W^\om=\{0\}
\end{equation}
\begin{claim}\label{claim:U W om} We have
\begin{equation}\label{eq:U +}U+W^\om=W.
\end{equation}
\end{claim}
\begin{proof}[Proof of Claim \ref{claim:U W om}] By (\ref{eq:U W om}) and the facts $U\sub W$ and $W^\om\sub W$, it suffices to show that 
\begin{equation}\label{eq:dim U}\dim U+\dim W^\om\geq\dim W.\end{equation}
To see this inequality, observe that (\ref{eq:U om W}) implies
\begin{eqnarray}\nn\dim V-\dim U&=&\dim U^\om\\
\nn&\leq&\dim W^\om+\dim {W'}^\om\\
\nn&=&\dim W^\om+\dim V-\dim W'.\end{eqnarray}
Since by (\ref{eq:dim W'}) we have $\dim W=\dim W'$, inequality (\ref{eq:dim U}) follows. This proves Claim \ref{claim:U W om}.
\end{proof}
Claim \ref{claim:iso} follows from (\ref{eq:U W om}) and Claim \ref{claim:U W om}.
\end{proof}
Claim \ref{claim:iso} implies that the map (\ref{eq:U W om W}) is a linear symplectic isomorphism. It follows that $U$ and hence $U^\om$ are symplectic subspaces of $V$. Since they are invariant under $\Psi$, the (Product) property in Proposition \ref{prop:rho} implies that 
\begin{equation}\label{eq:rho om Psi}\rho_\om(\Psi)=\rho_{\om|_U}(\Psi|_U)\rho_{\om|_{U^\om}}(\Psi|_{U^\om}).
\end{equation}
Furthermore, the (Naturality) property in Proposition \ref{prop:rho} implies that 
\begin{equation}\label{eq:rho om Psi U}\rho_{\om|U}(\Psi|_U)=\rho_{\om_W}(\Psi_W).
\end{equation}
Since $W^\om$ and ${W'}^\om$ are complementary Lagrangian subspaces of $U^\om$ that are invariant under $\Psi$, it follows from Lemma \ref{le:rho om Psi} that $\rho_{\om|_{U^\om}}(\Psi|_{U^\om})=\pm1$. Combining this with (\ref{eq:rho om Psi}) and (\ref{eq:rho om Psi U}), equality (\ref{eq:rho om Psi pm}) follows.

Assume now that $\det\Psi|_W>0$. Since $\Psi|_W=\Psi|_U\oplus\Psi|_{W^\om}$ and $\Psi|_U\in\Iso(\om|_U)$, it follows that $\det\Psi|_{W^\om}>0$. Hence Lemma \ref{le:rho om Psi} implies that $\rho_{\om|_{U^\om}}(\Psi|_{U^\om})=1$. Combining this with (\ref{eq:rho om Psi}) and (\ref{eq:rho om Psi U}), we obtain $\rho_\om(\Psi)=1$. 

Consider now the general case, in which we do not assume that a subspace $W'\sub V$ satisfying (\ref{eq:dim W'}) exists. We choose a coisotropic subspace $W'\sub V$ such that $\dim W'=\dim W$ and $W+{W'}^\om=V$, and denote $V_0:=W\cap W'$, $V_1:=W^\om$ and $V_2:={W'}^\om$. As in the proof of Claim (\ref{claim:iso}) it follows that $V$ is the direct sum of the $V_i$'s. We define $P_i:V\to V_i$ to be the linear projection along the subspace $\oplus_{j\neq i}V_j$, and we denote $\Psi_{ij}:=P_i\Psi|_{V_j}$, for $i,j=0,1,2$. We fix $t\in\R$ and define, using the splitting $V=\oplus_{i=0,1,2}V_i$, 
\begin{equation}\label{eq:Psi t}\Psi^t:= \left(\begin{array}{ccc}
\Psi_{00} & 0 & t\Psi_{02}\\
t\Psi_{10} & \Psi_{11} & t^2\Psi_{12} \\
0 & 0 & \Psi_{22}
\end{array}\right):V\to V.
\end{equation}
\begin{claim}\label{claim:Psi Psi 1} We have $\Psi^1=\Psi$. 
\end{claim}
\begin{proof}[Proof of Claim \ref{claim:Psi Psi 1}] Since $\Psi W=W$, we have $\Psi_{20}=0$, and since $\Psi W^\om=W^\om$, we have $\Psi_{01}=\Psi_{21}=0$. Hence $\Psi$ has the form (\ref{eq:Psi t}) with $t=1$. This proves Claim \ref{claim:Psi Psi 1}. 
\end{proof}
\begin{claim}\label{claim:Psi t om} The map $\Psi^t$ is an $\om$-symplectic. 
\end{claim}
\begin{proof}[Proof of Claim \ref{claim:Psi t om}] Since $\Psi$ is symplectic, we have for $v_0\in V_0$, $w_2\in V_2$,
\begin{equation}\label{eq:om v 0 w 2} 0=\om(v_0,w_2)=\om(\Psi v_0,\Psi w_2)=\om\big(\Psi_{00}v_0,\Psi_{02}w_2\big)+\om\big(\Psi_{10}v_0,\Psi_{22}w_2\big).\end{equation}
Furthermore, for $v_2,w_2\in V_2$,
\begin{equation}\label{eq:om v 2 w 2}0=\om(v_2,w_2)=\om\big(\Psi_{02}v_2,\Psi_{02}w_2\big)+\om\big(\Psi_{12}v_2,\Psi_{22}w_2\big)+\om(\Psi_{22}v_2,\Psi_{12}w_2\big).
\end{equation}
Hence, for every $v=v_0+v_1+v_2,$ $w=w_0+w_1+w_2\in V=V_0\oplus V_1\oplus V_2$, 
\begin{eqnarray}\nn \om(\Psi^tv,\Psi^tw)&=&\om\big(\Psi_{00}v_0,\Psi_{00}w_0\big)+\om\big(\Psi_{11}v_1,\Psi_{22}w_2\big)+\om\big(\Psi_{22}v_2,\Psi_{11}w_1\big) + \\
\nn&&t\Big(\om\big(\Psi_{00}v_0,\Psi_{02}w_2\big) + \om\big(\Psi_{02}v_2,\Psi_{00}w_0\big) + \\
\nn&&\phantom{t\Big(}\om\big(\Psi_{10}v_0,\Psi_{22}w_2\big) + \om\big(\Psi_{22}v_2,\Psi_{10}w_0\big)\Big) + \\
\nn&& + t^2\Big(\Psi_{02}v_2,\Psi_{02}w_2\big) + \om\big(\Psi_{12}v_2,\Psi_{22}w_2\big) + \om\big(\Psi_{22}v_2,\Psi_{12}w_2\big)\Big)\\
\nn&=&\om(\Psi^1 v,\Psi^1 w) + (t-1)(0 - 0) + (t^2 - 1)0\\
\nn&=& \om(v,w).
\end{eqnarray}
Here in the second equality we used equalities (\ref{eq:om v 0 w 2}) and (\ref{eq:om v 2 w 2}), and in the last equality we used Claim \ref{claim:Psi Psi 1} and the fact that $\Psi$ is symplectic. This proves Claim \ref{claim:Psi t om}.
\end{proof}
\begin{claim}\label{claim:rho Psi t} We have 
\[\rho_\om(\Psi^t)=\rho_\om(\Psi^0).\]
\end{claim}
\begin{proof}[Proof of Claim \ref{claim:rho Psi t}] We denote by $\si(\Phi)$ the set of eigenvalues of an endomorphism $\Phi$ of any vector space. We define 
\[S:=\big\{\pm\Pi_{\lam\in\si(\Psi^0)\cap S^1}\lam^{m_\lam}\,\big|\,m_\lam\in\{0,\ldots,\dim V\},\textrm{ for }\lam\in\si(\Psi^0)\big\}\sub S^1.\]
The block form (\ref{eq:Psi t}) implies that $\det(\lam\one-\Psi^t)=\det(\lam\one-\Psi^0)$. Hence $\si(\Psi^t)=\si(\Psi^0)$. Therefore, by the formula (\ref{eq:rho om Psi -}) of Lemma \ref{le:rho om m -} we have
\[f(t):=\rho_\om(\Psi^t)\in S. \]
By Proposition \ref{prop:rho} the map $\rho_\om:\Iso(\om)\to S^1$ is continuous, so the same holds for the map $f:\R\to S$. Since the set $S$ is finite, it follows that $f$ is constant. This proves Claim \ref{claim:rho Psi t}.
\end{proof}

Since $W'=V_0\oplus V_2$ and $\Psi^0$ leaves the subspaces $V_i$ invariant, we have $\Psi^0W'=W'$. Therefore, by what we already proved, $\rho_\om(\Psi^0)=\pm\rho_{\om_W}(\Psi^0_W)$. Combining this with Claims \ref{claim:Psi Psi 1} and \ref{claim:rho Psi t} and the fact $\Psi^0_W=\Psi_W$, we get $\rho_\om(\Psi)=\pm\rho_{\om_W}(\Psi_W)$. Similarly, if $\det\Psi|_W>0$ then it follows that $\rho_\om(\Psi)=\rho_{\om_W}(\Psi_W)$. This proves Proposition \ref{prop:rho om Psi}. 
\end{proof}

\begin{proof}[Proof of Proposition \ref{prop:homotopy}]\setcounter{claim}{0} Consider the map $f:[0,1]\x[0,1]\to S^1\sub \C$, $f(s,t):=\Phi([x(s,\cdot)|_{[0,t]})$. Let $s\in[0,1]$. Proposition \ref{prop:rho om Psi} implies that $\rho_\om(\Psi(s,0))=\pm f(s,0)=\pm 1$. Since $\Psi(s,1)W_{x(s,1)}=W_{x(s,1)}$ and $\Psi(s,1)_{W_{x(s,1)}}=\Phi([x(s,\cdot)])$, Proposition \ref{prop:rho om Psi} implies that $\rho_\om(\Psi(s,1))=\pm f(s,1)$. We define $\wt x:[0,1]\to X$ to be the concatenation of the paths $[0,1]\ni t\mapsto x(s(1-t),0)$, $x(0,\cdot)$ and $[0,1]\ni t\mapsto x(st,0)$. Then $\wt x$ is homotopic with fixed endpoints to $x(s,\cdot)$, and therefore $\Phi([x(s,\cdot)])=\Phi([\wt x])=\Phi([x'])\Phi([x(0,\cdot)])\Phi([x'])^{-1}$. By naturality of $\rho$, it follows that $f(s,1)=f(0,1)$, and hence $\rho_\om(\Psi(s,1))=\pm f(0,1)$. Combining this with the equality $\rho_\om(\Psi(s,0))=\pm1$, it follows that the map $[0,1]\ni s\mapsto m(\Psi(s,\cdot))$ is constant. This proves Proposition \ref{prop:homotopy}.
\end{proof}

\begin{proof}[Proof of Lemma \ref{le:m C om cont}]\setcounter{claim}{0} {\bf Statement (\ref{le:m C om cont:Pi})} follows from Proposition \ref{prop:homotopy}, and {\bf statement (\ref{le:m C om cont:X})} follows from an elementary argument. This proves Lemma \ref{le:m C om cont}.
\end{proof}

For the proof of Theorem \ref{thm:m}(\ref{thm:m:split}) we need the following remark. We denote by $\om_0$ the standard symplectic form on $\R^{2n}$, and by $\Sp(2n)=\Aut(\om_0)$ the linear symplectic group. We identify $S^1\iso\R/\Z$. 
\begin{rmk}\label{rmk:m 0} We define $m_0:C(S^1,\Sp(2n))\to\Z$ by $m_0(\Psi):=m_{\om_0}\big([0,1]\ni t\mapsto \Psi(t+\Z)\in\Sp(2n)\big)/2\in\Z$. This map equals the usual Maslov index of $\Psi$, as defined for example axiomatically in the book \cite{MS}. To see this, note that on $\U(n)=\Sp(2n)\cap\O(2n)$, $m_0$ agrees with the map $\wt m_0$ constructed in the proof of Theorem 2.29 in that book. Furthermore, $\Sp(2n)$ deformation retracts onto $\U(n)$ (see Proposition 2.22 in \cite{MS}). Since $m_0$ and $\wt m_0$ are invariant under homotopy, the statement follows. 
\end{rmk}

\begin{lemma}\label{le:Phi Psi} Let $X$ be a topological space, $Y\sub X$, $(E,\om)$ a symplectic vector bundle over $X$, $\Phi:Y\x Y\to\GL(\om)$ a morphism of topological groupoids whose composition with the canonical projection $\GL(\om)\to X\x X$ is the identity, and $(V,\Om)$ a symplectic vector space of dimension $\rank E$. If there is a homeomorphism $f:[0,1]\x Y\to X$ such that $f(0,x)\in Y$, for every $x\in Y$, then there exists $\Psi\in\Iso\big(X\x V,\Om;E,\om\big)$ such that $\Phi_x^{x'}\Psi^x=\Psi^{x'}$, for every $x,x'\in Y$. 
\end{lemma}
\begin{proof}[Proof of Lemma \ref{le:Phi Psi}]\setcounter{claim}{0} Assume without loss of generality that $Y\neq\empty$. We choose a homeomorphism $f:[0,1]\x Y\to X$ as above, a point $x_0\in Y$, and $\Psi_0\in\Iso\big(V,\Om;(E,\om)|_{x_0}\big)$. We denote by $\pr:[0,1]\x Y\to Y$ the canonical projection. By Lemma \ref{le:E Psi} there exists $\wt\Psi\in\Iso\big(\pr^*f^*(E,\om),f^*(E,\om)\big)$ such that $\wt\Psi|_{\{0\}\x Y}=\id$. We define $\Psi:X\x V\to E$ by $\Psi^x:=\wt\Psi^{f^{-1}(x)}\Phi_{x_0}^{\pr\circ f^{-1}(x)}\Psi_0:V\to  E_x$, for $x\in X$. This map has the required properties. This proves Lemma \ref{le:Phi Psi}.
\end{proof}

The next remark will be used in the proof of Theorem \ref{thm:m}(\ref{thm:m:split}). Let $X$ be a closed oriented curve. We denote by $\BAR{X}$ the curve $X$ with the opposite orientation.
\begin{rmk}\label{rmk:m C V om or} Let $(E,\om)$ a symplectic vector bundle over $X$ and $\Phi,\Psi\in\T(E,\om)$, with $\Phi$ regular. Then $m_{\BAR{X},\om}(\Psi,\Phi)=-m_{X,\om}(\Psi,\Phi)$. This follows from directly from the definition.
\end{rmk}

For the proof of Theorem \ref{thm:m}(\ref{thm:m:W om}) we need the following. 
\begin{lemma}\label{le:m W Phi} Let $X$ be a closed oriented curve, $(E,\om)$ a symplectic vector bundle over $X$, $W\sub E$ an $\om$-coisotropic subbundle, and $\Phi,\Psi\in\T(E,\om)$. Assume that $\Phi$ and $\Psi$ leave $W$ invariant, and that $\Phi$ is regular. Then $m_{X,\om}(\Psi,\Phi)=m_{X,\om_W}(\Psi_W,\Phi_W)$. 
\end{lemma}
\begin{proof}[Proof of Lemma \ref{le:m W Phi}]\setcounter{claim}{0} Without loss of generality we may assume that $X$ is connected. We choose $z\in C([0,1],X)$ such that $z(0)=z(1)$ and the map $S^1\iso[0,1]/\{0,1\}\ni [t]\mapsto z(t)\in X$ has degree one. For $s\in[0,1]$ we define $z_s\in C([0,1],X)$ by $z_s(t):=z(st)$. We define $F:[0,1]\to\Aut\om_{z(0)}$ by $F(s):=\Phi([z_s])^{-1}\Psi([z_s])$. Since $F(0)=\id$ and $F$ is continuous. It follows that $\det F(s)>0$, for every $s\in[0,1]$. Hence Proposition \ref{prop:rho om Psi} implies that $\rho_\om(F(s))=\rho_{\om_W}(F(s)_{W_{z(0)}})$, for every $s\in[0,1]$. The statement of Lemma \ref{le:m W Phi} follows.
\end{proof}

For the proof of Theorem \ref{thm:m}(\ref{thm:m:reg}) we need the following. 
\begin{lemma}\label{le:X W} Let $X$ be a closed curve, $(E,\om)$ a vector bundle over $X$, and $\Phi,\Psi\in\T(E,\om)$, with $\Phi$ regular. Assume that there exists a coisotropic subbundle $W\sub E$ that is invariant under $\Psi$, such that $\Psi_W$ is regular. Then $m_{X,\om}(\Psi,\Phi)\in\Z$. Furthermore, if there is an orientable such $W$ then $m_{X,\om}(\Psi,\Phi)\in2\Z$.
\end{lemma}
\begin{proof}[Proof of Lemma \ref{le:X W}]\setcounter{claim}{0} Let $X,E,\om,\Phi,\Psi$ and $W$ be as in the hypothesis. We choose a path $z\in C([0,1],X)$ such that $z(0)=z(1)$ and the map $S^1\iso[0,1]/\{0,1\}\ni[t]\mapsto z(t)\in X$ has degree one. By our regularity assumptions, we have $\Phi([z])=\id$ and $\Psi_W([z])=\id$. Hence by the first assertion of Proposition \ref{prop:rho om Psi}, we have $\rho_\om\big(\Phi([z])^{-1}\Psi([z])\big)=\pm\rho_{\om_W}(\Psi_W([z]))=\pm1\in S^1$. It follows that $m_{X,\om}(\Psi,\Phi)\in\Z$. 

To prove the second assertion, for $s\in[0,1]$ we define $z_s\in C([0,1],X)$ by $z_s(t):=z(st)$. We define $S$ to be the set of all $s\in[0,1]$ such that $\Psi([z_s])$ maps the orientation of $W_{z(0)}$ to the orientation of $W_{z(s)}$. This set is non-empty, since $0\in S$, open and closed. It follows that $S=[0,1]$, and therefore $\det\Psi([z])>0$. Therefore, by the second assertion of Proposition \ref{prop:rho om Psi} we have $\rho_\om\big(\Phi([z])^{-1}\Psi([z])\big)=\rho_{\om_{W_{z(0)}}}(\Psi([z])_{W_{z(0)}})=1\in S^1$. It follows that $m_{X,\om}(\Psi,\Phi)\in2\Z$. 

This proves Lemma \ref{le:X W}.

\end{proof}

For the proof of Theorem \ref{thm:m}(\ref{thm:m:wt W}) we need the following lemma. Let $X$ be a closed curve, $(E,\om)$ a symplectic vector bundle over $X$, $(V',\om')$ a symplectic vector space, $W\sub E$ an $\om$-coisotropic subbundle, $\Phi_0\in\T(E,\om)$ a regular transport, and $F:W\to X\x V'$ a surjective homomorphism such that $F^*\om'=\om$. We denote by $F_W:W_\om\to E'$ the map induced by $F$. We define $\Phi\in\T(W_\om,\om_W)$ by $\Phi([z]):=(F_W)_{z(1)}^{-1}(F_W)_{z(0)}$, $\wt E:=E\oplus (S^1\x V')$, $\wt\om:=\om\oplus(-\om')$, and $\wt W:=\big\{(z,v,F v)\,\big|\,z\in C,\,v\in W\big\}$. Then $\wt W$ is an $\wt\om$-Lagrangian subbundle of $\wt E$. Furthermore, we define $\Phi'_0\in\T\big(X\x V',\om'\big)$ to be the trivial transport $\Phi'_0\const \id$, and $\wt\Phi_0:=\Phi_0\oplus\Phi'_0\in\T(\wt E,\wt\om)$. Let $\Psi\in\T(E,\om)$ be a lift of $(W,\Phi)$. We define $\wt\Psi:=\Psi\oplus\Phi'_0\in\T(\wt E,\wt\om)$. Then $\wt\Psi$ is a lift of $(\wt W,0)$. 
\begin{lemma}\label{le:wt W} We have $m_{X,\wt\om}(\wt\Psi,\wt\Phi_0)=m_{X,\om}(\Psi,\Phi_0)$. 
\end{lemma}
\begin{proof}[Proof of Lemma \ref{le:wt W}]\setcounter{claim}{0} This follows from a straight-forward argument. 
\end{proof}

\subsection{Proof of Theorem \ref{thm:m} (Properties of the coisotropic Maslov map for bundles)}\label{subsec:proof:thm:m}
\begin{proof}[Proof of Theorem \ref{thm:m}]\setcounter{claim}{0}\label{proof:thm:m} {\bf Assertion (\ref{thm:m:nat})} follows directly from the definitions and {\bf assertion (\ref{thm:m:sum})} from Remark \ref{rmk:m C om}. {\bf Assertions (\ref{thm:m:homotopy},\ref{thm:m:weak homotopy})} follow from Lemma \ref{le:m C om cont} and Theorem \ref{thm:m X V om}(\ref{thm:m X V om:Pi},\ref{thm:m X V om:X}). 

To prove {\bf statement (\ref{thm:m:split})}, let $\Si,X,E,\om,W,\Phi,W'$ and $\Phi'$ be as in the hypothesis. Without loss of generality, we may assume that $\Si$ is connected and $X\neq\empty$. We choose a symplectic vector space $(V,\Om)$ of dimension $\rank E$. Assume first that $\d\Si\neq\empty$. We choose $\Psi\in\Iso\big(\Si\x V,\Om;E,\om\big)$. We define $\wt\Psi:=(\pr_X^\Si)^*\Psi\in\Iso\big(\Si_X\x V,\Om;(\pr_X^\Si)^*(E,\om)\big)$. 
\begin{claim}\label{claim:m C Si} We have $m_{X^\Si,\Om}\big(\wt\Psi^*{\pr_X^\Si}^*(W',\Phi')\big)=0$.
\end{claim}
\begin{proof}[Proof of Claim \ref{claim:m C Si}] Without loss of generality we may assume that $X$ is connected. We denote by $X_1$ and $X_2$ the two connected components of ${\pr_X^\Si}^{-1}(X)\sub\Si_X$. We denote by $f:X_1\to X_2$ the unique map such that $\pr_X^{\d\Si}|_{X_1}=\pr_X^{\d\Si}\circ f$. Furthermore, for $i=1,2$ we define $(W_i,\Phi_i):=\wt\Psi^*{\pr_X^\Si}^*(W',\Phi')|_{X_i}$. Then $f^*(W_2,\Phi_2)=(W_1,\Phi_1)$. Furthermore, the canonical orientation of $X_1$ (induced by the orientation of $\Si_X$) is opposite to the pullback under $f$ of the canonical orientation of $X_2$. Therefore, Claim \ref{claim:m C Si} follows from Remark \ref{rmk:m C V om or} and statement (\ref{thm:m:nat}). 
\end{proof}
Using Claim \ref{claim:m C Si}, it follows that
\begin{eqnarray}\nn &&m\big((\pr_X^\Si)^*\big(E,\om,(W,\Phi)\disj(W',\Phi')\big)\big)\\
\nn&&=m_{\d\Si_X,\Om}\big(\wt\Psi^*{\pr_X^\Si}^*\big((W,\Phi)\disj(W',\Phi')\big)\big)\\
\nn &&=m_{\d\Si,\Om}(\Psi^*(W,\Phi)).
\end{eqnarray}
Since $m_\Om(\Psi^*(W,\Phi))=m_\Si(E,\om,W,\Phi)$, equality (\ref{eq:m Si E}) follows. 

Assume now that $\d\Si=\empty$. We choose a connected closed curve $X'\sub\Si\wo X$, such that $\Si_{X'}$ is disconnected. We denote by $X_1$ and $X_2$ the connected components of $\Si_{X'}$, and by $f:X_1\to X_2$ the canonical map. We also choose a trivialization $\wt\Psi\in\Iso\big(\Si_{X'}\x \R^{2n},\om_0;E,\om\big)$. We define $\Psi:X_1\to\Aut(\om_0)$ by $\Psi(z):=\wt\Psi_z^{-1}\wt\Psi_{f(z)}$. Furthermore, we define $m_0$ as in Remark \ref{rmk:m 0}. As explained in the proof of Theorem 2.69 in \cite{MS}, we have $c_1(E,\om)=m_0(\Psi)$. (That theorem is stated for smooth surfaces, however, the proof carries over to topological surfaces.) We define $\Phi''\in\T((E,\om)|_{X'})$ to be the unique transport such that ${\pr_{X'}^\Si}^*\Phi''([z])=\wt\Psi_{z(1)}\wt\Psi_{z(0)}^{-1}$, for $z\in C([0,1],X_1)$. It follows that 
\begin{equation}\label{eq:wt Psi pr}\wt\Psi^*{\pr_{X'}^\Si}^*\Phi''([z])=\left\{\begin{array}{ll}\id_{\R^{2n}},&\textrm{for }z\in C([0,1],X_1),\\
\Psi(f^{-1}(z(1)))^{-1}\Psi(f^{-1}(z(0))),&\textrm{for }z\in C([0,1],X_2).
\end{array}\right.\end{equation}
We choose $z\in C([0,1],X_2)$ such that the map $z(0)=z(1)$ and the map $S^1\iso\R/\Z\ni t+\Z\mapsto z(t)\in X_2$ is an orientation reversing homeomorphism (with respect to the orientation on $X_2$ induced by the orientation of $\Si_{X'}$). Equality (\ref{eq:wt Psi pr}) implies that $m_{{\pr_{X'}^\Si}^{-1}(X'),\om_0}\big(\wt\Psi^*{\pr_{X'}^\Si}^*\Phi''|_{X_1}\big)=0$ and 
\begin{eqnarray}\nn m_{{\pr_{X'}^\Si}^{-1}(X'),\om_0}\big(\wt\Psi^*{\pr_{X'}^\Si}^*\Phi''|_{X_2}\big)&=&-2m_0\big((\Psi\circ f^{-1}\circ z)^{-1}\Psi(f^{-1}(z(0)))\big)\\
\nn&=&-2m_0\big((\Psi\circ f^{-1}\circ z)^{-1}\big)\\
\nn&=&2m_0(\Psi).\end{eqnarray}
It follows that 
\begin{equation}\label{eq:m pr C'}m\big({\pr_{X'}^\Si}^*\big(E,\om,E|_{X'},\Phi''\big)\big)=m_{{\pr_{X'}^\Si}^{-1}(X'),\om_0}\big(\wt\Psi^*{\pr_{X'}^\Si}^*\Phi''\big)=2m_0(\Psi).\end{equation}
On the other hand, using the canonical identifications $(\Si_{X'})_X\iso\Si_{X\disj X'}=(\Si_X)_{X'}$, by what we already proved, we have
\begin{eqnarray}\nn m\big({\pr_{X'}^\Si}^*\big(E,\om,E|_{X'},\Phi''\big)\big)&=&m\big({\pr_{X\disj X'}^\Si}^*\big(E,\om,(W',\Phi')\disj(E|_{X'},\Phi'')\big)\big)\\
\nn&=&m\big({\pr_X^\Si}^*(E,\om,W',\Phi')\big).
\end{eqnarray}
Combining this with (\ref{eq:m pr C'}) and the fact $m_0(\Psi)=c_1(E,\om)$, equality (\ref{eq:m Si E}) follows. This proves statement (\ref{thm:m:split}).

We prove {\bf (\ref{thm:m:Lag})}. For each natural number $n$ we denote by $\om_0$ and $J_0$ the standard symplectic form and complex structure on $\R^{2n}$, and by $R(n)$ the set of all totally real subspaces of $\R^{2n}$, and for $W_0\in R(n)$ we define $\rho_{W_0}:R(n)\to S^1$ by $\rho_{W_0}(W):=\det(\Psi)^2/|\det(\Psi)|^2$, where $\Psi:\C^n=\R^{2n}\to\C^n$ is a complex linear map such that $\Psi W_0=W$. For a closed oriented topological curve $X$ and a map $W\in C(X,R(n))$ we define $m^\real(W):=\deg\big(\d\Si\ni z\mapsto \rho_{W_0}(W(z))\in S^1\big)$, where $W_0\in R(n)$ is arbitrary. Let $\Si\in\S$. We define $\E^\real_\Si$ and $m^\real_\Si:\E^\real_\Si\to\Z$ as in the paragraph before (\ref{eq:m L Si}) in the appendix. Let $(E,\om,W)\in\E^L_\Si$. We denote $2n:=\rank E$. We choose $\Psi\in\Iso\big(\Si\x \R^{2n},\om_0;E,\om\big)$, and define $W':\d\Si\to R(n)$ by $W'(z):=\Psi|_{\d\Si}^{-1}W_z$ and view this also as a subbundle of $\d\Si\x\R^{2n}$. By \cite{MS}, the (Isomorphism) condition in Theorem C.3.5. and the (Trivial bundle) property in Theorem C.3.6., we have 
\[m^L_\Si(E,\om,W)=m^\real_\Si(E,\Psi_*J_0,W)=m_\Si\big(\Si\x\R^{2n},J_0,W'\big)=m(W').\]
On the other hand, $m_\Si(E,\om,W,0)=m_\Si(\Si\x\R^{2n},\om_0,W',0)=m_{\d\Si,\om_0}(W',0)$. Therefore, (\ref{thm:m:Lag}) is a consequence of the following claim. 
\begin{claim}\label{claim:wt m n} We have $m(W')=m_{\d\Si,\om_0}(W',0)$. 
\end{claim}
\begin{proof}[Proof of Claim \ref{claim:wt m n}] We fix a connected component $X$ of $\d\Si$. We choose a path $z\in C([0,1],X)$ such that the map $S^1\iso\R/\Z\ni t+\Z\mapsto z(t)\in X$ has degree one. We denote by $\U(n)\sub\C^{n\x n}$ and $\O(n)\sub\R^{n\x n}$ the unitary and orthogonal groups. Note that $\Gr(\om_0,n)$ is the Grassmannian of Lagrangian subspaces of $\R^{2n}$. For every $W_0\in\Gr(\om_0,n)$ the map $\U(n)/\O(n)\ni\Psi\O(n)\mapsto \Psi W_0\in\Gr(\om_0,n)$ is a well-defined diffeomorphism. Since the map $\U(n)\to\U(n)/\O(n)$ is a smooth fiber bundle, setting $W_0:=W'_{z(0)}$, it follows that there exists a path $\wt\Psi\in C([0,1],\U(n))$ such that $\wt\Psi(0)=\id$ and $\wt\Psi(t)W'_{z(0)}=W'_{z(t)}$, for every $t\in[0,1]$. By definition, we have $m_{\d\Si,\om_0}(W',0)=m(\wt\Psi)=2\al\big([0,1]\ni t\mapsto \rho_{\om_0}(\wt\Psi(t))\big)$. Let $t\in[0,1]$. Since $\Psi(t)\in\U(n)$, by the Determinant property of $\rho_{\om_0}$ we have $\rho_{\om_0}(\Psi(t)))^2=\det(\Psi(t))^2=\rho_{W'_{z(0)}}(W'_{z(t)})$. Claim \ref{claim:wt m n} follows.
\end{proof}

We prove {\bf (\ref{thm:m:full})}. Let $\Si,E,\om,W,\Phi$ and $X$ be as in the hypothesis. Without loss of generality we may assume that $\Si$ is connected. Assume first also that $\Si$ is homeomorphic to $[0,1]\x X$. We choose a symplectic vector space $(V,\Om)$ of dimension $\rank E$. Since $\Phi$ is regular, we may apply Lemma \ref{le:Phi Psi} with $(X,Y):=(\Si,X)$ and $\Phi$ replaced by the map $X\x X\ni(z_0,z_1)\mapsto\Phi([z])$, where $z\in C([0,1],X)$ is an arbitrary path satisfying $z(i)=z_i$, for $i=0,1$. It follows that there exists $\Psi\in \Iso\big(\Si\x V,\Om;E,\om\big)$ such that $\Psi^{z(1)}=\Phi([z])\Psi^{z(0)}$, for every $z\in C([0,1],X)$. We define $\Psi':\Si/X\x V\to E/\Phi$ by $\Psi'_{[z]}v:=[z,\Psi_zv]$. It follows that
\begin{eqnarray*}m(E,\om,W,\Phi)&=&m_\Om(\Psi^*(W,\Phi))\\
&=&m_\Om\big({\Psi'}^*(W,\Phi)|_{\d\Si\wo X}\big)\\
&=&m\big((E,\om,W,\Phi)/X\big). 
\end{eqnarray*}
Here in the second step we used the fact that $\Psi^*\Phi|_C([z])=\id:V\to V$, for every $z\in C([0,1],X)$. This proves the statement if $\Si$ is homeomorphic to $[0,1]\x X$.

In the general case we choose a curve $X'\sub \Si$ such that $\Si_{X'}$ is the disjoint union of two surfaces $\Si_0$ and $\Si_1$, such that $\Si_1$ is homeomorphic to $[0,1]\x X$. By statement (\ref{thm:m:split}) we have 
\begin{eqnarray}\nn m(E,\om,W,\Phi)&=&m\big((E,\om,W,\Phi)^{\Phi'}\big)\\
\label{eq:m Phi Phi'}&=&m\big((E,\om,W,\Phi)^{\Phi'}|_{\Si_0}\big)+m\big((E,\om,W,\Phi)^{\Phi'}|_{\Si_1}\big). 
\end{eqnarray} 
By what we already proved, we have 
\begin{equation}\label{eq:m / Phi}m\big((E,\om,W,\Phi)^{\Phi'}|_{\Si_1}\big)=m\big((E,\om,W,\Phi)^{\Phi'}|_{\Si_1}/X\big).\end{equation} 
Using again statement (\ref{thm:m:split}), we have $m\big((E,\om,W,\Phi)^{\Phi'}|_{\Si_0}\big)+m\big((E,\om,W,\Phi)^{\Phi'}|_{\Si_1}/X\big)=m\big((E,\om,W,\Phi)/X\big)$. Combining this with (\ref{eq:m Phi Phi'},\ref{eq:m / Phi}), statement (\ref{thm:m:full}) follows. 

We prove {\bf assertion (\ref{thm:m:W om})}. We choose a symplectic vector space $(V,\Om)$ of dimension $\rank E$ and a coisotropic subspace $W^0\sub V$ of dimension $\rank W$. 
\begin{claim}\label{claim:Psi W} There exists $\Psi\in\Iso\big(\D\x V,\Om;E,\om\big)$ such that $\Psi(\D\x W^0)=W$. 
\end{claim}
\begin{proof}[Proof of Claim \ref{claim:Psi W}] We choose an arbitrary $\wt\Psi\in\Iso\big(\D\x V,\Om;E,\om\big)$ and define $\wt W:=\wt\Psi^{-1}W\sub\D\x V$. We choose a map $f_0\in\Aut(\Om)$ such that $f_0W^0=\wt W_0$. It follows from Lemma \ref{le:bundle} and the homotopy lifting property for smooth fiber bundles that there exists $f\in C(\D,\Aut(\Om))$ such that $f(z)W^0=\wt W_z$, for every $z\in\D$. We define $\Psi:=\wt\Psi f$. Claim \ref{claim:Psi W} follows. 
\end{proof}
We choose $\Psi$ as in Claim \ref{claim:Psi W}. The assertion (\ref{thm:m:W om}) follows from Lemma \ref{le:m W Phi} with $\om,W,\Phi$ replaced by $\Om,W^0,\Psi^*\Phi$. {\bf Assertion (\ref{thm:m:reg})} follows from Lemma \ref{le:X W}.

We prove {\bf assertion (\ref{thm:m:wt W})}. We choose a symplectic vector space $(V,\Om)$ of dimension $\rank E$ and  $f\in\Iso\big(\D\x V,\Om;E,\om\big)$. Then the hypotheses of Lemma \ref{le:wt W} are satisfied with $X:=S^1$ and $\om,W,\Phi$ replaced by $\Om,f|_{S^1}^*(W,\Phi)$. We denote $\widehat{W}:=\big\{\big(z,fv,\Psi fv\big)\,\big|\,(z,v)\in W\big\}$. By the conclusion of Lemma \ref{le:wt W} we have $m_{S^1,\Om}\big(f|_{S^1}^*(W,\Phi)\big)=m_{S^1,\Om\oplus\om'}(\widehat{W},0)$. Combining this with Lemma \ref{le:m Phi Psi T}, it follows that $m(E,\om,W,\Phi)=m_{S^1,\om}(W,\Phi)=m_{S^1,\om\oplus\om'}(\wt W,0)$. This proves assertion (\ref{thm:m:wt W}) and completes the proof of Theorem \ref{thm:m}.
\end{proof}

\appendix

\section{Auxiliary results}\label{sec:aux}
The following result was used in Section \ref{sec:main}. 
\begin{lemma}\label{le:al} The winding map $\al:C([0,1],\R/\Z)\to\R$ is continuous.
\end{lemma}
\begin{proof}[Proof of Lemma \ref{le:al}]\setcounter{claim}{0} We denote by $d$ the standard metric on $\R/\Z$. By Lemma \ref{le:compact open}(\ref{le:compact open:d}) $C([0,1],\R/\Z)$ is metrized by the metric $d'$ defined as in (\ref{eq:d'}). Let $z_0\in C([0,1],\R/\Z)$. We denote by $\pi:\R\to\R/\Z$ the canonical projection. We choose a path $\wt z_0\in C([0,1],\R)$ such that $\pi\circ\wt z_0=z_0$. We define the map $\phi:B^{d'}_{1/2}(z_0)\to C([0,1],\R)$ by defining $\phi(z)(t)$ to be the unique point in $\big(\wt z_0(t)-1/2,\wt z_0(t)+1/2\big)$ such that $\pi(\phi(z)(t))=z(t)$. This map is continuous. Furthermore, by Lemma \ref{le:compact open}(\ref{le:compact open:ev}) the map $C([0,1],\R)\to\R$ given by $\wt z\mapsto \wt z(1)-\wt z(0)$ is continuous. Since $\al|_{B^{d'}_{1/2}(z_0)}$ is the composition of $\phi$ with this map, it is continuous. It follows that $\al$ is continuous. This proves Lemma \ref{le:al}. 
\end{proof}
The next result was used in Section \ref{sec:main} for the definition of the map $m_{C,\om}:\T\big(C\x V,\om\big)\to\R$. 
\begin{prop}\label{prop:rho}[D. A. Salamon and E. Zehnder, Theorem 3.1. in \cite{SZ}] There is a unique collection of continuous mappings $\rho_\om:\Iso(\om)\to S^1$ (one for every symplectic vector space $(V,\om)$) satisfying the following conditions:
\begin{enui}\item\label{prop:rho:nat}(Naturality:) If $(V,\om)$ and $(V',\om')$ are symplectic vector spaces, $\Phi\in\Iso(\om,\om')$ and $\Psi\in\Iso(\om)$ then $\rho_{\om'}(\Phi\Psi\Phi^{-1})=\rho_\om(\Psi)$.
\item\label{prop:rho:sum} (Direct sum:) If $(V,\om)$ and $(V',\om')$ are symplectic vector spaces and $\Phi\in\Iso(\om)$ and $\Phi'\in\Iso(\om')$ then $\rho_{\om\oplus\om'}(\Phi\oplus\Phi')=\rho_\om(\Phi)\rho_{\om'}(\Phi')$. 
\item\label{prop:rho:det} (Determinant:) If $\Phi\in\Sp(2n)\cap\OO(2n)$ then 
$\rho_{\om_0}(\Phi)=\det(X+iY)$, where $X,Y\in\R^{n\x n}$ are such that 
\begin{equation}\nn\label{eq:Phi}\Phi=\left(\begin{array}{cc}X & -Y\\
Y & X
\end{array}\right).\end{equation}
\item\label{prop:rho:norm} (Normalization:) If $\Phi\in\Iso(\om)$ has no eigenvalue on the unit circle then $\rho_\om(\Phi)=\pm1$. 
\end{enui}
\end{prop}
The maps $\rho_\om$ in the collection of this proposition are called Salamon-Zehnder maps. 

The next lemma was used in the proof of Proposition \ref{prop:rho om Psi}. We fix $\lam\in\C$ and denote 
\[E^\Psi_\lam:=\ker\Big(\big(\lam\,\id-\Psi\big)^{\dim V}:V\otimes\C\to V\otimes\C\Big).\] 
(If $\lam$ is an eigenvalue of $\Psi$ then this is the generalized eigenspace of $\lam$, otherwise it is $\{0\}$.) We fix $\lam\in S^1\wo\{\pm1\}$ and define We define 
\[m_+(\om,\Psi,\lam):=\max\big\{\dim_\C W\,\big|\,W\sub E^\Psi_\lam\textrm{ complex subspace:}\,i\om(\bar v,v)<0,\forall 0\neq v\in W\big\},\]
and we denote $m_-(\Psi):=\frac12\sum_{\lam\in(-\infty,0)}\dim_\C E^\Psi_\lam$.
\begin{lemma}\label{le:rho om m -} The number $m_-(\Psi)$ is an integer, and 
\begin{equation}\label{eq:rho om Psi -}\rho_{\om}(\Psi)=(-1)^{m_-(\Psi)}\prod_{\lam \in S^1\wo\{\pm1\}}\lam^{m_+(\om,\Psi,\lam)}.\end{equation}
\end{lemma}
\begin{proof}[Proof of Lemma \ref{le:rho om m -}]\setcounter{claim}{0} This follows from the proof of Theorem 3.1. in \cite{SZ}.
\end{proof}
\begin{rmk}\label{rmk:rho - om} By Remark \ref{rmk:-om} we have 
\begin{equation}\label{eq:rho - om}\rho_{-\om}(\Psi)=\BAR{\rho_\om(\Psi)}.\end{equation}
\end{rmk}
The following lemma was used in Section \ref{sec:main}.
\begin{lemma}\label{le:m om mu} Let $(M,\om,G,\mu)$ be a Hamiltonian $G$-manifold, and let $\Si\in\S$.  Assume that the action of $G$ on $N:=\mu^{-1}(0)$ is free. Then $m_{\Si,\om,\mu}=m_{\Si,\om,N}$. 
\end{lemma}

\begin{proof}[Proof of Lemma \ref{le:m om mu}]\setcounter{claim}{0} Let $a\in\big[\D,M;N,\om\big]$. We choose a representative $u\in C(\D,M)$ of $a$, and define $g\in C(S^1,G)$ to be the unique map satisfying $u(z)=g(z)u(1)$, for every $z\in S^1$. We choose a continuous symplectic trivialization $\Psi:\D\x\R^{2n}\to u^*TM$. We define $z\in C([0,1],S^1)$ By $z(t):=e^{2\pi it}$ and $\wt\Psi:[0,1]\to\Sp(2n)$ by $\wt\Psi(t):=\Psi_{z(t)}^{-1}g(z(t))\Psi_1$. It follows that $m_\GS(u)=m(\wt\Psi)$. We define the coisotropic subbundle $W\sub \R^{2n}$ by $W_z:=\Psi_z^{-1}T_{u(z)}N$, and $\Phi:=\Psi|_{S^1}^*\hol^{N,\om}$.  We have $\wt\Psi_{z(t)}W_1=W_{z(t)}$, for every $t\in[0,1]$. Lemma \ref{le:m om mu} follows now from the following claim.
\begin{claim}\label{claim:Phi} We have $\wt\Psi(t)_{W_1}=\Phi([z|_{[0,t]}])$, for every $t\in[0,1]$. 
\end{claim}
\begin{pf}[Proof of Claim \ref{claim:Phi}] We choose a smooth map $f:TN_\om\to N$ such that $f(0)=u(1)$ and $\pr_Ndf(0)=\id$. We define $\wt f:[0,1]\x TN_\om\to N$ by $\wt f(t,v):=g(z(t))f(v)$. It follows that $\hol^{N,\om}([u\circ z|_{[0,t]}])=\pr_Nd\wt f(t,\cdot)(0)=(g(t)\cdot)_{T_{u(1)}}$, for every $t\in[0,1]$. This implies that $\Phi([z|_{[0,t]}])=(\Psi_{z(t)}^{-1})_{T_{u\circ z(t)}N}(g(t)\cdot)_{T_{u(1)}}(\Psi_1)_{W_1}=\wt\Psi(t)_{W_1}$. This proves Claim \ref{claim:Phi}.
\end{pf}
\end{proof}
The following result was used in Section \ref{sec:main}. Let $(M,\om)$ be a closed connected symplectic manifold. Assume that there exists $a\in\R$ such that $[\om]=2ac_1(TM,\om)$ on $[S^2,M]$, and that $\T^k$ acts on $M$ with moment map $\mu$. The mixed action-Maslov index is a homomorphism $I:\pi_1(\Ham(M,\om))\to\R$, where homotopy is taken with respect to the $C^\infty$-topology on $\Ham(M,\om)$. It is defined as follows (see \cite{EP}). Let $A\in\pi_1(Ham(M,\om))$. We choose a representative $\phi\in C^\infty(S^1,\Ham(M,\om))$ of $A$. By Floer-theory there exists $u\in C^\infty(\D,M)$ such that $\phi_z\circ u(1)=u(z)$, for every $z\in S^1$. We define $m(u,\phi)\in\Z$ as follows. We choose a symplectic vector space $(V,\Om)$ of dimension $\dim M$, and a trivialization $\Psi\in\Iso\big(\D\x V,\Om;u^*(TM,\om)\big)$. We define $m(u,\phi):=m\big(S^1\ni z\mapsto \Psi_z^{-1}d\phi_z(u(1))\Psi_1\in\Aut(\Om)\big)$. By a standard homotopy argument this number does not depend on the choices of $V,\Om$ and $\Psi$. We define $I(A):=\int_\D u^*\om-\int_{S^1}F\circ u d\th-a m(u,\phi)$, where $\th\in S^1$ denotes the angular coordinate, and $F\in C^\infty([0,1]\x M,\R)$ is the unique map whose flow is $\phi$ and that satisfies $\int_M F(t,\cdot)\om^{\wedge n}=0$, for every $t\in[0,1]$. This number does not depend on the choice of $\phi$ and $u$ (see \cite{EP}). 

The exponential map $\exp:\tt\to\T$ induces an isomorphism $\Hom(\pi_1(\T),\R)\iso\tt^*$. The action of $\T$ on $M$ induces a homomorphism $\Hom(\pi_1(\Ham(M,\om)),\R)\to\Hom(\pi_1(\T),\R)$. We define $\xi_\spec\in\tt^*$ to be the image of $-I$ under the composition of these two maps, and the special fiber $N:=\mu^{-1}(\xi_\spec\sub M$.
\begin{lemma}\label{le:special} If $\TT$ acts freely on $N$ then for every $u\in C^\infty(\D,M;N,\om)$ we have $\int_\D u^*\om=am_{\D,N,\om}([u])$. 
\end{lemma}

\begin{proof}[Proof of Lemma \ref{le:special}]\setcounter{claim}{0} Let $a\in\big\lan\D,M;N,\om\big\ran$. We choose a representative $\wt u\in C(\D,M)$ of $a$. We define $\wt g\in C(S^1,G)$ by $\wt u(z)=\wt g(z)\wt u(1)$. Let $\xi\in\Ga\sub\tt$. By definition we have 
\begin{equation}\label{eq:lan mu xi ran}I(\lan\mu,\xi\ran)=-\A_{\lan\mu,\xi\ran}(u)+\frac a2m(u,\phi_{\lan\mu,\xi\ran}).\end{equation}
Furthermore, $\int_{S^1}\lan\mu\circ u(e^{2\pi i t}),\xi\ran dt=\lan\mu\circ u(1),\xi\ran=-I\circ \phi_{\#}([S^1\iso\R/\Z\ni t+\Z\mapsto \exp(t\xi)])=-I(\phi_{\lan\mu,\xi\ran})$, since $u(1)\in\mu^{-1}(p_{\spec})$. 
Combining this with (\ref{eq:lan mu xi ran}), we obtain $\int_\D u^*\om=a m_{\D,\om,\mu}([u])=a m_{\D,N,\om}([u])$. Here in the last step we used Lemma \ref{le:m om mu}. This proves Lemma \ref{le:special}.
\end{proof}
In order to define the collection of maps $m^L_\Si$ occuring in Theorem \ref{thm:m}, for $\Si\in\S$ we define $\E^\real_\Si$ to be the class of all triples $(E,J,W)$, where $(E,J)$ is a complex vector bundle over $\Si$ and $W\sub E|_{\d\Si}$ is a totally real subbundle. By Theorem C.3.5. in the book \cite{MS} by D. McDuff and D. A. Salamon there exists a unique collection of maps $m^\real_\Si:\E^\real_\Si\to \Z$, where $\Si\in\S$, satisfying suitable (Isomorphism), (Direct sum), (Composition) and (Normalization) conditions. Let $\Si\in\S$. We define the map $m^L_\Si:\E^L_\Si\to\Z$ as follows. Let $(E,\om,W)\in\E^L$. We choose a fiberwise complex structure $J$ on $E$ that is $\om$-compatible, and we define 
\begin{equation}\label{eq:m L Si}m^L_\Si(E,\om,W):=m^\real_\Si(E,J,W)
\end{equation} 

This number is well-defined, i.e.~it does not depend on the choice of $J$. 

In the following we define the linear holonomy along a leaf in a foliation. This was used in Section \ref{sec:main}, in order to define the Maslov map. Let $M$ be a manifold and $\F$ a foliation on $M$, i.e. a maximal atlas of foliation charts. We denote by $T\F\sub TM$ and $N\F:=TM/T\F$ the tangent and normal bundles of $\F$, by $\pr^\F\Colon TM\to N\F$ the canonical projection, by $\F_x\sub M$ the leaf through a point $x\in M$, and by $R^\F:=\big\{(x,y)\in M\x M\,\big|\,y\in\F_x\big\}$ the leaf relation. For $x\in M$ we write $T_x\F:=(T\F)_x$ and $N_x\F:=(N\F)_x$. Let $F$ be a leaf of $\F$, $a\leq b$, and $x\in C([a,b],F)$. The \emph{linear holonomy of $\F$ along $x$} is the linear map $\hol^\F_x:N_{x(a)}\F\to N_{x(b)}\F$, whose definition is based on the following result. 
\begin{prop}\label{prop:hol} Let $M,\F,F,a,b$ and $x$ be as above, $N$ a manifold, and $y_0\in N$. Then the following statements hold.
\begin{enui}\item\label{prop:hol:u} For every linear map $T\Colon T_{y_0}N\to T_{x(a)}M$ there exists a map $u\in C([a,b]\x N,M)$ such that
\begin{eqnarray}\label{eq:u y 0}&u(\cdot,y_0)=x,&\\
 \label{eq:u t y}&u(t,y)\in\F_{u(a,y)},\,\forall t\in[a,b],\,y\in N,&\\
\label{eq:u t}&u(t,\cdot)\textrm{ is differentiable at }y_0,\,\forall t\in[a,b],&\\
\label{eq:d u a}&d(u(a,\cdot))(y_0)=T.& 
\end{eqnarray}
\item\label{prop:hol:u u'} Let $u,u'\in C([a,b]\x N,M)$ be maps satisfying (\ref{eq:u y 0},\ref{eq:u t y},\ref{eq:u t}), such that
\begin{equation}\label{eq:u u' a}\pr^\F d(u(a,\cdot))(y_0)=\pr^\F d(u'(a,\cdot))(y_0).\end{equation}
Then $\pr^\F d(u(b,\cdot))(y_0)=\pr^\F d(u'(b,\cdot))(y_0)$. 
\end{enui}
\end{prop}
For the proof of Proposition \ref{prop:hol} we need the following lemmas.
\begin{lemma}\label{le:R open}Let $X$ be a connected topological space and $R\sub X\x X$ an equivalence relation on $X$. Assume that every equivalence class is open. Then $R=X\x X$. 
\end{lemma}
\begin{proof}[Proof of Lemma \ref{le:R open}] \setcounter{claim}{0}This follows from an elementary argument. 
\end{proof}
By a foliation chart we mean a pair $(U,\phi)$, where $U\sub M$ is an open subset and $\phi:U\to\R^n$ is a smooth chart satisfying $d\phi(x_0)T_{x_0}\F=\{0\}\x\R^k$, for every $x_0\in U$. We denote by $\pr_1:\R^n=\R^{n-k}\x\R^k\to \R^{n-k}$ and $\pr_2:\R^n\to\R^k$ the canonical projections.
\begin{lemma}\label{le:const} Let $F\sub M$ be a leaf of $\F$ and $(U,\phi)$ a foliation chart. Then the subset $\pr_1\circ\phi(U\cap F)\sub\R^{n-k}$ is at most countable.  
\end{lemma}
\begin{proof}[Proof of Lemma \ref{le:const}]\setcounter{claim}{0} Let $(M,\F)$ be a foliated manifold. By definition, the leaf topology on $F$ is the topology $\tau^\F_F$ generated by the sets $\phi^{-1}(\{0\}\x\R^k)$, where $(U,\phi)\in\F$ is such that $\phi^{-1}(\{0\}\x\R^k)\sub F$. It is second countable, see for example Lemma 1.3. on p. 11 in the book \cite{Mol}. It follows that there exists a countable collection of surjective foliation charts $\phi_i:U_i\to\R^n$ $(i\in\N)$, such that $\big(\phi_i^{-1}(\{0\}\x\R^k)\big)_{i\in\N}$ is a basis for $\tau^\F_F$. Let $(U,\phi)\in\F$. Then $U\cap F\in\tau^\F_F$, and therefore there exists a subset $S\sub\N$ such that $U\cap F=\bigcup_{i\in S}U_i$. For each $i\in S$ compatibility of $\phi$ and $\phi_i$ implies that $\phi^\xi$ is constant on $U_i$. It follows that $\phi^\xi(U\cap F)\sub\R^{n-k}$ is at most countable. The statement of Lemma \ref{le:const} follows from this.
\end{proof}
\begin{lemma} \label{le:U phi x} Let $M,\F,F,a,b,N$ and $y_0$ be as above, and $u\in C\big([0,1]\x[a,b]\x N,F\big)$ be such that (\ref{eq:u t},\ref{eq:u u' a}) hold and $u(s,i,y_0)=u(0,i,y_0)$, for every $s\in[0,1]$ and $i=0,1$. If there exists a surjective foliation chart $(U,\phi)$ such that $u(s,t,y_0)\in U$ for every $s\in[0,1]$ and $t\in[a,b]$, then $\pr^\F d(u(0,b,\cdot))(y_0)=\pr^\F d(u(1,b,\cdot))(y_0)$. 
\end{lemma} 
For the proof of Lemma \ref{le:U phi x} we need the following.
\begin{lemma}\label{le:pr 1 d phi} Let $M,\F,a,b,N$ and $y_0$ be as above, $(U,\phi)$ a foliation chart, and $u\in C\big([a,b]\x N,U\big)$ be such that (\ref{eq:u y 0},\ref{eq:u t y},\ref{eq:u t}) hold. Then $\pr_1d\phi d(u(t,\cdot))(y_0)=\pr_1d\phi d(u(a,\cdot))(y_0)$, for every $t\in[a,b]$. 
\end{lemma}
\begin{proof}[Proof of Lemma \ref{le:pr 1 d phi}]\setcounter{claim}{0} We choose an open neighborhood $V\sub N$ of $y_0$, such that $u([a,b]\x V)\sub U$ and $u'([a,b]\x V)\sub U$. Let $y\in V$. Since $u([a,b]\x\{y\})=\{\F_{u(a,y)}\}$ and $\pr_1\circ\phi\circ u([a,b]\x\{y\})\sub\R^{n-k}$ is connected, it follows from Lemma \ref{le:const} that $\pr_1\circ\phi\circ u(t,y)=\pr_1\circ\phi\circ u(a,y)$, for every $t\in[a,b]$. The statement of Lemma \ref{le:pr 1 d phi} follows from this.
\end{proof}
\begin{rmk}\label{rmk:Psi x 0} Let $(U,\phi)$ be a foliation chart and $x_0\in U$. Then there exists a unique linear isomorphism $\Psi_{x_0}\Colon \R^{n-k}\to N_{x_0}\F=T_{x_0}M/T_{x_0}\F$ satisfying $\pr^\F_{x_0}=\Psi_{x_0}\pr_1d\phi(x_0)$. To see this, observe that the map $\pr_1d\phi(x_0)\Colon T_{x_0}M\to\R^{n-k}$ is surjective and has kernel $T_{x_0}\F$. 
\end{rmk}

\begin{proof}[Proof of Lemma \ref{le:U phi x}]\setcounter{claim}{0} For $x_0\in U$ we define define $\Psi_{x_0}$ as in Remark \ref{rmk:Psi x 0}. It follows from Lemma \ref{le:pr 1 d phi} that
\[\pr^\F d(u(b,\cdot))(y_0)=\Psi_{x(b)}\pr_1d\phi d(u(a,\cdot))(y_0)=\Psi_{x(b)}\Psi_{x(a)}^{-1}\pr^\F d(u(a,\cdot))(y_0),\]
and $\pr^\F d(u'(b,\cdot))(y_0)=\Psi_{x(b)}\Psi_{x(a)}^{-1}\pr^\F d(u'(a,\cdot))(y_0)$. Using equality (\ref{eq:u u' a}), the conclusion of (\ref{prop:hol:u u'}) follows. This proves Lemma \ref{le:U phi x}.
\end{proof}
We will use the following notations and conventions. Let $a,b\in\R$. If $a\leq b$ then we equip the interval $[a,b]$ with the positive orientation. If $a>b$ then we define $[a,b]:=[b,a]$ and equip this interval with the negative orientation. We call $a$ the initial point and $b$ the end-point of $I$. Let $I$ and $I'$ be closed oriented intervals. We define the equivalence relation $\sim$ on $I\disj I'$ by $t\sim t'$ iff $t=t'$ or $t$ is the end-point of $I$ and $t'$ is the initial point of $I'$. Furthermore, we define the connected sum $I\#I'$ to be the oriented topological one-manifold $(I\disj I')/\sim$. Let now $X$ and $Y$ be sets, and $u:I\x Y\to X$ and $u':I'\x Y\to X$ maps. We define the concatenation $u\# u':(I\#I')\x Y\to X$ by $u\#u'([t],y):=u(t,y)$, if $t\in I$, and $u\#u'([t],y):=u'(t,y)$, otherwise. If $r>0$ and $t\in[a,b]$ then we denote $\bar B_r(t):=[t-r,t+r]\cap[a,b]$. 

\begin{proof}[Proof of Proposition \ref{prop:hol}]\setcounter{claim}{0} We prove {\bf statement (\ref{prop:hol:u})}. We define $\wt R$ to be the set of all $(t_1,t_2)\in[a,b]\x[a,b]$ such that for every linear map $T:T_{y_0}N\to T_{x(t_1)}M$ there exists a map $u\in C\big([t_1,t_2]\x N,M\big)$ such that $u(\cdot,y_0)=x|_{[t_1,t_2]}$, $u(t,y)\in\F_{u(t_1,y)}$, for every $t\in[t_1,t_2]$ and $y\in N$, $u(t,\cdot)$ is differentiable at $y_0$, for every $t\in[t_1,t_2]$, and $d(u(t_1,\cdot))(y_0)=T$. This is a relation on $[a,b]$. 
\begin{claim}\label{claim:wt R} The relation $\wt R$ is reflexiv and transitive. 
\end{claim}
\begin{proof}[Proof of Claim \ref{claim:wt R}] To prove reflexivity, let $t_1\in[a,b]$. We show that $(t_1,t_1)\in\wt R$. Let $T:T_{y_0}N\to T_{x(t_1)}M$ be a linear map. We choose local parametrizations $\phi:\R^n\to N$ and $\psi:\R^m\to M$ such that $\phi(0)=y_0$ and $\psi(0)=x(t_1)$, and a smooth function $\rho:\R^n\to\R$ with compact support, such that $\rho=1$ in a neighborhood of $0$. We define $u:\{t_1\}\x N\to M$ by $u(t_0,y):=\psi\big(\rho\circ\phi^{-1}(y)d\psi(0)^{-1}Td\phi(0)\phi^{-1}(y)\big)$, if $y\in\phi(\R^n)$, and $u(t_1,y):=x(t_1)$, otherwise. This map satisfies the condition in the definition of $\wt R$ with $t_2=t_1$. This proves reflexivity.

To prove transitivity, let $t_1,t_2,t_3\in[a,b]$ be such that $(t_1,t_2)\in\wt R$ and $(t_2,t_3)\in\wt R$, and $T:T_{y_0}N\to T_{x(t_1)}M$ be a linear map. We choose $u$ as in the definition of $\wt R$, and $v$ as in this definition, with $t_1,t_2$ and $T$ replaced by $t_2,t_3$ and $d(u(t_2,\cdot))(y_0)$. Then the map $u\# v$ satisfies the conditions in the definition of $\wt R$, with $t_1,t_2$ replaced by $t_1,t_3$. It follows that $(t_1,t_3)\in\wt R$. This proves transitivity and completes the proof of Claim \ref{claim:wt R}.
\end{proof}
\begin{claim}\label{claim:t 1 open} For every $t_1\in[a,b]$ the set $S_{t_1}:=\big\{t_2\in[a,b]\,\big|\,(t_1,t_2)\in\wt R\big\}$ is open.
\end{claim}
\begin{proof}[Proof of Claim \ref{claim:t 1 open}] Let $t_2\in S_{t_1}$. We choose a map $u\in C\big([t_1,t_2]\x N,M\big)$ as in the definition of $\wt R$ and a pair $(U,\phi)$, where $U\sub M$ is an open neighborhood of $x(t_2)$, and $\phi:U\to\R^m$ is a surjective foliation chart. We also choose a number $\eps>0$ so small that $x(\bar B_\eps(t_2))\sub U$. Let $t_3\in\bar B_\eps(t_2)$. We define $v:[t_2,t_3]\x N\to M$ by $v(t,y):=\phi^{-1}\big(\phi\circ u(t_2,y)+\phi\circ x(t)-\phi\circ x(t_2)\big)$. Since $x([t_2,t_3])\sub F$ and $\phi\circ x([t_2,t_3])\sub\R^{n-k}$ is connected, it follows from Lemma \ref{le:const} that $\pr_1\circ\phi\circ x([t_2,t_3])=\pr_1\circ\phi\circ x(t_2)$. Hence $v(t,y)\in\F_{u(t_2,y)}=\F_{u(t_1,y)}$, for every $t\in[t_2,t_3]$ and $y\in N$. It follows that the concatenation $u\#v$ satisfies the conditions in the definition of $\wt R$, with $t_1,t_2$ replaced by $t_2,t_3$. Therefore, $t_3\in S_{t_1}$. This proves that $S_{t_1}$ is open. This proves Claim \ref{claim:t 1 open}.
\end{proof}
We define $R:=\wt R\cap\big\{(t_1,t_2)\,\big|\,(t_2,t_1)\in\wt R\big\}$. By Claim \ref{claim:wt R} this is an equivalence relation on $[a,b]$. Let $t_1\in[a,b]$. We define $S^{t_1}:=\big\{t_2\in[a,b]\,\big|\,(t_2,t_1)\in\wt R\big\}$. Interchanging the roles of $t_1$ and $t_2$, Claim \ref{claim:t 1 open} implies that $S^{t_1}$ is open. The $R$-equivalence class of $t_1$ equals $S_{t_1}\cap S^{t_1}$ and hence is open. Therefore, by Lemma \ref{le:R open} $R=[a,b]\x[a,b]$. Statement (\ref{prop:hol:u}) follows. 

We prove {\bf assertion (\ref{prop:hol:u u'})}. Let $u$ and $u'$ be as in the hypothesis. We define $\wt R$ to be the set of all pairs $(t_1,t_2)\in[a,b]\x[a,b]$ such that 
\[\pr^\F d(u(t_1,\cdot))(y_0)=\pr^\F d(u'(t_1,\cdot))(y_0)\then\pr^\F d(u(t_2,\cdot))(y_0)=\pr^\F d(u'(t_2,\cdot))(y_0).\]
This is a reflexive and transitive relation on $[a,b]$. 
\begin{claim}\label{claim:S t 1} For every $t_1\in[a,b]$ the set $S_{t_1}:=\big\{t_2\in[a,b]\,\big|\,(t_1,t_2)\in\wt R\big\}$ is open.
\end{claim}
\begin{proof}[Proof of Claim \ref{claim:S t 1}] We choose $\eps>0$ so small that there exists a surjective foliation chart $(U,\phi)$ such that $x(\bar B_\eps(t_1))\sub U$. Let $t_2\in\bar B_\eps(t_1)$. Lemma \ref{le:U phi x} with $a,b,u$ replaced by $t_1,t_2,u|_{[t_1,t_2]}$ implies that $t_2\in S_{t_1}$. This proves Claim \ref{claim:S t 1}.
\end{proof}
We define $R:=\wt R\cap\big\{(t_1,t_2)\,\big|\,(t_2,t_1)\in\wt R\big\}$. This is an equivalence relation on $[a,b]$. Let $t_1\in[a,b]$. We define $S^{t_1}:=\big\{t_2\in[a,b]\,\big|\,(t_2,t_1)\in\wt R\big\}$. Interchanging the roles of $t_1$ and $t_2$, Claim \ref{claim:S t 1} implies that $S^{t_1}$ is open. Since the $R$-equivalence class of $t_1$ equals $S_{t_1}\cap S^{t_1}$, the hypotheses of Lemma \ref{le:R open} are satisfied. It follows that $R=[a,b]\x[a,b]$. Assertion (\ref{prop:hol:u u'}) follows. 

This completes the proof of Proposition \ref{prop:hol}. 
\end{proof}
We define $N:=N_{x(a)}\F$ and $y_0:=0$, and we canonically identify $T_0\big(N_{x(a)}\F\big)=N_{x(a)}\F$. We choose a linear map $T:N_{x(a)}\F\to T_{x(a)}M$, such that $\pr^\F T=\id_{N_{x(a)}\F}$, and a map $u\in C^\infty\big([a,b]\x N_{x(a)}\F,M\big)$ such that (\ref{eq:u y 0},\ref{eq:u t y},\ref{eq:u t},\ref{eq:d u a}) hold. We define
\[\hol^\F_x:=\pr^\F d(u(b,\cdot))(0)\Colon N_{x(a)}\F(=T_0(N_{x(a)}\F))\to N_{x(b)}\F.\]
It follows from Proposition \ref{prop:hol} that this map is well-defined. Consider now the set
\[\X^\F:=\big\{x\in C([0,1],M)\,\big|\,\exists F:\textrm{ leaf of }\F:\,x([0,1])\sub F\big\}.\]
We define the map $\hol^\F:\X^\F\to\GL(N\F)$ by $\hol^\F(x):=\hol^\F_x$. 
\begin{rmk}\label{rmk:const concat} This map is a morphism of groupoids. To see this, observe that if $x:[a,b]\to M$ is constant then $\hol^\F_x=\id_{N_{x(a)}\F}$. Furthermore, If $a\leq b$ and $a'\leq b'$ are numbers, $x\in C([a,b],F)$ and $x'\in C([a',b'],M)$ are such that $x(b)=x'(a')$, then $\hol^\F_{x\# x'}=\hol^\F_{x'}\hol^\F_x$. These assertions follow immediately from the definition of the holonomy along a path. 
\end{rmk}
Denoting by $\bar x$ the map $x$ together with the reversed orientation of $[a,b]$, it follows from Remark \ref{rmk:const concat} that $\hol^\F_{\bar x}=(\hol^\F_x)^{-1}$. 
\begin{prop}\label{prop:hol transport} The map $\hol^\F$ is continuous with respect to the compact open topology on $\X^\F$.
\end{prop}
\begin{rmk}\label{rmk:hol F continuous} If $x_0\in\X^\F$ is such that there exists a surjective foliation chart $(U,\phi)$ such that $x_0([a,b])\sub U$ then $\hol^\F$ is continuous at $x_0$. To see this, we define $\Psi_{x_0}$ as in Remark \ref{rmk:Psi x 0}. Let $x\in \X^\F$ be such that $x([a,b])\sub U$. It follows from Lemma \ref{le:pr 1 d phi} that $\hol^\F_x=\Psi_{x(b)}\Psi_{x(a)}^{-1}$. This depends continuously on $x$. 
\end{rmk}
\begin{proof}[Proof of Proposition \ref{prop:hol transport}]\setcounter{claim}{0} Let $x\in\X^\F$. We define $R$ to be the set of all pairs $(t_1,t_2)\in[a,b]\x[a,b]$ such that $\hol^\F$ is continuous at the restriction $x|_{[t_1,t_2]}\in\X^\F$. It follows from Remark \ref{rmk:hol F continuous} that $R$ is a reflexive relation. Remark \ref{rmk:const concat} implies that it is symmetric and transitive. Furthermore, Remarks \ref{rmk:const concat} and \ref{rmk:hol F continuous} imply that the $R$-equivalence classes are open. Therefore, by Lemma \ref{le:R open} we have $R=[a,b]\x[a,b]$. It follows that $\hol^\F$ is continuous at $x$. This proves Proposition \ref{prop:hol transport}.
\end{proof}
\begin{prop}\label{prop:hol homotopy} If $F$ is a leaf of $\F$ and $u\in C\big([0,1]\x[a,b],F\big)$ is such that $u(s,i)=u(0,i)$, for every $s\in[0,1]$ and $i=a,b$, then $\hol^\F_{u(0,\cdot)}=\hol^\F_{u(1,\cdot)}$. 
\end{prop}
For the proof of Proposition \ref{prop:hol homotopy} we need the following. Let $s_1,s_2,t_1,t_2\in\R$. We define $x_{s_1,s_2,t_1,t_2}$ to be the concatenation of the paths $[s_1,s_2]\ni s\mapsto (s,t_1)\in\R^2$, $[t_1,t_2]\ni t\mapsto(s_2,t)\in\R^2$, $[s_2,s_1]\ni s\mapsto(s,t_2)\in\R^2$ and $[t_2,t_1]\ni t\mapsto(s_1,t)\in\R^2$. 
\begin{rmk}\label{rmk:prop:hol homotopy} Let $(U,\phi)$ be a surjective foliation chart and $u\in C\big([s_1,s_2]\x[t_1,t_2]\to U\cap F\big)$. Then $\hol^\F_{u\circ x_{s_1,s_2,t_1,t_2}}=\id_{N_{u(s_1,t_1)}M}$. This follows from Lemma \ref{le:U phi x} and the fact that $x_{s_1,s_2,t_1,t_2}$ is homotopic in $[s_1,s_2]\x[t_1,t_2]$ to a constant path.
\end{rmk}
Let $F$ be a leaf of $\F$, $a,b,c,d\in\R$, $s_1\in[a,b]$, and $u\in C\big([a,b]\x[c,d],F\big)$. We define $R_{s_1}$ to be the set of all pairs $(t_1,t_2)\in[c,d]\x[c,d]$ such that there exists $\eps>0$ such that $\hol^\F_{u\circ x_{s_1,s_2,t_1,t_2}}=\id_{N_{u(s_1,t_1)}M}$, for every $s_2\in\big[s_1-\eps,s_1+\eps\big]\cap[a,b]$. It follows from Remark \ref{rmk:const concat} that $R_{s_1}$ is an equivalence relation on $[c,d]$. 
\begin{lemma}\label{le:hol homotopy} We have $R_{s_1}=[c,d]\x[c,d]$. 
\end{lemma}
\begin{proof}[Proof of Lemma \ref{le:hol homotopy}]\setcounter{claim}{0} By Lemma \ref{le:R open} it suffices to prove that for every $t_1\in[c,d]$ the $R_{s_1}$-equivalence class of $t_1$ is open. To see this, let $t_2\in[c,d]$ be such that $(t_1,t_2)\in R_{s_1}$. We choose $\eps_1:=\eps$ as in the definition of $R_{s_1}$, a surjective foliation chart $(U,\phi)$ such that $u(s_1,t_2)\in U$, and $\eps_2>0$ so small that $u\big(\big[s_1-\eps_2,s_1+\eps_2\big]\x\big[t_2-\eps_2,t_2+\eps_2\big]\big)\sub U$. We define $\eps:=\min\{\eps_1,\eps_2\}$. Let $s_2\in[s_1-\eps,s_1+\eps]$ and $t_3\in[t_2-\eps,t_2+\eps]$. It follows that $\hol^\F_{u\circ x_{s_1,s_2,t_1,t_2}}=\id_{N_{u(s_1,t_1)}M}$. Furthermore, by Remark \ref{rmk:prop:hol homotopy} we have $\hol^\F_{u\circ x_{s_1,s_2,t_2,t_3}}=\id_{N_{u(s_1,t_2)}M}$. Using Remark \ref{rmk:const concat}, it follows that $\hol^\F_{u\circ x_{s_1,s_2,t_1,t_2}}=\id_{N_{u(s_1,t_1)}M}$. Therefore, $(t_1,t_3)\in R_{s_1}$. This proves Lemma \ref{le:hol homotopy}.
\end{proof}
\begin{proof}[Proof of Proposition \ref{prop:hol homotopy}]\setcounter{claim}{0} We define $R:=\big\{(s_1,s_2)\in[0,1]\x[0,1]\,\big|\,\hol^\F_{u\circ x_{s_1,s_2,a,b}}=\id_{N_{u(s_1,a)}M}\big\}$. It follows from Remark \ref{rmk:const concat} that this is an equivalence relation on $[0,1]$. 
\begin{claim}\label{claim:R open} For every $s_1\in[0,1]$ the $R$-equivalence class of $s_1$ is open.
\end{claim}
\begin{proof}[Proof of Claim \ref{claim:R open}] Let $s_2\in[0,1]$ be such that $(s_1,s_2)\in R$. Thus we have $\hol^\F_{u\circ x_{s_1,s_2,a,b}}=\id_{N_{u(s_1,a)}M}$. By Lemma \ref{le:hol homotopy} we have $(a,b)\in R_{s_2}$. Hence there exists $\eps>0$ such that for $s_3\in[s_2-\eps,s_2+\eps]$ we have $\hol^\F_{u\circ x_{s_2,s_3,a,b}}=\id_{N_{u(s_2,a)}M}$. Let $s_3\in[s_2-\eps,s_2+\eps]$. Using Remark \ref{rmk:const concat}, it follows that $\hol^\F_{u\circ x_{s_1,s_3,a,b}}=\id_{N_{u(s_1,a)}M}$, i.e. $(s_1,s_3)\in R$. This proves Claim \ref{claim:R open}. \end{proof}
Claim \ref{claim:R open} and Lemma \ref{le:R open} imply that $R=[0,1]\x[0,1]$. Using Remark \ref{rmk:const concat}, the statement of Proposition \ref{prop:hol homotopy} follows.
\end{proof}
It follows from Proposition \ref{prop:hol homotopy} that the map $\hol^\F:\X^\F\to\GL(N\F)$ descends to a morphism of topological groupoids 
\begin{equation}\label{eq:hol F}\hol^\F:\X^\F/\!\!\sim_{\X^\F}\to\GL(N\F).\end{equation} 
(We use the same notation for this map.) We call this map the \emph{linear holonomy of $\F$}. 

The following result was used in the proof of Theorem \ref{thm:m M om N} in Section \ref{sec:proof:thm:m M om N}. Let $X$ be a topological manifold, $Y\sub X$, $E\to X$ a vector bundle, $\Phi:Y\x Y\to\GL(E)$ a morphism of topological groupoids whose composition with the canonical projection $\GL(E)\to X\x X$ is the identity, $k\in\N$ and $T:E^{\oplus k}\to\R$ be a (continuous) tensor. Assume that $T(\Phi_x^{x'}v_1,\ldots,\Phi_x^{x'}v_k)=T(v_1,\ldots,v_k)$, for every $x,x'\in Y$ and $v_1,\ldots,v_k\in E_x$. We define $\sim_\Phi$, $E_\Phi,\pi_\Phi$ and $T_\Phi$ as in Section \ref{sec:proof:thm:m M om N}.
\begin{lemma}\label{le:E Phi} Assume that there exists a continuous injective map $f:[0,1)\x Y\to X$ such that $f(\{0\}\x Y)=Y$. Then the pair $(E_\Phi,\pi_\Phi)$ is a vector bundle and $T_\Phi$ is continuous. 
\end{lemma}
The next result is used in the proof of Lemmas \ref{le:E Phi} and \ref{le:Phi Psi}. 
\begin{lemma}\label{le:E Psi} Let $X$ be a paracompact topological space and $E\to [0,1)\x X$ a vector bundle. We denote by $\pr:[0,1)\x X\to X$ the canonical projection. Then there exists $\Psi\in\Iso(\pr^*E,E)$ such that $\Psi_{(0,x)}=\id_{E_{(0,x)}}$, for every $x\in X$. Furthermore, if $\om$ is a fiberwise symplectic structure on $E$ then there exists $\Psi$ as above that preserves $\om$. 
\end{lemma}
\begin{proof}[Proof of Lemma \ref{le:E Psi}]\setcounter{claim}{0} To prove the first assertion, note that there exists $\wt\Psi\in\Iso(\pr^*E,E)$, see for example \cite{Hu}, Part I Chap. 3, 4.4 Corollary (p.~28). We define $\Psi^{(t,x)}:=\wt\Psi^{(t,x)}(\wt\Psi^{(0,x)})^{-1}$. The second assertion follows from a version of that corollary for symplectic vector bundles. This version is proved by the argument in \cite{Hu}, by choosing the local trivializations to be symplectic.
\end{proof}

\begin{proof}[Proof of Lemma \ref{le:E Phi}]\setcounter{claim}{0} To show that $(E_\Phi,\pi_\Phi)$ is a vector bundle, let $x_0\in X$. Assume that $x\not\in Y$. We denote by $n$ the rank of $E$ and choose a pair $(U,\Psi)$, where $U\sub X$ is an open neighborhood of $x_0$ and $\Psi:U\x\R^n\to E$ is a local trivialization. Then viewing $U\wo Y$ as a subset of $X/Y$ the map $(U\wo Y)\x\R^n\to E_\Phi$, $(x,v)\mapsto[x,v]$, is a local trivialization for $E_\Phi$ around the point $x_0$. Assume now that $x_0\in Y$. We denote $n:=\rank E$ and $U:=f([0,1)\x Y)\sub X$.
\begin{claim}\label{claim:Psi} There exists $\Psi\in\Iso\big(U\x\R^n,E|_U\big)$ such that $\Psi^{(0,x')}=\Phi_x^{x'}\Psi^{(0,x)}$, for $x,x'\in Y$.
\end{claim}
\begin{proof}[Proof of Claim \ref{claim:Psi}] We denote by $\pr:[0,1)\x Y\to Y$ the canonical projection, and choose $f$ as in the hypothesis. We denote $E':=f^*E\to[0,1)\x Y$. By an elementary argument there exists $\wt\Psi\in\Iso(\pr^*E',E')$ such that $\wt\Psi_{(0,x)}=\id_{E'_x}$, for every $x\in Y$. We denote $n:=\rank E$ and choose a point $x_0\in Y$ and $\Psi_0\in\Iso(\R^n,E_{x_0})$. We define $\Psi:U\x\R^n\to E$ by $\Psi^x:=\wt\Psi_{f^{-1}(x)}\Phi_{x_0}^{\pr\circ f^{-1}(x)}\Psi_0$, for $x\in U$. This map has the required properties. This proves Claim \ref{claim:Psi}.
\end{proof}
We choose a map $\Psi$ as in Claim \ref{claim:Psi}, and define $\Psi':U/Y\x\R^n\to E_\Phi$, $\Psi'([x],v):=[x,v]$. This is a local trivialization for $E_\Phi$ around $x_0$. Furthermore, the map $U/Y\x(\R^n)^k\to \R$, $([x],v_1,\ldots,v_k)\mapsto T_\Phi(\Psi'_{[x]}v_1,\ldots,\Psi'_{[x]}v_k)=T_x(\Psi_xv_1,\ldots,\Psi_xv_k)$ is continuous. It follows that $(E_\Phi,\pi_\Phi)$ is a vector bundle and $T_\Phi$ is continuous. This proves Lemma \ref{le:E Phi}.
\end{proof}

The following result was used in the proof of Theorem \ref{thm:Fix}. Let $X$ be a topological space, $Y\sub X$ a subset, and $\sim$ an equivalence relation on $Y$. We denote by $\iota:Y\to X$ the inclusion and by $\pi:Y\to Y/\!\!\sim$ the canonical projection. We fix a positive integer $k$, and we denote by $\bar B^k\sub\R^k$ and $S^{k-1}\sub \R^k$ the closed unit ball and the unit sphere. We define the map
\begin{eqnarray}\label{eq:hat phi}&\hat\phi:\big\{u\in C(\bar B^k,X)\,\big|\,u\textrm{ is }(S^{k-1},Y/\!\!\sim)\textrm{-compatible}\big\}\to C\big(\bar B^k,X\x Y/\!\!\sim\big),&\\
\label{eq:hat phi u}&\hat\phi(u):=\big(u,\pi\circ u(z_0)\big),&
\end{eqnarray}
where $z_0\in S^{k-1}$ is an arbitrary point and we view $\pi\circ u(z_0)$ as a constant map from $\bar B^k$ to $Y/\!\!\sim$. Note that $(S^{k-1},Y/\!\!\sim)$-compatibility of $u$ implies that the right hand side of (\ref{eq:hat phi u}) does not depend on the choice of $z_0$. Furthermore, the map $\hat\phi(u)$ is $\big(S^{k-1},\{\im(\iota,\pi)\}\big)$-compatible, and the $\big(S^{k-1},\{\im(\iota,\pi)\}\big)$-compatible homotopy class of this map is invariant under $(S^{k-1},Y/\!\!\sim)$-compatible homotopies. Hence $\hat\phi$ descends to a map
\begin{equation}\label{eq:phi}\phi:\big[\bar B^k,S^{k-1};X,Y/\!\!\sim\big]\to \big[\bar B^k,S^{k-1};X\x Y/\!\!\sim,\{\im(\iota,\pi)\}\big],
\end{equation}
defining $\phi([u]):=[\hat\phi(u)]$. Recall that a Serre fibration is a continuous map with the homotopy lifting property for all CW-complexes.
\begin{prop}\label{prop:phi} Let $X,Y,\sim,\iota,\pi$ and $k$ be as above. Assume that the map $\pi:Y\to Y/\!\!\sim$ is a Serre fibration. Then the map $\phi$ above is a bijection.
\end{prop}
\begin{rmk}\label{rmk:pi X X'} Let $\pi:X\to X'$ be a Serre fibration, $Y$ be a CW-complex, and $u_0,u_1,v:[0,1]\x Y\to X$ be continuous maps. Assume that 
\begin{equation}\label{eq:v i u i}v(i,\cdot)=u_i(0,\cdot),\,i=0,1,\end{equation} 
and there exists a continuous map $u':[0,1]\x[0,1]\x Y\to X'$ such that 
\begin{equation}\label{eq:pi u i}\pi\circ u_i=u'(i,\cdot,\cdot),\,i=0,1,\quad \pi\circ v=u'(\cdot,0,\cdot).\end{equation}
Then there exists a continuous map $u:[0,1]\x[0,1]\x Y\to X$ such that $\pi\circ u=u'$ and 
\[u(i,\cdot,\cdot)=u_i,\,i=0,1,\quad u(\cdot,0,\cdot)=v.\]
This follows from the homotopy lifting property for $\big(\pi,[0,1]\x Y\big)$, applied to the map $u'\circ(\phi\x\id_Y):[0,1]\x[0,1]\x Y\to X'$, where $\phi:[0,1]\x[0,1]\to [0,1]\x[0,1]$ is a homeomorphism that maps $[0,1]\x\{0\}$ to $\{0,1\}\x[0,1]\cup[0,1]\x\{0\}$. 
\end{rmk}

\begin{proof}[Proof of Proposition \ref{prop:phi}]\setcounter{claim}{0} Let $X,Y,\sim,\iota,\pi$ and $k$ be as in the hypothesis. We define the map 
\[\psi:\big[\bar B^k,S^{k-1};X\x Y/\!\!\sim,\{\im(\iota,\pi)\}\big]\to \big[\bar B^k,S^{k-1};X,Y/\!\!\sim\big]\]
as follows. Namely, let $\wt a\in\big[\bar B^k,S^{k-1};X\x Y/\!\!\sim,\{\im(\iota,\pi)\}\big]$. We fix a representative $(u,v'):\bar B^k\to X\x Y/\!\!\sim$ of $\wt a$.
\begin{claim}\label{claim:f B k} There exists a continuous map $f:\bar B^k\to X$ such that 
\begin{eqnarray}\label{eq:f z u}&f(z)=u(2z),\textrm{ if }|z|\leq\frac12,&\\
\label{eq:f z Y} &f(z)\in Y,\quad\pi\circ f(z)=v'\big((2/|z|-2)z\big),\quad\textrm{if }\frac12<|z|\leq1.&
\end{eqnarray}
\end{claim}
\begin{proof}[Proof of Claim \ref{claim:f B k}] We define $w':[0,1]\x S^{k-1}\to Y/\!\!\sim$ by $w'(r,z):=v'((1-r)z)$. By $\big(S^{k-1},\{\im(\iota,\pi)\}\big)$-compatibility of $(u,v')$ we have $w'(0,z)=\pi\circ u(z)$, for every $z\in S^{k-1}$. Therefore, by the homotopy lifting property of $\pi$ there exists a continuous map $w:[0,1]\x S^{k-1}\to Y$ such that $\pi\circ w=w'$ and $w(0,z)=u(z)$, for every $z\in S^{k-1}$. We define 
\[f:\bar B^k\to X,\quad f(z):=\left\{\begin{array}{ll}
u(2z),&\textrm{if }|z|\leq\frac 12,\\
w\big(2|z|-1,z/|z|\big),&\textrm{if }\frac12< |z|\leq1.
\end{array}\right.\]
This proves Claim \ref{claim:f B k}.
\end{proof}
Note that if $f$ is as in Claim \ref{claim:f B k} then $\pi\circ f(z)=v'(0)$, for $z\in S^{k-1}$, and therefore the map $f$ is $(S^{k-1},Y/\!\!\sim)$-compatible.
\begin{claim}\label{claim:f u v'} If $(u_0,v_0')$ and $(u_1,v_1')$ are two representatives of $\wt a$ and $f_0,f_1:\bar B^k\to X$ are continuous maps satisfying (\ref{eq:f z u},\ref{eq:f z Y}) with $u,v',f$ replaced by $u_i,v'_i,f_i$, for $i=0,1$, then the maps $f_0$ and $f_1$ are homotopic compatibly with $(S^{k-1},Y/\!\!\sim)$. 
\end{claim}
\begin{proof}[Proof of Claim \ref{claim:f u v'}] Let $(u,v'):[0,1]\x\bar B^k\to X\x Y/\!\!\sim$ be continuous maps such that 
\begin{equation}\label{eq:}(u,v')(i,\cdot)=(u_i,v'_i),\quad (u,v')(\{s\}\x S^{k-1})\sub \im(\iota,\pi),\,\forall s\in[0,1].\end{equation} 
We define 
\begin{equation}\label{eq:w i}w_i:[0,1]\x S^{k-1}\to Y,\quad w_i(t,z):=f_i\big(\frac{t+1}2z\big),\end{equation}
 for $i=0,1$, and $v:=u|_{[0,1]\x S^{k-1}}$. Then the conditions of Remark \ref{rmk:pi X X'} with $Y:=S^{k-1}$ and $X,X',u_i$ replaced by $Y,Y',w_i$ are satisfied. To see this, note that (\ref{eq:v i u i}) follows from (\ref{eq:f z u}). We define 
\begin{equation}\label{eq:w'}w':[0,1]\x[0,1]\x S^{k-1}\to Y/\!\!\sim,\quad w'(s,t,\cdot):=v'\big(s,(1-t)z\big).\end{equation} 
Then for every $t\in[0,1]$, $z\in S^{k-1}$, we have
\begin{eqnarray}\nn \pi\circ w_i(t,z)&=&\pi\circ f_i\big(\frac{t+1}2z\big)\\ 
\nn&=&v_i'((1-t)z)\\
\nn&=&v'\big(i,(1-t)z\big)\\
\nn&=&w'(i,t,z).
\end{eqnarray}
Here in the first step we used (\ref{eq:f z u}) with $f,u$ replaced by $f_i,u_i$ and $z\in S^{k-1}_{1/2}$, and in the second step we used (\ref{eq:f z Y}) with $f,v'$ replaced by $f_i,v_i'$. So the second hypothesis of Remark \ref{rmk:pi X X'} is also satisfied, with $Y:=S^{k-1},u':=w'$ and $X,X',u_i$ replaced by $Y,Y',w_i$. It follows that there exists a continuous map $w:[0,1]\x[0,1]\x S^{k-1}\to Y$ such that 
\begin{equation}\label{eq:w i w i}\pi\circ w=w',\quad w(i,\cdot,\cdot)=w_i,\,i=0,1,\quad w(\cdot,0,\cdot)=v=u|_{[0,1]\x S^{k-1}}.\end{equation}
We define $f:[0,1]\x\bar B^k\to X$ by 
\[f(s,z):=\left\{\begin{array}{ll}u(s,2z),&\textrm{if }|z|\leq\frac12,\\
w\big(s,2|z|-1,z/|z|\big),&\textrm{if }\frac12<|z|\leq1.
\end{array}\right.\]
By the third equality in (\ref{eq:w i w i}) the map $f$ is continuous. Furthermore, (\ref{eq:f z u}) with $f,u$ replaced by $f_i,u_i$, (\ref{eq:w i}) and the second equality in (\ref{eq:w i w i}) imply that $f(i,z)=f_i(z)$, for $i=0,1$ and every $z\in\bar B^k$. Finally, let $s\in[0,1]$. Then by (\ref{eq:w'}) and the first equality in  (\ref{eq:w i w i}) we have, $\pi\circ f(s,z)=v'(s,0)$, for $z\in S^{k-1}$. Hence $f(s,\cdot)$ is $\big(S^{k-1},Y/\!\!\sim\big)$-compatible, and therefore $f$ is a $\big(S^{k-1},Y/\!\!\sim\big)$-compatible homotopy from $f_0$ to $f_1$. This proves Claim \ref{claim:f u v'}.
\end{proof}
We choose a map $f:\bar B^k\to X$ as in Claim \ref{claim:f B k} and define $\psi(\wt a)$ to be the $(S^{k-1},Y/\!\!\sim)$-compatible homotopy class of $f$. By Claim \ref{claim:f u v'} this definition does not depend on the choice of $f$. 
\begin{claim}\label{claim:phi psi} The maps $\phi$ and $\psi$ are inverses of each other.
\end{claim}
\begin{pf}[Proof of Claim \ref{claim:phi psi}] To see that $\psi\circ\phi=\id$ let $a\in\big[\bar B^k,S^{k-1};X,Y/\!\!\sim\big]$. We choose a representative $u$ of $a$, and define $f:\bar B^k\to X$ by
\[f(z):=\left\{\begin{array}{ll}u(2z),&\textrm{if }|z|\leq\frac12,\\
u(z/|z|),&\textrm{otherwise.}\end{array}\right.\] 
Note that $f$ is continuous and $\big(S^{k-1},Y/\!\!\sim\big)$-compatible. Note that $\big(S^{k-1},Y/\!\!\sim\big)$-compatible homotopy classes of $u$ and $f$ agree. The identity $\psi\circ\phi=\id$ follows now from the next claim. 
\begin{claim}\label{claim:f a} The $\big(S^{k-1},Y/\!\!\sim\big)$-compatible homotopy class of $f$ equals $\psi\circ\phi(a)$. 
\end{claim}
\begin{proof}[Proof of Claim \ref{claim:f a}] By definition $\phi(a)=\big[u,\pi\circ u(z_0)\big]$, where $z_0\in S^{k-1}$ is an arbitrary point. Furthermore, equalities (\ref{eq:f z u},\ref{eq:f z Y}) are satisfied, with $v':=\pi\circ u(z_0)$. Hence $f$ represents $\psi\circ\phi(a)$. This proves Claim \ref{claim:f a}.
\end{proof}
To see that $\phi\circ\psi=\id$ let $\wt a\in\big[\bar B^k,S^{k-1};X\x Y/\!\!\sim,\{\im(\iota,\pi)\}\big]$. We choose a representative $(u,v')$ of $\wt a$ and a continuous map $f:\bar B^k\to X$ such that the conditions (\ref{eq:f z u},\ref{eq:f z Y}) hold. We fix $z_0\in S^{k-1}$. Then by definition the maps $\big(f,\pi\circ f(z_0)\big)$ and $\phi\circ\psi(\wt a)$ are  homotopic compatibly with $\big(S^{k-1},\{\im(\iota,\pi)\}\big)$. The identity $\phi\circ\psi=\id$ follows now from the next claim.
\begin{claim}\label{claim:f wt a} The $\big(S^{k-1},\{\im(\iota,\pi)\}\big)$-compatible homotopy class of $\big(f,\pi\circ f(z_0)\big)$ equals $\wt a$. 
\end{claim}
\begin{pf}[Proof of Claim \ref{claim:f wt a}] We define the map $h:[0,1]\x\bar B^k\to X\x Y/\!\!\sim$ by
\[h(s,z):=\Big(f\Big(\big(1-\frac s2\big)z\Big),v'(sz)\Big).\]
Then 
\[h(0,\cdot)=\big(f,\pi\circ f(z_0)\big),\quad h(1,\cdot)=(u,v'),\]
\[ h(s,z)\in\im\{\iota,\pi\},\,\forall s\in[0,1],\,z\in S^{k-1}.\]
Here in the first equality we used (\ref{eq:f z Y}), in the second equality we used (\ref{eq:f z u}), and in the condition we used (\ref{eq:f z Y}) again. Hence $h$ is a $\big(S^{k-1},\{\im(\iota,\pi)\}\big)$-compatible homotopy from $\big(f,\pi\circ f(z_0)\big)$ to $(u,v')$. This proves Claim \ref{claim:f wt a} and hence Claim \ref{claim:phi psi}, and concludes the proof of Proposition \ref{prop:phi}.
\end{pf}
\end{pf}
\end{proof}

\subsection*{Open compact topology}\label{subsec:open compact}

The following lemma was used in the proofs of Lemmas \ref{le:Psi * Phi}, \ref{le:al}. For two topological spaces $X$ and $Y$ we equip the set of continuous maps $C(X,Y)$ with the compact open topology.
\begin{lemma}\label{le:compact open} Let $X,Y$ and $Z$ be topological spaces. Then the following statements hold. 
\begin{enui}
\item\label{le:compact open:comp} If $Y$ is locally compact and Hausdorff then the composition map $C(X,Y)\x C(Y,Z)\ni(f,g)\mapsto g\circ f\in C(X,Z)$ is continuous.
\item\label{le:compact open:ev} If $X$ is locally compact and Hausdorff then the evaluation map $C(X,Y)\x X\ni(f,x)\mapsto f(x)\in Y$ is continuous. 
\item\label{le:compact open:exp}If $X$ is Hausdorff and $Y$ is locally compact and Hausdorff then the map $\phi:C(X,C(Y,Z))\to C(X\x Y,Z)$ defined by $\phi(f)(x,y):=f(x)(y)$, for $(x,y)\in X\x Y$, is well-defined and a homeomorphism. 
\item\label{le:compact open:d}If $X$ is compact Hausdorff and $Y$ is metrized by a metric $d$, then $C(X,Y)$ is metrized by the metric $d'$ defined by 
\begin{equation}\label{eq:d'}d'(f,g):=\sup\big\{d(f(x),g(x))\,\big|\,x\in X\big\}.\end{equation}
\end{enui}
\end{lemma}
\begin{proof}[Proof of Lemma \ref{le:compact open}]\setcounter{claim}{0} These are standard results, see for example the book \cite{Hat}.
\end{proof}
\begin{lemma}\label{le:m om cont} For every symplectic vector space $(V,\om)$ the map $m_\om:C([0,1],\Aut\om)\to\R$ is continuous with respect to the compact open topology.
\end{lemma}
\begin{proof}[Proof of Lemma \ref{le:m om cont}]\setcounter{claim}{0} Since $\rho_\om$ is continuous, by Lemma \ref{le:compact open}(\ref{le:compact open:comp}) the map $C([0,1],\Aut\om)\ni\Phi\mapsto \rho_\om\circ\Phi\in C([0,1],S^1)$ is continuous. By Lemma \ref{le:al} the winding map $\al:C([0,1],S^1)\to \R$ is continuous. Since the map $m_\om$ is the composition of these two maps, the statement of Lemma \ref{le:m om cont} follows.
\end{proof}

\end{document}